\documentclass[11pt]{article}

\usepackage{enumitem}

\usepackage[utf8]{inputenc}
\usepackage[T1]{fontenc}
\usepackage{xcolor,graphicx}
\usepackage{eepic}
\usepackage{fouriernc}
\usepackage[english]{babel}
\usepackage{amsfonts, amssymb, amsmath, amsthm}

\usepackage{mathrsfs}
\usepackage{cancel}

\usepackage{amsmath}
\numberwithin{equation}{section}

\usepackage{graphicx}
\usepackage{color}
\usepackage{graphicx}
\usepackage[colorinlistoftodos]{todonotes}
\usepackage[colorlinks=true, allcolors=blue]{hyperref}

\usepackage{amsmath}
\usepackage{amssymb}
\usetikzlibrary{arrows.meta, angles, quotes, calc}

\usepackage{array} 
\usepackage{amsfonts}
\usepackage{amsmath}
\usepackage{amssymb}
\usepackage{graphicx}
\usepackage{float} 
\usepackage{color}
\usepackage{esint}    

\usepackage{enumitem} 

\usepackage{multirow}

\newtheorem{thm}{Theorem}[section]
\newtheorem{cor}[thm]{Corollary}
\newtheorem{lem}[thm]{Lemma}

\newtheorem{defn}[thm]{Definition}
\newtheorem{rem}[thm]{Remark}
\newtheorem{nota}[thm]{Notation}
\numberwithin{equation}{section}

\newcommand{\dy}{\,{\rm d}y}

\def\LL{\mathrm{L}} 

\newcommand{\ua}{\overline{u}}

\newcommand{\ub}{\underline{u}}

\newcommand{\vb}{\underline{v}}

\newcommand{\RR}{\mathbb{R}}

\newcommand{\NN}{\mathbb{N}}

\def\ee{\mathrm{e}} 
\DeclareMathOperator*{\osc}{osc}
\def\qed{\,\unskip\kern 6pt \penalty 500
	\raise -2pt\hbox{\vrule \vbox to8pt{\hrule width 6pt
			\vfill\hrule}\vrule}\par}

\def\quotient#1#2{\raise1ex\hbox{$#1$}\Big/\lower1ex\hbox{$#2$}}

\definecolor{darkblue}{rgb}{0.05, .05, .65}
\definecolor{darkgreen}{rgb}{0.1, .65, .1}
\definecolor{darkred}{rgb}{0.8,0,0}

\usepackage{mathtools}
\mathtoolsset{showonlyrefs}

\newcommand{\ue}{u_\epsilon}

\newcommand{\plap}{\Delta_p}
\newcommand{\plapN}{\Delta_p^{N}}
\newcommand{\qlap}{\Delta_q}

\newcommand{\gr}{\nabla}
\newcommand{\ph}{\varphi}

\newcommand{\intn}{\fint}

\DeclareMathOperator{\sgn}{sgn}

\newcommand{\A}{\mathcal{A}}
\newcommand{\B}{\mathcal{B}}
\newcommand{\C}{\mathcal{C}}
\newcommand{\wb}{\underline{w}}
\newcommand{\wa}{\overline{w}}

\newcommand{\Rd}{\mathbb{R}^d}

\newcommand{\beps}{\overline{\epsilon}}

\begin{document}
	\title{\textbf{A mean value formula for the $p-$Laplacian}}
	
	\author{\Large Carlos Fuertes-Moran
		~and~ Julio D. Rossi
	}
	\date{} 
	
	\maketitle
	
	
	\begin{abstract}
		In this paper we introduce an asymptotic mean value formula of the form
		$u(x) - \B_\epsilon^p (u) (x) =\epsilon^2 \Delta_p u(x) + o(\epsilon^2)$ for the $p-$Laplacian, $\plap u := \nabla \cdot \big( |\nabla u |^{p-2} \nabla u \big)$. Here $\B_\epsilon^p (u)$ is a nonlinear expression that 
		involve the mean value and the surmemum/infimum of $u$ in
		suitable balls. 
		This asymptotic mean value formula
		allows us to characterize viscosity solutions to equations with different homogeneities in different sets of variables, like  
		$\plap^{x'} u(x)+\qlap^{x''}u(x)=f(x)$ (for $x=(x',x'') \in\mathbb{R}^d$, $x'=(x_1,...,x_n)$ and $x''=(x_{n+1,...,x_d})$), 
		and the parabolic $p$-Laplace heat equation
		$\partial_t u (x,t)-\Delta_p u(x,t)=0$.
		Moreover, the resulting asymptotic mean value formulas are related to two-player zero sum games
		whose value functions, characterized as the unique solution to the corresponding dynamic programming principle, converge as
		the parameter that controls the size of the steps in the game goes to zero, 
		to the unique viscosity solution to the corresponding boundary value problem for the operator under
		consideration. Finally, we also include a Holder regularity result for $\plap^{x'} u(x)+\qlap^{x''}u(x)=0$ proved using the game theoretical 
		interpretation of the problem. 
	\end{abstract}
	
	
	\small
	
	\noindent Carlos Fuertes Moran. 
	
	\noindent Departamento de Matem\'{a}ticas, Universidad Aut\'{o}noma de Madrid,\\
	ICMAT - Instituto de Ciencias Matem\'{a}ticas, CSIC-UAM-UC3M-UCM, \\
	Campus de Cantoblanco, 28049 Madrid, Spain.
	
	\noindent E-mail:\texttt{~carlos.fuertesm@uam.es}
	
	\bigskip
	
	\noindent Julio D. Rossi.
	
	\noindent Departamento de Matemáticas y Estadística, 
	Universidad Torcuato di Tella, \\
	Campus Alcorta - Av. Figueroa Alcorta 7350, Buenos Aires, Argentina.
	
	\noindent E-mail:\texttt{~julio.rossi@utdt.edu}	
	
	\bigskip

	
	
	
	\smallskip\noindent {\sl\small\copyright~2026 by the authors. This paper may be reproduced, in its entirety, for non-commercial purposes.}

	\newpage
	
	
	\normalsize
	\tableofcontents

	\newpage
	
	\section{Introduction}\label{sec1}
	
	There is a deep connection between Partial Differential Equations and Probability theory. For linear operators, like the Laplacian,
	this relation turns out to rely on the validity of mean value formulas
	for the solutions in the PDE side and martingale identities in the
	probability side. 
	Moreover, approximation of solutions to 
	linear and nonlinear PDEs using game theory is by now a classical subject, see \cite{libroBR} and \cite{Lewicka_book}.
	
	For the Laplacian, 
	it is a well-known fact that a mean value formula characterizes harmonic functions;  a function $u$ is a solution to $\Delta u=0$ in a domain
	$\Omega\subset\mathbb{R}^d$ 
	if and only if $u$ satisfies the mean value property
	\begin{equation}\label{MVPrperty_HarmonicFunctions}
		u(x)
		=
		\intn_{B_\epsilon(x)}u(y)\, dy
	\end{equation}
	for each $x\in\Omega$ and all $0 < \epsilon< \textrm{dist}(x,\partial\Omega)$ (here we denote by 
	$\intn_{A}u(y)\, dy = \frac{1}{|A|} \int_{A} u (y) \,d y$ the mean value of the function $u$ in the set $A$). 
	In fact, a weaker statement,  
	known as an asymptotic mean value property, holds: A continuous function $u$ is harmonic in $\Omega$ if and only if 
	\begin{equation}\label{Asymptotic.MVP.Harmonic}
		u(x)= \intn_{B_\epsilon(x)} u(y)\,dy+o(\epsilon^2)\quad\textrm{as}\ \epsilon\to0,
	\end{equation}
	see \cite{Blaschke,Privaloff}.
	Moreover, the mean value property can be extended to characterize sub- and superharmonic functions replacing the equality by the appropriate inequality
	in \eqref{MVPrperty_HarmonicFunctions} and \eqref{Asymptotic.MVP.Harmonic}.
	Extensions of these ideas to the classical Poisson equation $-\Delta u=f$  can be obtained from the formula
	\begin{equation}\label{Laplaciano.formula.asintotica}
		\Delta u(x)=\lim_{\epsilon \to0}\left[\frac{2(n+2)}{\epsilon^2}\left(\intn_{B_\epsilon(x)} u(y)\,dy-u(x)\right)\right],
	\end{equation}
	from where we conclude that when $f$ is continuous, a function $u$ satisfies $-\Delta u=f$ in $\Omega$ if and only if
	\begin{equation}
		\label{Asymptotic.MVP.Harmonic.f}
		u(x)= \intn_{B_\epsilon (x)} u(y)\,dy+\frac{\epsilon^2}{2(n+2)}\,f(x)+o(\epsilon^2)\qquad\textrm{as} \ \epsilon \to0
	\end{equation}
	for each $x \in \Omega$.

	Mean value formulas for operators other than the Laplacian are known. In \cite{Littman.et.al1963}, the authors prove a mean value theorem for linear divergence form operators with bounded measurable coefficients.
	In this case, the statement of the mean value formula involves the Green function for the operator.
	For extensions to very different equations (like Pucci operators, the Monge-Ampere equation, the obstacle problem, etc), we refer to the recent books 
	\cite{libroBR} and \cite{Lewicka_book} and references therein.

	\indent The present work is motivated by the study of game-theoretical asymptotic expansions, asymptotic mean value properties, and dynamic programming principles for various partial differential equations involving the $p$-Laplacian operator, with $p>2$. The $p$-Laplacian of a real valued function is given by
	$$	\plap u = \nabla \cdot \big( |\nabla u |^{p-2} \nabla u \big).$$
	
	The story begins with seminal papers \cite{PSSW} and \cite{PSSW2}. In those papers the aithors introduce and deeply study 
	a game (called Tug-of-War or Tug-of-War with noise in the literature) whose value function approximates $p-$harmonic functions, that is, 
	functions that satisfy $\plap u = 0$. Inspired by these results in \cite{Manfredi_Parviainen_Rossi_2010} an asymptotic mean value formula that characterizes viscosity 
	$p-$harmonic functions was found. In fact, it holds that a function $u$ is  a viscosity
	solution to $\Delta_p u = \mbox{div} ( |\nabla u|^{p-2}
	\nabla u)=0$ in a domain $\Omega \subset \mathbb{R}^d$ if and only if the expansion
	\begin{equation} \label{as-1}
		u(x) = \frac{\beta_p}{2} \left\{ \sup_{{B_\epsilon
				(x)}} u +
		\inf_{{B_\epsilon (x)}} u
		\right\} + (1-\beta_p) \intn_{B_\epsilon (x)} u (y) \,d y
		+ o (\epsilon^2)
	\end{equation}
	holds as $\epsilon \to 0$ for $x\in \Omega$.
	Here the coefficient $\beta_p$ depends on $p$ as $\beta_p = (p-2) / (d+p)$.
	The key for thei result is the following identity
	\begin{align*}
		\Delta_p u = |\nabla u|^{p-2} \plap^N u, \quad \text{with} \quad \plap^N u \coloneqq (p-2) 
		\left\langle D^2 u \frac{\nabla u}{|\nabla u |}, \frac{\nabla u}{|\nabla u |}\right\rangle + \Delta u, 
	\end{align*}
	where $\plap^N$ is the so-called \emph{normalized} $p$-Laplacian. 
	This is useful due to the fact that $\plap u = 0$ if and only if $\plap^N u = 0$,
	and the fact that the asymptotic mean value formula \eqref{as-1} is related to $\plap^N u$ as simple Taylor expanssions
	show.
	Notice that the operator $\plap^N$ (and also \eqref{as-1}) is $1-$homogeneous in the sense that
	$\plap^N (\lambda u ) = \lambda \plap^N u$.
	
	When one adds a nontrivial right hand side to the equation it is not true that $\plap u = f$ if and only if $\plap^N u = f$,
	and hence new ideas are needed. 
	Recently, in~\cite{delTeso_Rossi_2024}, a mean value formula for 
	$$
	\plap u(x) =f(x)
	$$ in a bounded domain $\Omega\subset \mathbb{R}^d$,
	was found. 
	In the previously mentioned reference, the authors derived a mean value operator that approximates $(\Delta_p u)^{\frac{1}{p-1}}$ at points where the gradient of the function does not vanish, and whose form depends on the sign of the operator (that coincides with the sign of $f$).
	
	Subsequently, in~\cite{delTeso+Fuertes+Rossi_2026}, the authors succeeded in approximating the signed power of the $p$-Laplacian, namely, $(\plap u)^{\frac{1}{p-1}}=\sgn (\plap u) \left|\plap u\right|^{\frac{1}{p-1}}$ (here and in what follows we will denote by $(a)^r = \sgn (a) \left| a \right|^{r}$ the signed power of a real number). 
	Moreover, the new formula for $(\plap u)^{\frac{1}{p-1}}$, has the form
	$$
	u (x) = \A_\epsilon^p[u](x) - \epsilon^2\left(\plap u (x)\right)^{\frac{1}{p-1}}+ o(\epsilon^2).
	$$
	(we shall present $\A_\epsilon^p$ in Section \ref{sec3}), 
	and is also valid at points where the gradient vanishes.
	
	The resulting formula is more robust in the sense that it provides a single algebraic expression independently of the sign of the $p$-Laplacian. This made it possible to develop a complete game-theoretical framework for parabolic equations with the same homogeneity, 
	$$
	\left( \partial_t u(x,t) \right)^{p-1} = \plap u(x,t).
	$$ 
	 Notice that $(\plap u)^{\frac{1}{p-1}}$ and also $\A_\epsilon^p$ are again $1-$homogeneous in
	the sense that $(\plap u)^{\frac{1}{p-1}} (\lambda u)= \lambda (\plap u)^{\frac{1}{p-1}} u$ and $\A_\epsilon^p [\lambda u] = \lambda \A_\epsilon^p [u]$.
	
	The fact that all the previous asymptotic mean value formulas are related to $1-$homogeneous operators prevents
	us to deal with equations with different homogeneities, the primary example being the following anisotropic elliptic equation,
	\begin{equation}\label{eq_p-q}
		\plap^{x'} u(x)+\qlap^{x''}u(x)=f(x).
	\end{equation}
	Here, we denote $x\in\mathbb{R}^d$ by $x=(x',x'')$, where  $x'=(x_1,...,x_n)$ and $x''=(x_{n+1,...,x_d})$. 
	
	Operators of this type, and even more general ones, are known in the literature as anisotropic $p$-Laplacians. For existence and uniqueness results, see \cite{Antontsev+Chipot, Nardo+Feo-2014, Nardo+Feo+Guibe-2013}. Moreover, \cite{Vetois-max-prin} establishes a strong maximum principle, while \cite{Ochoa+Valvelde_2024} proves, under suitable regularity assumptions, an equivalence between viscosity and weak solutions. Some regularity issues are studied in \cite{Demengel-regularity}, or, in the more general setting where the exponent depends on the spatial variable, $p=p(x)$, in \cite{Coscia+Mingione-holder, deFilipis+Mingione-schauder, deFilipis+Mingione-lipschitz}. Although throughout this work we will always assume $2<q<p<\infty$, the limiting cases in which the exponents tend to $1$ and $\infty$, respectively, are considered in \cite{Mercaldo+Rossi+deLeon+Trombetti} and \cite{Llanos+Rossi}.
	
	Our main goal in the present paper is to find a new strategy to handle this problem
	and to deal with non-homogeneous operators. 
	Relying on the previous results established in~\cite{delTeso+Fuertes+Rossi_2026} for signed powers, we find here a new mean value operator	that approximates the $p$-Laplacian without involving any power.
	We find $\B_\epsilon^p$ such that
	$$
	u (x) = \B_\epsilon^p[u](x) - \epsilon^2 \plap u (x) + o(\epsilon^2).
	$$ 
	Now, the mean value operator $\B_\epsilon^p[u]$ is not $1-$homogeneous. 
	
	Once this new mean value formula is stablished, we are able to study equations like \eqref{eq_p-q} both in the homogeneous case, $f(x)=0$, and in the non-homogeneous case, even when more than two $p$-Laplacian operators are involved.
	
	We also are able to handle the parabolic $p$-Laplace heat equation
	\begin{equation}\label{eq-parabolic}
		\partial_t u (x,t)-\Delta_p u(x,t)=0.
	\end{equation}
	Indeed, once a mean value operator approximating the $p$-Laplacian has been obtained, approximating the time derivative becomes standard. 
	
	This type of degenerate parabolic equation has been the subject of several classical works, such as \cite{DIB_1993, DiBenedetto+Friedman, DiBenedetto+Gianazza+Vespri-2008, DiBenedetto+Gianazza+Vespri-2012, DiBenedetto-1988, Kilpelainen_Lindqvist_1996}, which address issues of regularity, Hölder continuity, and Harnack inequalities. In addition, \cite{Wang+Yin+Cao-cauchy} studies the Cauchy problem in the whole space, while \cite{Porcio-t-evolution} analyzes the long-time evolution as $t\to\infty$ and the corresponding asymptotic behavior. Finally, \cite{Guidolin+Schutz+Ziebell-comparison} establishes a comparison principle for weak solutions under certain assumptions.
	
	After deriving the asymptotic mean value formula for equation~\eqref{eq_p-q}, we formulate and analyze the associated Poisson problem for the corresponding Dynamic Programming Principle (DPP)
	(that is, the equation that appears for $\epsilon$ fixed dropping the $o(\epsilon^2)$ term in the asymptotic mean value formula) 
	in a bounded domain with regular boundary $\Omega\subset\mathbb{R}^d$. More precisely, we show that, given an exterior datum and a right hand side, this DPP admits a unique solution, which moreover converges to the viscosity solution of equation~\eqref{eq_p-q} in the same bounded domain and with the same exterior datum.
	
	Finally, we describe the game induced by the asymptotic mean value operator. Using probabilistic arguments, we prove that the game possesses a value which coincides with the solution to the DPP. Furthermore, again by means of probabilistic techniques, we establish that viscosity solutions to the homogeneous equation~\eqref{eq_p-q}, namely when $f(x)=0$, are locally Hölder continuous in the interior.

	
	\section{Preliminaries and Main Results}\label{sec2}
	
	In this section let us state precisely the main results of this paper. 
	Before introducing our mean value operator for the $p$-Laplacian with $p>2$, we recall several fundamental results that will be used throughout this work. One of the key results in the development of the theory extending the classical mean value property for the Laplacian to nonlinear and non-homogeneous problems involving the $p$-Laplacian was established in~\cite{Manfredi_Parviainen_Rossi_2010}; see also~\cite{libroBR,Lewicka_book}.
	
	We refer to the following theorem.
	\begin{thm}\label{exp-plapN}
		Let $d\geq 1$, $p\geq 2$, $x\in\Rd$, $R>0$, and $\ph\in C^2(B_R(x))$ with $\gr\ph(x)\neq 0$. Let $\gamma=\sqrt{2(p+d)}$ and define
		\[
		M_\rho[\ph](x)
		\coloneqq
		\frac{p-2}{2(p+d)}
		\Big(
		\sup_{B_{\gamma \rho}(x)}\ph
		+
		\inf_{B_{\gamma \rho}(x)}\ph
		\Big)
		+
		\frac{d+2}{p+d}
		\intn_{B_{\gamma \rho}(x)} \ph(y)\dy.
		\]
		Then, for $\rho$ sufficiently small, we have
		\begin{equation*}
			M_\rho[\ph](x)-\ph(x)
			=
			\rho^2 \plapN \ph(x)
			+
			\rho^2 E(\rho,\ph,x),
		\end{equation*}
		where $E(\rho,\ph,x)=o_\rho(1)$.
		Moreover, if $\ph \in C^3(B_R(x))$, there exists a constant $k=k(p,d)>0$ such that
		\begin{equation}\label{error-norm-plaplacian}
			\left|E(\rho,\ph,x)\right|
			\leq
			k \rho
			\left(
			\|D^3\ph\|_{\LL^\infty(B_R(x))}
			+
			\frac{|D^2\ph(x)|^2}{|\gr\ph(x)|}
			\right).
		\end{equation}
	\end{thm}
	
	The corresponding asymptotic mean value formula for the signed power of the $p$-Laplacian, (remember that we have defined  $(a)^\beta = \operatorname{sgn}(a)\,|a|^\beta$, for $\beta>0$), presented in~\cite{delTeso+Fuertes+Rossi_2026}, follows from
	\[
	\left(\plap u\right)^{\frac{1}{p-1}}
	=
	\left|\gr u\right|^{\alpha_p}
	\left(\plapN u\right)^{1-\alpha_p},
	\qquad
	\mbox{with } \alpha_p=\frac{p-2}{p-1}.
	\]
	Here, we use the following numerical identity, valid for positive $a,b\in\RR$,
	\[
	a^\alpha b^{1-\alpha}
	=
	\inf_{c>0}
	\left\{
	\alpha c^{1-\alpha}a
	+
	(1-\alpha)c^{-\alpha}b
	\right\},
	\]
	together with the asymptotic mean value formulas for the gradient and the normalized $p$-Laplacian given in Theorem~\ref{exp-plapN}. Moreover, in order to obtain a closed-form expression that is independent of the sign of the $p$-Laplacian, we decompose its signed power into its positive and negative parts, thereby obtaining the following mean operator. 
	
	\begin{equation}\label{operator-A}
		\begin{split}
			\A_\epsilon^p[\ph](x)=&
			\frac{1}{2}\max\left[
			\inf_{c\in[m_p(\epsilon),M_p(\epsilon)]}
			\left\{
			\alpha_p \sup_{B_{2\epsilon^2 c^{1-\alpha_p}}(x)}\ph
			+
			(1-\alpha_p) M_{\sqrt{2}\epsilon c^{-\alpha_p/2}}[\ph](x)
			\right\},
			\ph(x)
			\right]\\
			&+\frac{1}{2}\min\left[
			\sup_{c\in[m_p(\epsilon),M_p(\epsilon)]}
			\left\{
			\alpha_p \inf_{B_{2\epsilon^2 c^{1-\alpha_p}}(x)}\ph
			+
			(1-\alpha_p) M_{\sqrt{2}\epsilon c^{-\alpha_p/2}}[\ph](x)
			\right\},
			\ph(x)
			\right].
		\end{split}
	\end{equation}
	For a more detailed motivation, see \cite{delTeso+Fuertes+Rossi_2026}.
	
	Here, we restrict the range of $c$ in the infimum and supremum appearing in the above formula to
	$c\in[m_p(\epsilon),M_p(\epsilon)]$
	in order to ensure that all the terms are well defined. More precisely, we set
	\begin{equation}\label{m-M}\tag{m-M}
		m_p(\epsilon)=\epsilon^{\frac{2(p-1)}{3p-4}}
		\hspace{3mm}\mbox{and}\hspace{3mm}
		M_p(\epsilon)=\epsilon^{-2+\frac{2}{p}}.
	\end{equation}
	
	We have the following asymptotic mean value formula (AMVF) for the signed power of the $p$-Laplacian operator.
	\begin{thm}\label{thm.asexp}
		Let $d\geq 1$, $p>2$, $x\in\Rd$, $R\in(0,1)$, and assume \eqref{m-M}. Then, there exists $\beps\in(0,R)$ such that for all $\epsilon\in(0,\beps)$ the operator \eqref{operator-A} is well defined for any bounded measurable function $\varphi$ in $B_{R}(x)$. Moreover, if $\ph\in C^2(B_{R}(x))$, then the following asymptotic expansions holds
		\begin{equation}\label{asexp}
			\A_\epsilon^p[\ph](x)=\ph(x)+\epsilon^2\left(\plap \ph(x)\right)^{\frac{1}{p-1}}+\epsilon^2 E_1(\epsilon, \ph, x).
		\end{equation}
		We further have the following estimates for the error term $E_1(\epsilon, \ph, x)$:
		\begin{enumerate}[label=(\alph*), leftmargin=*]
			\item\label{thm.asexp-item-a} If $\ph\in C^2(B_{R}(x))$, then $E_1(\epsilon, \ph, x)=o_\epsilon(1)$ as $\epsilon\to0^+$. 
			\item\label{thm.asexp-item-b} If $\ph\in C^3(B_{R}(x))$ and  $\nabla\varphi(x)\not=0$,  there exists constants $k_1=k_1(d,p)>0$, $k_2=k_2(d,p)>0$ such that 
			\begin{equation}\label{E1}
				\left|E_1(\epsilon,\ph,x)\right|\leq  k_1 \epsilon^{2-\frac{4}{p}}\|D^2\ph\|_{\LL^\infty(B_{\beps}(x))}
				+ k_2 \epsilon^{\frac{2}{3p-4}}\left(\|\gr\ph\|_{\LL^\infty(B_{\beps}(x))}+\|D^3\ph\|_{\LL^\infty(B_{\beps}(x))}+\frac{|D^2\ph(x)|^2}{|\gr\ph(x)|}\right).
			\end{equation}
			\item\label{thm.asexp-item-c} If $\ph\in C^2(B_{R}(x))$, there exists a constant $k_3=k_3(d,p)>0$ such that
			\begin{equation}\label{E1-gr-not}
				\left|E_1(\epsilon,\ph,x)\right|\leq k_3 \left( \|D^2\ph\|_{\LL^\infty(B_{\beps}(x))}^{\frac{1}{p-1}}|\gr \ph(x)|^{\frac{p-2}{p-1}}+ \epsilon^{2-\frac{4}{p}} \|D^2\ph\|_{\LL^\infty(B_{\beps}(x))}+\epsilon^{\frac{2}{3p-4}}|\gr\ph(x)|\right).
			\end{equation}
		\end{enumerate}
	\end{thm}
	
	\begin{rem}\rm
		First, we observe that the error estimate in Theorem \ref{thm.asexp}\ref{thm.asexp-item-a} holds for $C^2$ functions regardless of whether the gradient vanishes or not. However, the estimate obtained there is not locally uniform in $x$. In contrast, the error estimate in Theorem \ref{thm.asexp}\ref{thm.asexp-item-b}, which requires $C^3$ regularity, is locally uniform away from points where the gradient vanishes. Nevertheless, this estimate degenerates as one approaches points where the gradient is zero. This is precisely why Theorem \ref{thm.asexp}\ref{thm.asexp-item-c} is needed, since it allows us to avoid such degeneracies.  These features are essential for the characterization of viscosity solutions and for the convergence of the Dynamic programming principle established later. 
	\end{rem}
	
	In addition to the mean value operator $\A_\epsilon^p$, which approximates the signed power $(\Delta_p \ph)^{\frac{1}{p-1}}$, the other key ingredient in deriving our new formulation is the following numerical lemma. 
		\begin{lem}\label{signed-power-lemma}
		Let $a\in\RR$ and $r>1$. Then, the signed power of $a$ can be written as follows
		\begin{equation}\label{signed-power}
			(a)^r
			=
			\sup_{b>0}
			\left\{
			rab-(r-1)b^{\frac{r}{r-1}}
			\right\}
			+
			\inf_{b>0}
			\left\{
			rab+(r-1)b^{\frac{r}{r-1}}
			\right\}.
		\end{equation}
		Moreover, if we restrict the supremum and infimum to $b\in[m,M]$, then
		\[
		(a)^r
		=
		\sup_{b\in[m,M]}
		\left\{
		rab-(r-1)b^{\frac{r}{r-1}}
		\right\}
		+
		\inf_{b\in[m,M]}
		\left\{
		rab+(r-1)b^{\frac{r}{r-1}}
		\right\}
		+
		E(a,r,m,M),
		\]
		where
		\begin{equation}\label{signed-power-error}
			\left|E(a,r,m,M)\right|
			\leq
			2r
			\left(
			m^{\frac{r}{r-1}}
			+
			m\left|a\right|
			+
			\left|a\right|
			\left(\left|a\right|^{r-1}-M\right)_+
			\right).
		\end{equation}
	\end{lem}
	The proof can be found in Appendix~\ref{apenA}.
	
	Therefore, our strategy is to use the mean operator $\mathcal{A}_\epsilon^p$ to approximate $\left(\Delta_p u\right)^{\frac{1}{p-1}}$, together with the preceding lemma with $r=p-1$. We have
	\begin{align*}
		\plap \ph(x) = [\left(\Delta_p \phi  (x)\right)^{\frac{1}{p-1}}]^{p-1}
		&=
		\sup_{b>0}
		\left\{
		(p-1)b (\plap \ph(x))^{\frac{1}{p-1}}
		-
		(p-2)b^{\frac{p-1}{p-2}}
		\right\}\\
		&\quad+
		\inf_{b>0}
		\left\{
		(p-1)b (\plap \ph(x))^{\frac{1}{p-1}}
		+
		(p-2)b^{\frac{p-1}{p-2}}
		\right\}\\
		&\approx
		\sup_{b>0}
		\left\{
		\frac{\A_{\epsilon\sqrt{2(p-1)b}}^p[\ph](x)-\ph(x)}
		{2\epsilon^2}
		-
		(p-2)b^{\frac{p-1}{p-2}}
		\right\}\\
		&\quad+
		\inf_{b>0}
		\left\{
		\frac{\A_{\epsilon\sqrt{2(p-1)b}}^p[\ph](x)-\ph(x)}
		{2\epsilon^2}
		+
		(p-2)b^{\frac{p-1}{p-2}}
		\right\}.
	\end{align*}
	
	Thus, we define the mean value operator
	\begin{equation}\label{operator-B}
		\begin{split}
			\B_\epsilon^p[\ph](x)
			&=
			\frac{1}{2}
			\sup_{b\in[m_b,M_b]}
			\left\{
			\A_{\epsilon\sqrt{2(p-1)b}}^p[\ph](x)
			-
			2(p-2)\epsilon^2 b^{\frac{p-1}{p-2}}
			\right\}\\
			&\quad+
			\frac{1}{2}
			\inf_{b\in[m_b,M_b]}
			\left\{
			\A_{\epsilon\sqrt{2(p-1)b}}^p[\ph](x)
			+
			2(p-2)\epsilon^2 b^{\frac{p-1}{p-2}}
			\right\}.
		\end{split}
	\end{equation}
	Here, we restrict the range of $b\in [m_b,M_b]$ in order to ensure that all terms are well defined. In fact, we set
		\begin{equation}\label{m-M-b}
		m_b(\epsilon)=\epsilon^2
		\qquad\mbox{and}\qquad
		M_b(\epsilon)
		=
		\epsilon^{-\frac{1}{2}
			\min\left[
			\frac{p-2}{p-1},
			\frac{2}{3(p-1)}
			\right]}.
	\end{equation}
	Observe that, in the mean operator $\B_\epsilon^p$, the choice of the constant $b$ determines the size of the balls appearing in the operator $\A_\epsilon^p$, and this choice is controlled by the penalization factor.
	
	\subsection{Asymtotic Mean Value Formula}
	Now we are ready to establish the asymptotic expansions.
	
	\begin{thm}\label{main-asym-exp-plap}
		Let $d\in \NN$, $p>2$, $x\in\Rd$, $R\in(0,1)$ and assume that \eqref{m-M} and \eqref{m-M-b} holds. Then, there exists $\beps\in(0,R)$ such that for all $\epsilon\in(0,\beps)$ the operator \eqref{operator-B} is well defined for any bounded measurable function $\ph$ in $B_R(x)$. Moreover, if $\ph$ is regular enough, namely $\ph\in C^2(B_R(x))$, then the following asymptotic expansion holds
		\begin{equation*}
			\B_\epsilon^p[\ph](x)=\ph(x)+\epsilon^2 \plap \ph(x)+\epsilon^2 E_p(\epsilon,\ph,x).
		\end{equation*}
		We futher have the following estimates for the error term:
		\begin{enumerate}[label=(\alph*), leftmargin=*]
			\item If $\ph\in C^2(B_R(x))$, then $E_p(\epsilon,\ph,x)=o_\epsilon(1)$ as $\epsilon\to 0^+$.
			
			\item If $\ph\in C^3(B_R(x))$ and $\gr \ph(x)\not=0$, there exists positive constants $k_4$, $k_5$ and $k_6$ which only depends on $d$ and $p$ such that
			\begin{equation*}
				\begin{split}
					&\left|E_p(\epsilon,\ph,x)\right|
					\leq k_4 \Bigg(
					\epsilon^{\frac{2(p-1)}{p-2}}
					+\epsilon^2
					\|D^2\ph\|_{\LL^\infty(B_{\beps}(x))}^{\frac{1}{p-1}}
					\|\gr \ph\|_{\LL^\infty(B_{\beps}(x))}^{\frac{p-2}{p-1}}
					\\
					&\quad
					+\|D^2\ph\|_{\LL^\infty(B_{\beps}(x))}^{\frac{1}{p-1}}
					\|\gr\ph\|_{\LL^\infty(B_{\beps}(x))}^{\frac{p-2}{p-1}} 
					\Bigg(
					\|D^2\ph\|_{\LL^\infty(B_{\beps}(x))}^{\frac{p-2}{p-1}}
					\|\gr\ph\|_{\LL^\infty(B_{\beps}(x))}^{\frac{(p-2)^2}{p-1}}
					-\epsilon^{-\frac{1}{2}\min\left[
						\frac{p-2}{p-1},
						\, \frac{2}{3(p-1)}
						\right]}
					\Bigg)_+
					\Bigg)
					\\
					&\quad
					+ k_5
					\epsilon^{1-\frac{2}{p}}
					\|D^2\ph\|_{\LL^\infty(B_{\beps}(x))}
					+ k_6
					\epsilon^{\frac{1}{3p-4}}
					\Bigg(
					\|\gr\ph\|_{\LL^\infty(B_{\beps}(x))}
					+\|D^3\ph\|_{\LL^\infty(B_{\beps}(x))}
					+\frac{\left|D^2\ph(x)\right|^2}
					{\left|\gr \ph(x)\right|}
					\Bigg).
				\end{split}
			\end{equation*}
			
			\item If $\ph\in C^2(B_R(x))$ and we allow its gradient to vanish in $x$, then there exists a positive constants  $k_4$ and $k_7$ which only depends on $d$ and $p$ such that we may reformulate the previous estimate as follows
			\begin{equation*}
				\begin{split}
					&\left|E_p(\epsilon,\ph,x)\right|
					\leq k_4 \Bigg(
					\epsilon^{\frac{2(p-1)}{p-2}}
					+\epsilon^2
					\|D^2\ph\|_{\LL^\infty(B_{\beps}(x))}^{\frac{1}{p-1}}
					\|\gr \ph\|_{\LL^\infty(B_{\beps}(x))}^{\frac{p-2}{p-1}}
					\\
					&\quad
					+\|D^2\ph\|_{\LL^\infty(B_{\beps}(x))}^{\frac{1}{p-1}}
					\|\gr\ph\|_{\LL^\infty(B_{\beps}(x))}^{\frac{p-2}{p-1}} 
					\Bigg(
					\|D^2\ph\|_{\LL^\infty(B_{\beps}(x))}^{\frac{p-2}{p-1}}
					\|\gr\ph\|_{\LL^\infty(B_{\beps}(x))}^{\frac{(p-2)^2}{p-1}}
					-\epsilon^{-\frac{1}{2}\min\left[
						\frac{p-2}{p-1},
						\, \frac{2}{3(p-1)}
						\right]}
					\Bigg)_+
					\Bigg)
					\\
					&\quad
					+ k_7
					\Bigg(\epsilon^{-\frac{2}{3}}
					\|D^2\ph\|_{\LL^\infty(B_{\beps}(x))}^{\frac{1}{p-1}}\left|\gr \ph(x)\right|^{\frac{p-2}{p-1}}+\epsilon^{1-\frac{2}{p}}\|D^2\ph\|_{\LL^\infty(B_{\beps}(x))}+\epsilon^{\frac{1}{3p-4}}\left|\gr \ph(x)\right|
					\Bigg).
				\end{split}
			\end{equation*}
		\end{enumerate}
	\end{thm}
	
	One of the advantages of having asymptotic expansions for the $p$-Laplacian, this time without the power $(p-1)^{-1}$, is that they can be used for several equations with different homogeneities. The two main examples that we will consider in this work are \eqref{eq-parabolic} and \eqref{eq_p-q}; see \ref{sect-final} for other possible applications. Once Theorem \ref{main-asym-exp-plap} is established, the asymptotic expansions for the remaining equations involving the $p$-Laplacian follow directly.
	
	\begin{thm}\label{main-exp-parab}
		Let $d\in \NN$, $p>2$, $x\in \Rd$, $>0$, $R\in(0,1)$, $\tau\in(0,t)$ and assume that \eqref{m-M} and \eqref{m-M-b} hold. Then, there exists $\beps\in (0,\min[R,\tau])$ such that for all $\epsilon\in(0,\beps)$ the operator \eqref{operator-B} is well defined for any bounded measurable function $\ph$ in $B_R(x)\times [t-\tau,t+\tau]$. Moreover, if $\ph\in C^2(B_R(x)\times [t-\tau,t+\tau])$, then the following asymptotic expansion holds
		\begin{equation}
			\B_\epsilon^p[\ph](x,t)=\ph(x,t+\epsilon^2)+\epsilon^2 \left(\plap \ph(x,t)-\partial_t \ph(x,t)\right)+\epsilon^2 E(\epsilon,\ph,x,t).
		\end{equation}
		\begin{enumerate}[label=(\alph*), leftmargin=*]
			\item If $\ph\in C^2(B_R(x)\times [t-\tau,t+\tau])$, then $E(\epsilon,\ph,x,t)=o_\epsilon(1)$ as $\epsilon\to 0^+$.
			
			\item If $\ph\in C^3(B_R(x)\times [t-\tau,t+\tau])$, then
			$$\left|E(\epsilon,\ph,x,t)\right|\leq \left|E_p(\epsilon,\ph,x,t)\right|+\frac{\epsilon^2}{2}\|\partial_{tt} \ph\|_{\LL^\infty(t-\beps,t+\beps)}.$$
			Where $E_p$ is the error term that appears in Theorem \ref{main-asym-exp-plap}.
		\end{enumerate}
	\end{thm}
	
Before proceeding to the next result, we need to introduce some notation.
\begin{nota}\rm
	 Let $d\geq 2$, $1\leq n<d$, and $2<q<p$. We denote the points in $\mathbb{R}^d$ by $x=(x',x'')$, where
	$$
	x'=(x_1,\ldots,x_n)
	\quad \mbox{and} \quad
	x''=(x_{n+1},\ldots,x_d).
	$$
	
	On the one hand, $\Delta_p^{x'}u(x',x'')$ denotes the $p$-Laplacian acting on the first $n$ variables, while keeping the remaining $d-n$ variables fixed. Moreover, when we write the mean operator $\B_\epsilon^p[\ph](x',x'')$, we refer to the operator \eqref{operator-B} with exponent $p$ acting on the first $n$ variables, while keeping the remaining $d-n$ variables fixed (in the same way as time is kept fixed in the parabolic setting). Analogously, $\Delta_q^{x''}$ and $\B_\epsilon^q[\ph](x',x'')$ denote the corresponding operators acting on the last $d-n$ variables.
	
	Furthermore, we shall say that conditions \eqref{m-M} and \eqref{m-M-b} hold simultaneously for $p$ and $q$ when the corresponding restrictions on the infima and suprema of $\B_\epsilon^p$ and $\B_\epsilon^q$ hold with powers of $\epsilon$ having exponent $p$ in the former case and exponent $q$ in the latter.
\end{nota}

\begin{thm}\label{main-exp-aniso}
	Let $d\geq 2$ be a integer, $1\leq n<d$, $p>q>2$, $x=(x',x'')\in \Rd$ with $x'=(x_1,\ldots, x_n)$ and $x''=(x_{n+1},\ldots,x_d)$, $R\in(0,1)$ and assume that \eqref{m-M} and \eqref{m-M-b} hold simultaneously for $p$ and $q$. Then, there exists $\beps\in(0,R)$ such that for all $\epsilon\in(0,\beps)$ the following operator
	\begin{equation}\label{operator-C}
		\C_\epsilon[\ph](x)
		=
		\frac{1}{2}\B_\epsilon^p[\ph](x',x'')
		+
		\frac{1}{2}\B_\epsilon^q[\ph](x',x''),
	\end{equation}
	is well defined for any bounded measurable function $\ph$ in $B_R(x',x'')$. Moreover, if $\ph\in C^2(B_R(x',x''))$, then the following asymptotic expansion holds
	\begin{equation*}
		\C_\epsilon[\ph](x',x'')=\ph(x',x'')+\frac{\epsilon^2}{2}\left(\plap^{x'}\ph(x',x'')+\qlap^{x''}\ph(x',x'')\right)+\epsilon^2\left(E_p(\epsilon,\ph,x)+E_q(\epsilon,\ph,x)\right).
	\end{equation*}
	Here, these $E_p(\epsilon,\ph,x)$ and $E_q(\epsilon,\ph,x)$ are the error terms that appear in Theorem \ref{main-asym-exp-plap} with exponents $p$ and $q$ respectively.
\end{thm}

These lasts results allow us to provide asymptotic mean value characterizations of viscosity solutions to equations \eqref{eq-parabolic} and \eqref{eq_p-q}.

\begin{thm}
	The following statements are true:
	
	\begin{enumerate}[label=(\roman*)]
		\item Let $d\in \NN$, $p>2$, $T>0$, let $\Omega\subset \Rd$ be an open set and assume that \eqref{m-M} and \eqref{m-M-b} hold. Then, a continuous function $u$ is a viscosity solution of
		\begin{equation*}
			\partial_tu(x,t)=\plap u(x,t)
			\qquad\mbox{for all }(x,t)\in\Omega\times(0,\infty),
		\end{equation*}
		if and only if $u$ is a viscosity solution of the AMVF
		\begin{equation}
			u(x,t+\epsilon^2)
			=\B_\epsilon^p[u](x,t)+o(\epsilon^2)
			\qquad\mbox{as }\epsilon\to0^+.
		\end{equation}
		
		\item Let $1\leq n<d$, $2<q<p$, let $\Omega\subset \Rd$ be an open set and assume that \eqref{m-M} and \eqref{m-M-b} hold simultaneously for $p$ and $q$. Then, a continuous function $u$ is a viscosity solution of
		\begin{equation*}
			\plap^{x'}\ph(x',x'')+\qlap^{x''}\ph(x',x'')= f(x)
			\qquad\mbox{for all }x\in \Omega,
		\end{equation*}
		if and only if $u$ is a viscosity solution of the AMVF
		\begin{equation}
			u(x)=\C_\epsilon[u](x)-\frac{\epsilon^2}{2}f(x)+o(\epsilon^2)
			\qquad\mbox{as }\epsilon\to 0^+.
		\end{equation}
	\end{enumerate}
\end{thm}

\subsection{Dynamic Programming Principles}
	
	We present here a dynamic programming principle (DPP) for the following parabolic problem
	\begin{equation}\label{parabolic-problem}
		\left\{\begin{array}{cll} \partial_t u(x,t)=\plap u(x,t), &\mbox{in }\Omega\times (0,\infty),
			\\ u(x,0)=u_0(x), &\mbox{in } \Omega,
			\\ u(x,t)=g(x), &\mbox{in }\partial\Omega \times(0,\infty).\end{array}\right.
	\end{equation}
	Here, $\Omega\subset\Rd$ is a bounded smooth domain, and the initial and exterior data belong to the space of bounded continuous functions, namely, $g\in C_b(\Rd\setminus\Omega)$ and $u_0\in C_b(\overline{\Omega})$. We will assume that the boundary datum is extended to a function defined in $\Rd\setminus \Omega$.
	
	The AMVF given in Theorem \ref{main-exp-parab} allow us to derive the following DPP associated
	\begin{equation}\label{DPP-parab}\tag{DPP I}
		\left\{\begin{array}{cll} \ue(x,t)=\B_\epsilon^p[\ue](x,t-\epsilon^2), &\mbox{in }\Omega\times (0,\infty),
			\\ \ue(x,t)=u_0(x), &\mbox{in } \Omega\times (-\epsilon^2,0],
			\\ \ue(x,t)=g(x), &\mbox{in }\Rd\setminus \Omega\times (-\epsilon^2,\infty).\end{array}\right.
	\end{equation}
	
	The relationship between the two problems is that, as stated in the following result, the unique solution of the DPP converges to the unique solution of the parabolic problem.
	
	\begin{thm}\label{main-convergence-parabolic}
		Let $d\geq 1$, $p>2$, $\epsilon\in(0,1)$, and let $\Omega\subset\Rd$ be a bounded domain satisfying the uniform exterior ball condition. Let $u_0\in C_b(\Omega)$, $g\in C_b(\Rd\setminus\Omega)$, and assume that \eqref{m-M} and \eqref{m-M-b} hold.
		Then:
		\begin{enumerate}[label=(\roman*)]
			\item There exists a unique bounded measurable function $\ue$ that is the solution of \eqref{DPP-parab}.
			
			\item There exists a function $u\in C(\overline{ \Omega}\times[0,\infty))$ such that
			\begin{equation*}
					\ue(x,t)\to u(x,t)
				\quad \mbox{uniformly in } \overline{\Omega}\times\RR^+
				\quad \mbox{as } \epsilon\to 0^+.
			\end{equation*}
			This limit is the unique viscosity solution of \eqref{parabolic-problem}.
		\end{enumerate}
	\end{thm}

	We now turn to the analysis of the elliptic problem, 
		\begin{equation}\label{anosotropic-problem}
		\left\{\begin{array}{rlll}
			\plap^{x'}u(x',x'')+\qlap^{x''}u(x',x'')&=&f(x), & \mbox{in } \Omega,\\ 
			u(x)&=&g(x), & \mbox{in } \partial \Omega.
		\end{array}\right.
	\end{equation}
	where $\Omega\subset\mathbb{R}^d$ is a bounded smooth domain, and the data satisfy $f\in C(\overline{\Omega})$ and $g\in C_b(\mathbb{R}^d\setminus\Omega)$.
	
	Before moving on to the DPP corresponding to problem \eqref{anosotropic-problem}, some slight modifications to the operator are required. More specifically, we approximate the value of the function at a given point by its average over a small ball centered at that point. That is, we consider the following operators:
	
	\begin{equation}\label{operator-A-promedio}
		\begin{split}
			\widehat{\A^p_\epsilon}[\ph](x)&=\frac{1}{2}\max\left[ \inf_{c\in[m_p,M_p]}\left\{ \alpha_p \sup_{B_{2\epsilon^2 c^{1-\alpha_p}}(x)}\ph +(1-\alpha_p) M_{\sqrt{2}\epsilon c^{-\alpha_p/2}}[\ph](x) \right\}, \fint_{B_{\epsilon^8}(x)} \ph(y)\, \dy \right]\\
			&+\frac{1}{2} \min\left[ \sup_{c\in[m_p,M_p]}\left\{ \alpha_p \inf_{B_{2\epsilon^2 c^{1-\alpha_p}}(x)}\ph +(1-\alpha_p) M_{\sqrt{2}\epsilon c^{-\alpha_p/2}}[\ph](x) \right\}, \fint_{B_{\epsilon^8}(x)} \ph(y)\, \dy \right],
		\end{split}
	\end{equation} 
	and
	\begin{equation}\label{operator-B-promedio}
		\begin{split}
			\widehat{\B_\epsilon^p}[\ph](x)&=\frac{1}{2}\sup_{b\in[m_b,M_b]}\left\{\widehat{\A_{\epsilon\sqrt{2(p-1)b}}^p}[\ph](x)-2(p-2)\epsilon^2 b^{\frac{p-1}{p-2}}\right\}\\&+\frac{1}{2}\inf_{b\in[m_b,M_b]}\left\{\widehat{\A_{\epsilon\sqrt{2(p-1)b}}^p}[\ph](x)+2(p-2)\epsilon^2 b^{\frac{p-1}{p-2}}\right\}.
		\end{split}
	\end{equation}
	
	The reason for this modification is that, if we consider the operator as originally defined, an admissible strategy for both players is to remain stationary at all times. Indeed, in the maxima and minima appearing in~\eqref{operator-A}, one may always select $\ph(x)$, thereby preventing the game from ever terminating. However, these mean operators are equivalent to $\A_\epsilon^p$ and $\B_\epsilon^p$. In other words, they satisfy exactly the same results as those established in the previous section. The only difference lies in the estimate of the error term, where an additional term of the following form appears:
	\begin{equation}\label{error-A-promedio}
			\left|\widehat{E_1}(\epsilon,\ph,x)\right|
		\leq
		\left|E_1(\epsilon,\ph,x)\right|
		+
		\frac{\epsilon^{14}}{2}\,\|D^2 \ph\|_{\LL^\infty(B_{\beps}(x))},
	\end{equation}
	in Theorem \ref{thm.asexp} and
	\begin{equation}\label{error-B-promedio}
		\left|\widehat{E_p}(\epsilon,\ph,x)\right|
		\leq
		\left|E_p(\epsilon,\ph,x)\right|
		+
		k_8 \epsilon^{11} \|D^2\ph\|_{\LL^\infty(B_{\beps}(x))},
	\end{equation}
	in Theorem \ref{main-asym-exp-plap}.
	
	Then the corresponding AMVF give us the following DPP associated
	\begin{equation}\label{DPP-aniso}\tag{DPP II}
		\left\{\begin{array}{rlll}
			\ue(x)&=&\frac{1}{2}\left(\widehat{\B_\epsilon^p}[\ue](x',x'')+\widehat{\B_\epsilon^q}[\ue](x',x'')\right)-\frac{\epsilon^2}{2}f(x), & \mbox{in } \Omega,\\ 
			\ue(x)&=&g(x), & \mbox{in } \Rd\setminus \Omega.
		\end{array}\right.
	\end{equation}
	
		\begin{thm}\label{main-convergence-elliptic}
		Let $d\geq 2$, $p>q>2$, $\epsilon\in (0,1)$, and let $\Omega\subset \Rd$ be a bounded domain satisfying the uniform exterior ball condition. Assume that \eqref{m-M} and \eqref{m-M-b} hold for both $p$ and $q$. Then:
		\begin{enumerate}[label=(\roman*)]
			\item There exists a unique bounded measurable function $\ue$ that is the solution of \eqref{DPP-aniso}.
			
			\item There exists a function $u\in C(\overline{ \Omega})$ such that
			\begin{equation*}
				\ue \longrightarrow u \qquad \mbox{ as }\epsilon\to 0^+\quad \mbox{uniformly in }\overline{\Omega}.
			\end{equation*}
			This limit coincides with the viscosity solution $u$ of~\eqref{anosotropic-problem}.
		\end{enumerate}
	\end{thm}
	
	\subsection{Game Theoretical Interpretation}\label{sec-main-game}
	
	Both of the DPP are related with different two-players zero-sum games. Firstly we describe the game in the elliptic case. Before proceeding, we recall the game associated with $M_\rho$, known as \emph{tug-of-war with noise}
	
	\textbf{Description of the tug-of-war with noise for $p$-harmonic functions.}
	Fix a parameter $\rho>0$, an open bounded domain $\Omega$, a starting point $x_0\in\Omega$, and a  final payoff function $g\in C_b(\Rd\setminus \Omega)$.
	\begin{enumerate}[label=(\roman*)]
		\item At each round, a biased coin is tossed with probabilities $\beta_p=(p-2)/(p+d)$ for heads and $1-\beta_p$ for tails.
		\begin{itemize}[leftmargin=1em]
			\item If the outcome of the coin toss is tails, the next position of the game is chosen randomly inside the ball of radius $\sqrt{2(p+d)}\rho$ (with uniform distribution).
			\item If the outcome is heads, a round of tug-of-war is played: a fair coin
			(probabilities $1/2-1/2$) determines which of the two players chooses the next position of the game inside the ball of radius $\sqrt{2(p+d)}\rho$.
		\end{itemize}
		
		\item The procedure described in (i) is repeated starting at the new position until the position of the game exits $\Omega$ for the first time; this exit point is denoted by $x_\tau$. At that moment, the second player pays the amount $g(x_\tau)$ to the first player.
	\end{enumerate}
	
	Once we have described \emph{tug-of-war with noise}, we can introduce the two-player zero-sum game associated with the operator $\widehat{\A_\epsilon^p}$.
	
	\textbf{Description of the game for $\widehat{\A_\epsilon^p}$.}
	Fix $\epsilon>0$, a bounded open domain $\Omega$, a starting position $x_0\in\Omega$, and a final payoff function $g\in C_b(\Rd\setminus \Omega)$.
	
	\begin{enumerate}[label=(\roman*)]
		
		\item[(I)] At each round, a fair coin is tossed, with probability $1/2$ of landing heads and probability $1/2$ of landing tails.
		
		\item[(II)(a)] If the outcome of (I) is heads, Player $I$ may either move randomly within a ball of radius $\epsilon^8$ centered at $x_0$ or initiate the following process. Player $II$ selects a parameter $c\in[m_p(\epsilon),M_p(\epsilon)]$. After this choice, a biased coin is tossed, with probabilities $\alpha_p=(p-2)/(p-1)$ of landing heads and $1-\alpha_p$ of landing tails:
		\begin{itemize}[leftmargin=1em]
			\item If the outcome is heads, Player $I$ chooses $x_{1}\in B_{2\epsilon^2 c^{1-\alpha_p}}(x_0)$.
			\item If the outcome is tails, a round of tug-of-war with noise is played in $B_{\sqrt{4(p+d)}\epsilon c^{-\alpha_p/2}}(x_0)$, with probabilities $\beta_p$ and $1-\beta_p$ as described above, which determines $x_{1}$.
		\end{itemize}
		
		\item[(II)(b)] If the outcome of (I) is tails, the same procedure as in (II)(a) is followed, but with the roles of Player $I$ and Player $II$ interchanged.
		
		\item[(III)] The procedure described in (I) and (II) determines $x_1$ and is repeated until the end of the game, which occurs when the position $x_\tau$ leaves $\Omega$ for the first time. The final payoff for a particular realization of the game is given by Player $II$ paying $g(x_{\tau})$ to Player $I$.
		
	\end{enumerate}
	
	\begin{rem} {\rm In this game, the player who chooses $c$ determines the size of the balls in which they can move the token, but this choice also affects the size of the movements in the tug-of-war with noise.}
	\end{rem}
	
	\textbf{Description of the game for $\widehat{\B_\epsilon^p}$.}
	We now describe a game whose value function is related to the mean value formula for the $p$-Laplacian associated with $\widehat{\B_\epsilon^p}$.
	
	Fix $\epsilon>0$, a bounded open domain $\Omega$, a starting point $x_0\in\Omega$, a final payoff function $g\in C_b(\Rd\setminus \Omega)$, and a running payoff function $f\in C(\overline{\Omega})$.
	
	The rules of the game are as follows:
	
	\begin{enumerate}[label=(\roman*)]
		
		\item[(I)] At each round, a fair coin is tossed, with probability $1/2$ of landing heads and probability $1/2$ of landing tails.
		
		\item[(II)(a)] If the outcome of (I) is heads, Player $I$ may choose a parameter $b\in[m_b(\epsilon),M_b(\epsilon)]$ (recall that $m_b$ and $M_b$ are given by \eqref{m-M-b}). With this choice of $b$, a round of the game associated with $\widehat{\A_{\epsilon\sqrt{2(p-1)b}}^p}$ is played, with the parameter controlling the size of the steps set to $\epsilon\sqrt{2(p-1)b}$. In addition, Player $I$ pays $2(p-2)\epsilon^2 b^{\frac{p-1}{p-2}}$ to Player $II$.
		
		\item[(II)(b)] If the outcome of (I) is tails, the same procedure as in (II)(a) is followed, but with the roles of Player $I$ and Player $II$ interchanged (this time, Player $II$ chooses $b\in[m_b(\epsilon),M_b(\epsilon)]$ and pays $2(p-2)\epsilon^2 b^{\frac{p-1}{p-2}}$ to Player $I$).
		
		\item[(III)] At each turn, independently of the coin toss in (I), Player $I$ receives from Player $II$ a running payoff given by $\frac12\epsilon^2 f(x_0)$.
		
		\item[(IV)] The procedure described in (I), (II)(a), (II)(b) and (III) dertermines the next position $x_1$ and is repeated until the position $x_\tau$ leaves $\Omega$ for first time. The final payoff for a particular realization of the game is given as follows: if $x_{\tau}\notin\Omega$, Player $II$ pays $g(x_{\tau})$ to Player $I$.
		
	\end{enumerate}
	
	\begin{rem} {\rm In this game, the player who chooses $b$ determines the size of the balls in which a round of the previous game is played, but this choice also entails a larger penalty when the desired size of the movements is large.}
	\end{rem}
	
	\textbf{Description of the game for $\widehat{\C_\epsilon}=\frac{1}{2}\left(\widehat{\B_\epsilon^p}+\widehat{\B_\epsilon^q}\right)$.}
	The previous construction also allows us to provide a game-theoretic interpretation of problem \eqref{anosotropic-problem} through \eqref{DPP-aniso}. The game again involves two players: Player~$I$, who aims to maximize the final payoff, and Player~$II$, who aims to minimize it.
	
	As before, fix $\epsilon>0$, a bounded open domain $\Omega$, a starting point $x_0\in\Omega$, a final payoff function $g\in C_b(\Rd\setminus \Omega)$, and a running payoff function $f\in C(\overline{\Omega})$. Recall that we use the following notation for $x=(x',x'')\in \Rd$, where $x'=(x_1,\cdots,x_n)$ and $x''=(x_{n+1},\cdots,x_d)$.
	
	The rules of the game are as follows:
	
	\begin{enumerate}[label=(\roman*)]
		
		\item[(I)] At each round, a fair coin is tossed, with probability $1/2$ of landing heads and probability $1/2$ of landing tails.
		
		\item[(II)(a)] If the outcome of (I) is heads, a round of the game associated with $\widehat{\B_\epsilon^p}$ is played in balls in $\mathbb{R}^n$ centered at $x_0=(x_0',x_0'')\in \Rd$, with the last coordinates $x_0''$ kept fixed.
		
		\item[(II)(b)] If the outcome of (I) is tails, a round of the game associated with $\widehat{\B_\epsilon^q}$ is played in balls in $\mathbb{R}^{d-n}$ centered at $x_0=(x_0',x_0'')\in \Rd$, with the first coordinates $x_0'$ kept fixed.
		
		\item[(III)] At each turn, independently of the coin toss in (I), Player $I$ receives a running payoff given by $\frac12\epsilon^2 f(x_0)$ from Player $II$.
		
		\item[(IV)] The procedure described in (I), (II)(a), (II)(b) and (III) determines the next position $x_1$ and is repeated until the position $x_\tau$ leaves $\Omega$ for first time. The final payoff for a particular realization of the game is given as follows: if $x_{\tau}\notin\Omega$, Player $II$ pays $g(x_{\tau})$ to Player $I$.
		
	\end{enumerate}
	
	\begin{rem} {\rm Now, the movements of the game take place inside horizontal balls (balls in $\mathbb{R}^n$ centered at $x=(x',x'')\in \Rd$, with the last coordinates $x''$ kept fixed) or vertical balls (balls in $\mathbb{R}^{d-n}$ centered at $x=(x',x'')\in \Rd$, with the first coordinates $x'$ kept fixed). This reflects the anisotropic character of the problem under consideration.}
	\end{rem}
	
	For the latter game described above, we have the following result concerning its value.
	\begin{thm} \label{main-juego-p-q}
		Let the assumptions of Theorem \ref{main-convergence-elliptic} hold. Then, the game associated with $\widehat{\C_\epsilon}$ has a value and it is given by the unique solution to \eqref{DPP-aniso}.
		
		In particular, the value of the game converges uniformly in $\overline{\Omega}$ as $\epsilon\to 0^+$ to viscosity solution to the elliptic problem \eqref{anosotropic-problem}.
	\end{thm}
	
	We now turn our attention to the parabolic DPP and its associated game, which we proceed to describe. Before that, we recall its parabolic version—that is, the version in space and time—generated by the operator $\A_\epsilon^p$, which was previously studied in~\cite{delTeso+Fuertes+Rossi_2026}.
	
	\begin{enumerate}[label=(\roman*)]
		
		\item[(I)] At each round, a fair coin is tossed, with probability $1/2$ of landing heads and probability $1/2$ of landing tails.
		
		\item[(II)(a)] If the outcome of (I) is heads, Player~$I$ may either keep the spatial position, $x_{k+1}=x_k$, and update the time to $t_{k+1}=t_k-\epsilon^2$, or initiate the following procedure. Player~$II$ selects a parameter $c\in[m(\epsilon),M(\epsilon)]$, after which a biased coin is tossed, with probabilities $\alpha_p=(p-2)/(p-1)$ of landing heads and $1-\alpha_p$ of landing tails:
		\begin{itemize}[leftmargin=1em]
			\item If the outcome is heads, Player~$I$ chooses $x_{k+1}\in B_{2\epsilon^2 c^{1-\alpha_p}}(x_k)$ and $t_{k+1}=t_k-\epsilon^2$.
			\item If the outcome is tails, a round of tug-of-war with noise is played in $B_{\sqrt{4(p+d)}\epsilon c^{-\alpha_p/2}}(x_k)$, with probabilities $\beta_p=(p-2)/(p+d)$ and $1-\beta_p$ (this game was above), which determines $x_{k+1}$, and $t_{k+1}=t_k-\epsilon^2$.
		\end{itemize}
		
		\item[(II)(b)] If the outcome of (I) is tails, the same procedure as in (II)(a) is followed, with the roles of Player~$I$ and Player~$II$ interchanged.
		
		\item[(III)] The procedure described in (I) and (II) is repeated until the game ends. It happends when either $x_k\not \in \Omega$ and $t_k>0$ or $x_k\in \Omega$ and $t_k<0$. After that, the final payoff is determined as follows: if $t_{k+1}\leq 0$ and $x_{k+1}\in\Omega$, Player~$II$ pays $u_0(x_{k+1})$ to Player~$I$; if $x_{k+1}\notin\Omega$ and $t_{k+1}>0$, Player~$II$ pays $g(x_{k+1})$ to Player~$I$.
	\end{enumerate}
	
	We now describe the game associated with \eqref{DPP-parab}. It is a two-player zero-sum game in which Player~I aims to maximize the payoff, whereas Player~II aims to minimize it. Given $\epsilon>0$, a bounded domain $\Omega\subset\Rd$ satisfying the uniform exterior ball condition, $g\in C_b(\Rd\setminus\Omega)$, $u_0\in C_b(\Omega)$, and an initial position $(x_0,t_0)\in\Omega\times(0,\infty)$, the position evolves according to the following rules.
	
	\textbf{Description of the game for $\B_\epsilon^p$.}
	\begin{enumerate}[label=(\roman*)]
		\item[(I)] At each round, a fair coin is tossed, with probability $1/2$ of landing heads and probability $1/2$ of landing tails.
		\item[(II)(a)] If the outcome of (I) is heads, Player~I chooses $b\in[m_b,M_b]$. After that, a round of the game described above is played with the new radius $\epsilon\sqrt{2(p-1)b}$ instead of $\epsilon$, and Player~I pays the amount $2(p-2)\epsilon^2 b^{\frac{p-1}{p-2}}$ to Player~II. The next spatial position $x_{k+1}$ is determined by the game associated with $\A^p_{\epsilon\sqrt{2(p-1)b}}$, and $t_{k+1}=t_k-\epsilon^2$.
		
		\item[(II)(b)] If the outcome of (I) is tails, the same procedure as in (II)(a) is followed, with the roles of Player~$I$ and Player~$II$ interchanged.
		
		\item[(III)] The procedure in (I), (II)(a), and (II)(b) is repeated until the game ends. This happens when $t_{k+1}\leq 0$ for the first time and $x_{k+1}\in \Omega$, in which case Player~II pays $u_0(x_{k+1})$ to Player~I, or when $x_{k+1}\notin \Omega$ with $t_{k+1}>0$, in which case Player~II pays $g(x_{k+1})$ to Player~I.
	\end{enumerate}

	\begin{thm}\label{main-game-parabolic}
		Let the assumptions of Theorem \ref{main-convergence-parabolic} hold. Then, the game associated with $\B_\epsilon^p$, defined above has a value  and it is given by the unique solution to \eqref{DPP-parab}.
		
		In particular, the value of the game converges uniformly on $\overline{\Omega}\times\RR^+$ as $\epsilon\to 0^+$ to the viscosity solution of \eqref{parabolic-problem}.
	\end{thm}
	
	\subsection{Hölder continuity and Harnack Inequality}
	
	Finally, we establish interior Hölder regularity results for the solutions to the equation in \eqref{anosotropic-problem} with $f\equiv 0$, by means of probabilistic arguments. This Hölder estimate will be the key to subsequently obtaining classical Harnack-type inequalities.
	
	\begin{thm}\label{main-holder}
		Let the assumptions of Theorem \ref{main-convergence-elliptic} hold. Let $u$ be the viscosity solution of \eqref{anosotropic-problem} in $B_R(0)$ with $R>1$. Assume also that $f=0$. Then, the following estimates hold:
		
		\begin{enumerate}[label=(\roman*)]
			\item There exists a constant $K>0$, depending only on $p$, $q$, and $d$, such that, for every $0<\delta<1$,
			\begin{equation*}
				\sup_{B_\delta} u - \inf_{B_\delta} u \leq (K+2)\delta^{1/3}\left(\sup_{B_1} u -\inf_{B_1} u\right).
			\end{equation*}
			
			\item Moreover, if $u$ is positive, then it satisfies the Harnack Inequality, there exists a constant $H>0$ such that
			\begin{equation*}
				\sup_{B_1(0)} u \leq H \inf_{B_1(0)} u.
			\end{equation*}
		\end{enumerate}
	\end{thm}
	
	\noindent\textbf{Plan of the paper.}
	Section \ref{sec3} is devoted to the asymptotic expansions of the $p$-Laplacian and the equations \eqref{eq_p-q} and \eqref{eq-parabolic}, together with the equivalence between viscosity solutions of the AMVF and of the corresponding equations. We then turn to problems \eqref{parabolic-problem} and \eqref{anosotropic-problem}, starting with the elliptic case. In Section \ref{sec4}, we establish existence and uniqueness for the associated DPP and prove the uniform convergence, as $\epsilon\to0^+$, of its solutions to the unique viscosity solution of \eqref{anosotropic-problem}. A key ingredient is the construction of a family of smooth barriers independent of $\epsilon$.
	
	Section \ref{sec5} concerns the game associated with $\widehat{\C_\epsilon}$, for which we establish the existence of a value and its coincidence with the unique solution of the DPP. In Section \ref{sec6}, probabilistic arguments yield the interior Hölder continuity, with exponent $\beta=1/3$, of solutions to \eqref{anosotropic-problem} with $f=0$.
	
	Section \ref{sec7} is devoted to the parabolic problem \eqref{parabolic-problem}, where we extend the analysis of Sections \ref{sec4} and \ref{sec5}. This case is treated last, as several tools and barriers developed in \cite{delTeso+Fuertes+Rossi_2026} are readily applicable. Finally, Section \ref{sect-final} outlines possible extensions of our approach, while the appendices collect the technical computations used throughout the paper and recall the relevant notions of viscosity solution.
	
	
	\section{Asymptotic expansions for $C^2$ and $C^3$ functions}\label{sec3}
	
	The first objective of this section is to prove rigorously that the operator $\B_\epsilon^p$ (and, consequently, $\widehat{\B_\epsilon^p}$) is well defined for any bounded measurable function in a neighborhood of a given point. We then derive the corresponding asymptotic expansions for sufficiently regular functions. More precisely, we first show that the mean operator $\B_\epsilon^p$ approximates the $p$-Laplacian and subsequently use this result to obtain asymptotic expansions for \eqref{eq-parabolic} and \eqref{eq_p-q}.
	
	To carry out these arguments rigorously, we need suitable bounds for the sets over which the relevant infima and suprema are taken. While particular choices were introduced in the previous section in \eqref{m-M} and \eqref{m-M-b}, a more general choice is possible here. This flexibility allows us both to establish the well-definedness of $\B_\epsilon^p$ and to derive the asymptotic expansions without requiring explicit control of the error term.
	
	\begin{equation}\label{HT}\tag{HT}
		\begin{split}
			&m_p:\RR_+\longrightarrow \RR_+
			\mbox{ is nondecreasing, with }
			m_p(\epsilon)\to 0
			\mbox{ and }
			\epsilon\,m_p(\epsilon)^{-\alpha_p/2}\to 0
			\mbox{ as }\epsilon\to 0^+,\\
			&M_p:\RR_+\longrightarrow \RR_+
			\mbox{ is nonincreasing, with }
			M_p(\epsilon)\to +\infty
			\mbox{ and }
			\epsilon^2 M_p(\epsilon)^{1-\alpha_p}\to 0
			\mbox{ as }\epsilon\to 0^+,
		\end{split}
	\end{equation}
	and
	\begin{equation}\label{HT-b}
		\begin{split}
			&m_b:\RR_+\longrightarrow \RR_+
			\mbox{ is nondecreasing, with }
			m_b(\epsilon)\to 0,\\
			&M_b:\RR_+\longrightarrow \RR_+
			\mbox{ is nonincreasing, with }
			M_b(\epsilon)\to +\infty
			\mbox{ such that}\\
			&\qquad\qquad\qquad
			\epsilon m_p^{-\alpha_p/2}M_b^{1/2}\to 0,
			\qquad
			\epsilon^2 M_p(\epsilon)^{1-\alpha_p}M_b\to 0
			\qquad\mbox{as }\epsilon\to 0^+.
		\end{split}
	\end{equation}
	It can be checked that the bounds \eqref{m-M} and \eqref{m-M-b} satisfy \eqref{HT} and \eqref{HT-b}. We now show that the operators $\B_\epsilon^p$ and $\C_\epsilon$ are well defined under suitable assumptions.
	
	\begin{lem}\label{well-def-B}
		Let $d\in\NN$, $p>2$, $x\in\Rd$, $R>0$, and let $\ph:B_R(x)\longrightarrow\RR$ be a measurable and bounded function. Assume that~\eqref{HT} and~\eqref{HT-b} hold. Then, there exists $\beps\in(0,1)$ such that, for all $\epsilon\in(0,\beps)$, the operator $\B_\epsilon^p$ is well defined.
		
		In particular the mean operator $\C_\epsilon$ is well defined too.
	\end{lem}
	
	\begin{proof}
		By~\eqref{HT} and~\eqref{HT-b}, there exists $\beps\in(0,1)$ sufficiently small such that
		\[
		B_{2(p-1)b\epsilon^2c^{1-\alpha_p}}(x)
		\subset
		B_{2(p-1)M_b\epsilon^2M_p^{1-\alpha_p}}(x)
		\subset
		B_R(x),
		\]
		and
		\[
		B_{\sqrt{4(p+d)(p-1)b}\epsilon c^{-\alpha_p/2}}(x)
		\subset
		B_{\sqrt{4(p+d)(p-1)M_b}\epsilon m_p^{-\alpha_p/2}}(x)
		\subset
		B_R(x).
		\]
		Therefore, all the infima, suprema, and integrals appearing in $\A_\epsilon^p$ are well defined and finite. This, in turn, implies that the supremum and infimum appearing in $\B_\epsilon^p$ are also well defined. 
		
		Note that $2(p-2)\epsilon^2b^{\frac{p-1}{p-2}}>0$, so the presence of this term does not affect the well-definedness of either the infimum or the supremum, despite its negative sign in the latter.
		
		The mean operator $\C_\epsilon$ is also well defined since $\C_\epsilon[\ph](x) = \frac{1}{2} \left(\B_\epsilon^p[\ph](x',x'') + \B_\epsilon^q[\ph](x',x'')\right)$.
	\end{proof}
	
	We are now ready to establish the asymptotic expansion for $C^2$ functions.
	
	\begin{thm}\label{exp-c2}
		Let $d\in\NN$, $p>2$, $x\in\Rd$, $R>0$, and let $\ph\in C^2(B_R(x))$. Assume that~\eqref{HT} and~\eqref{HT-b} hold. Then, there exists $0<\beps<R$ such that, for all $\epsilon\in(0,\beps)$,
		\begin{equation}\label{exp-c2-formula}
			\B_\epsilon^p[\ph](x)
			=
			\ph(x)
			+
			\epsilon^2\plap\ph(x)
			+
			o(\epsilon^2).
		\end{equation}
	\end{thm}
	\begin{proof}
		
		A direct computation, using the result established in Theorem~\ref{thm.asexp} together with formula~\eqref{operator-B}, yields
		\begin{equation}\label{expression}
			\begin{split}
				\frac{\B_\epsilon^p[\ph](x)-\ph(x)}{\epsilon^2}
				&=
				\sup_{b\in[m_b,M_b]}
				\left\{
				(p-1)b
				\frac{\A_{\epsilon\sqrt{2(p-1)b}}^p[\ph](x)-\ph(x)}
				{\epsilon^2 2(p-1)b}
				-
				(p-2)b^{\frac{p-1}{p-2}}
				\right\}\\
				&\quad+
				\inf_{b\in[m_b,M_b]}
				\left\{
				(p-1)b
				\frac{\A_{\epsilon\sqrt{2(p-1)b}}^p[\ph](x)-\ph(x)}
				{\epsilon^2 2(p-1)b}
				+
				(p-2)b^{\frac{p-1}{p-2}}
				\right\}\\
				&=
				\sup_{b\in[m_b,M_b]}
				\left\{
				(p-1)b
				\left(
				\left(\plap\ph(x)\right)^{\frac{1}{p-1}}
				+
				o_\epsilon(1)
				\right)
				-
				(p-2)b^{\frac{p-1}{p-2}}
				\right\}\\
				&\quad+
				\inf_{b\in[m_b,M_b]}
				\left\{
				(p-1)b
				\left(
				\left(\plap\ph(x)\right)^{\frac{1}{p-1}}
				+
				o_\epsilon(1)
				\right)
				+
				(p-2)b^{\frac{p-1}{p-2}}
				\right\}.
			\end{split}
		\end{equation}
		Observe that we have used the fact that
		$\epsilon^2 2(p-1)b\to 0$ as $\epsilon\to 0^+$.
		We aim to obtain upper and lower bounds for the expression in~\eqref{expression}.
		Let us fix $\delta>0$. Since~\eqref{HT} holds, there exists $\beps>0$ such that, for all $\epsilon\in(0,\beps)$,
		\[
		-\delta\leq o_\epsilon(1)\leq\delta.
		\]
		
		We shall show that
		$$
		\frac{\B_\epsilon^p[\ph](x)-\ph(x)}{\epsilon^2}
		=
		\plap \ph(x)+o_\epsilon(1).
		$$
		To this end, we bound the quantity above from both below and above by $\plap \ph(x)+o_\epsilon(1)+o_\delta(1)$.
		
		\noindent\textit{Upper bound.}
		Starting from~\eqref{expression}, we obtain
		\begin{align*}
			\frac{\B_\epsilon^p[\ph](x)-\ph(x)}{\epsilon^2}
			&\leq
			\sup_{b\in[m_b,M_b]}
			\left\{
			(p-1)b
			\left(
			\left(\plap\ph(x)\right)^{\frac{1}{p-1}}
			+\delta
			\right)
			-
			(p-2)b^{\frac{p-1}{p-2}}
			\right\}\\
			&\quad+
			\inf_{b\in[m_b,M_b]}
			\left\{
			(p-1)b
			\left(
			\left(\plap\ph(x)\right)^{\frac{1}{p-1}}
			+\delta
			\right)
			+
			(p-2)b^{\frac{p-1}{p-2}}
			\right\}.
		\end{align*}
		
		Assume first that $\plap\ph(x)\geq0$. Then the infimum above is attained at $b=m_b(\epsilon)$ and is therefore equal to
		\[
		(p-1)m_b(\epsilon)
		\left(
		\left(\plap\ph(x)\right)^{\frac{1}{p-1}}+\delta
		\right)
		+
		(p-2)m_b(\epsilon)^{\frac{p-1}{p-2}}
		=
		o_\epsilon(1),
		\]
		see~\eqref{signed-power-error}, with $r=p-1$ and
		$a=\left(\plap\ph(x)\right)^{\frac{1}{p-1}}+\delta$.
		Therefore, by Lemma~\ref{signed-power-lemma},
		\begin{align*}
			\frac{\B_\epsilon^p[\ph](x)-\ph(x)}{\epsilon^2}
			&\leq
			\sup_{b\in[m_b,M_b]}
			\left\{
			(p-1)b
			\left(
			\left(\plap\ph(x)\right)^{\frac{1}{p-1}}
			+\delta
			\right)
			-
			(p-2)b^{\frac{p-1}{p-2}}
			\right\}
			+
			o_\epsilon(1)\\
			&=
			\left(
			\left(\plap\ph(x)\right)^{\frac{1}{p-1}}
			+\delta
			\right)^{p-1}
			+
			o_\epsilon(1)\\
			&\leq
			\plap\ph(x)
			+
			o_\delta(1)
			+
			o_\epsilon(1).
		\end{align*}
		
		Now assume that $\plap\ph(x)<0$. In this case, the supremum in~\eqref{expression} is equal to $o_\epsilon(1)$; see the precise estimate in~\eqref{signed-power-error}. Hence, by Lemma~\ref{signed-power-lemma},
		\begin{align*}
			\frac{\B_\epsilon^p[\ph](x)-\ph(x)}{\epsilon^2}
			&\leq
			\inf_{b\in[m_b,M_b]}
			\left\{
			(p-1)b
			\left(
			\left(\plap\ph(x)\right)^{\frac{1}{p-1}}
			+\delta
			\right)
			+
			(p-2)b^{\frac{p-1}{p-2}}
			\right\}
			+
			o_\epsilon(1)\\
			&=
			\left(
			-\left(-\plap\ph(x)\right)^{\frac{1}{p-1}}
			+\delta
			\right)^{p-1}
			+
			o_\epsilon(1)\\
			&\leq
			\plap\ph(x)
			+
			o_\delta(1)
			+
			o_\epsilon(1).
		\end{align*}
		
		Thus, regardless of the sign of the $p$-Laplacian, we conclude that
		\begin{equation}\label{upper-b}
			\frac{\B_\epsilon^p[\ph](x)-\ph(x)}{\epsilon^2}
			\leq
			\plap\ph(x)
			+
			o_\delta(1)
			+
			o_\epsilon(1).
		\end{equation}
		
		\noindent\textit{Lower Bound: } If we apply the upper bound to $-\ph$, it is standard to check that 
		\begin{equation}\label{lower-b}
			\frac{\B_\epsilon^p[\ph](x)-\ph(x)}{\epsilon^2}\geq \plap \ph(x)+o_\delta(1)+o_\epsilon(1).
		\end{equation}
		
		Then, by the arbitrariness of $\delta$, we conclude from the estimates \eqref{upper-b} and \eqref{lower-b}  that 
		$$\frac{\B_\epsilon^p [\ph](x)-\ph(x)}{\epsilon^2}= \plap\ph(x) + o_\epsilon(1),$$
		which completes the proof.
	\end{proof}
	
	This latter asymptotic expansion allows us to derive analogous expansions for several equations involving the $p$-Laplacian, in this paper we are focusing in \eqref{eq_p-q} and~\eqref{eq-parabolic}.
	
	\begin{cor}\label{exp-c2-p_q}
		Let $d\geq 2$, $p>q>2$, $R>0$, and $x=(x',x'')\in\Rd$, with $x'=(x_1,\cdots,x_n)$ and $x''=(x_{n+1},\cdots,x_d)$. Let $\ph\in C^2(B_R(x))$ and assume that~\eqref{HT} and~\eqref{HT-b} hold for both $p$ and $q$. Then, there exists $0<\beps<R$ such that, for all $\epsilon\in(0,\beps)$,
		\begin{equation}
			\C_\epsilon[\ph](x',x'')
			=
			\ph(x',x'')
			+
			\frac{\epsilon^2}{2}
			\left(
			\plap^{x'}\ph(x',x'')
			+
			\qlap^{x''}\ph(x',x'')
			\right)
			+
			o(\epsilon^2).
		\end{equation}
	\end{cor}
	
	\begin{proof}
		This result follows directly from Theorem~\ref{exp-c2}, since
		\begin{align*}
			\frac{\C_\epsilon[\ph](x',x'')-\ph(x',x'')}{\epsilon^2}
			&=
			\frac{
				\B_\epsilon^p[\ph](x',x'')-\ph(x',x'')
				+
				\B_\epsilon^q[\ph](x',x'')-\ph(x',x'')
			}
			{2\epsilon^2}\\
			&=
			\frac{1}{2}
			\left(
			\plap^{x'}\ph(x',x'')
			+
			\qlap^{x''}\ph(x',x'')
			\right)
			+
			o_\epsilon(1).
		\end{align*}
	\end{proof}
	\begin{cor}\label{exp-c2-parabolic}
		Let $d\in \NN$, $p>2$, $x\in \Rd$, $t>0$ $R>0$ and let $\ph \in C^2(B_R(x)\times (t-R,t+R))$. Assume that \eqref{HT} and \eqref{HT-b} holds. Then, there exists $0<\beps<R$ such that for all $\epsilon\in(0,\beps)$ we have
		\begin{equation}
			\B_\epsilon^p[\ph](x,t)=\ph(x,t+\epsilon^2)+\epsilon^2 \left(\plap \ph(x,t)-\partial_t \ph(x,t)\right)+o(\epsilon^2).
		\end{equation}
	\end{cor}
	\begin{proof}
		It also follows from Theorem~\ref{exp-c2}, since
		\begin{align*}
			\frac{\B_\epsilon^p[\ph](x,t)-\ph(x,t+\epsilon^2)}{\epsilon^2}&=\frac{\B_\epsilon^p[\ph](x,t)-\ph(x,t)-\left(\ph(x,t+\epsilon^2)-\ph(x,t)\right)}{\epsilon^2}\\
			&=\plap \ph(x,t)-\partial_t \ph(x,t)+o_\epsilon(1).
		\end{align*}
	\end{proof}
	
	Under stronger regularity assumptions, namely $\ph\in C^3$, one can obtain a more refined version of the preceding asymptotic expansions. In particular, the error terms in these expansions can be explicitly estimated in terms of $\epsilon$.
	
	\begin{thm}\label{exp-c3}
		Let $d\in\NN$, $p>2$, $x\in\Rd$, $R>0$, and let $\ph\in C^3(B_R(x))$. Assume that~\eqref{m-M} and~\eqref{m-M-b} hold. Then, there exists $0<\beps<R$ such that, for all $\epsilon\in(0,\beps)$,
		\begin{equation}
			\B_\epsilon^p[\ph](x)
			=
			\ph(x)
			+
			\epsilon^2\plap\ph(x)
			+
			\epsilon^2 E_p(\epsilon,\ph,x).
		\end{equation}
		Moreover, if $k_4$, $k_5$, $k_6$ and $k_7$ are positive constants depending only on $d$ and $p$, then the following estimates hold for the error term:
		\begin{enumerate}[label=(\alph*)]
			\item If $\gr\ph(x)\neq0$,
			\begin{equation}\label{error-B-grad}
				\begin{split}
					&\left|E_p(\epsilon,\ph,x)\right|
					\leq k_4 \Bigg(
					\epsilon^{\frac{2(p-1)}{p-2}}
					+\epsilon^2
					\|D^2\ph\|_{\LL^\infty(B_{\beps}(x))}^{\frac{1}{p-1}}
					\|\gr \ph\|_{\LL^\infty(B_{\beps}(x))}^{\frac{p-2}{p-1}}
					\\
					&\quad
					+\|D^2\ph\|_{\LL^\infty(B_{\beps}(x))}^{\frac{1}{p-1}}
					\|\gr\ph\|_{\LL^\infty(B_{\beps}(x))}^{\frac{p-2}{p-1}} 
					\Bigg(
					\|D^2\ph\|_{\LL^\infty(B_{\beps}(x))}^{\frac{p-2}{p-1}}
					\|\gr\ph\|_{\LL^\infty(B_{\beps}(x))}^{\frac{(p-2)^2}{p-1}}
					-\epsilon^{-\frac{1}{2}\min\left[
						\frac{p-2}{p-1},
						\, \frac{2}{3(p-1)}
						\right]}
					\Bigg)_+
					\Bigg)
					\\
					&\quad
					+ k_5
					\epsilon^{1-\frac{2}{p}}
					\|D^2\ph\|_{\LL^\infty(B_{\beps}(x))}
					+ k_6
					\epsilon^{\frac{1}{3p-4}}
					\Bigg(
					\|\gr\ph\|_{\LL^\infty(B_{\beps}(x))}
					+\|D^3\ph\|_{\LL^\infty(B_{\beps}(x))}
					+\frac{\left|D^2\ph(x)\right|^2}
					{\left|\gr \ph(x)\right|}
					\Bigg).
				\end{split}
			\end{equation}
			
			\item If we allow the gradient to be zero or not we have,
			\begin{equation}\label{error-B-gr-not}
				\begin{split}
					&\left|E_p(\epsilon,\ph,x)\right|
					\leq k_4 \Bigg(
					\epsilon^{\frac{2(p-1)}{p-2}}
					+\epsilon^2
					\|D^2\ph\|_{\LL^\infty(B_{\beps}(x))}^{\frac{1}{p-1}}
					\|\gr \ph\|_{\LL^\infty(B_{\beps}(x))}^{\frac{p-2}{p-1}}
					\\
					&\quad
					+\|D^2\ph\|_{\LL^\infty(B_{\beps}(x))}^{\frac{1}{p-1}}
					\|\gr\ph\|_{\LL^\infty(B_{\beps}(x))}^{\frac{p-2}{p-1}} 
					\Bigg(
					\|D^2\ph\|_{\LL^\infty(B_{\beps}(x))}^{\frac{p-2}{p-1}}
					\|\gr\ph\|_{\LL^\infty(B_{\beps}(x))}^{\frac{(p-2)^2}{p-1}}
					-\epsilon^{-\frac{1}{2}\min\left[
						\frac{p-2}{p-1},
						\, \frac{2}{3(p-1)}
						\right]}
					\Bigg)_+
					\Bigg)
					\\
					&\quad
					+ k_7
					\Bigg(\epsilon^{-\frac{2}{3}}
					\|D^2\ph\|_{\LL^\infty(B_{\beps}(x))}^{\frac{1}{p-1}}\left|\gr \ph(x)\right|^{\frac{p-2}{p-1}}+\epsilon^{1-\frac{2}{p}}\|D^2\ph\|_{\LL^\infty(B_{\beps}(x))}+\epsilon^{\frac{1}{3p-4}}\left|\gr \ph(x)\right|
					\Bigg).
				\end{split}
			\end{equation}
		\end{enumerate}
	\end{thm}
	\begin{proof}
		Recall that the error term arising from the operator $\A_\epsilon$ is denoted by $E_1(\epsilon,\ph,x)$; see Theorem~\ref{thm.asexp}. Repeating the computations from Theorem~\ref{exp-c2}, we obtain
		\begin{align*}
			\frac{\B_\epsilon^p[\ph](x)-\ph(x)}{\epsilon^2}&= \sup_{b\in[m_b,M_b]}\left\{(p-1)b \left(\left(\plap \ph(x)\right)^{\frac{1}{p-1}}+E_1(\epsilon\sqrt{2(p-1)b},\ph,x)\right)-(p-2)b^{\frac{p-1}{p-2}}\right\}\\
			&\quad +\inf_{b\in[m_b,M_b]}\left\{(p-1)b \left(\left(\plap \ph(x)\right)^{\frac{1}{p-1}}+E_1(\epsilon\sqrt{2(p-1)b},\ph,x)\right)+(p-2)b^{\frac{p-1}{p-2}}\right\}\\
			&= \sup_{b>0}\left\{(p-1)b \left(\plap \ph(x)\right)^{\frac{1}{p-1}}-(p-2)b^{\frac{p-1}{p-2}}\right\}\\
			&\quad +\inf_{b>0}\left\{(p-1)b \left(\plap \ph(x)\right)^{\frac{1}{p-1}}+(p-2)b^{\frac{p-1}{p-2}}\right\}\\
			&\quad E_p(\epsilon,\ph,x),
		\end{align*}
		where
		\begin{align*}
			\left|E(\epsilon,\ph,x)\right|&\leq 
			2(p-1)\left(m_b(\epsilon)^{\frac{p-1}{p-2}} + m_b(\epsilon) \left|\plap \ph(x)\right|^{\frac{1}{p-1}}+\left|\plap \ph(x)\right|^{\frac{1}{p-1}}\left(\left|\plap \ph(x)\right|^{\frac{p-2}{p-1}}-M_b(\epsilon)\right)_+\right)\\
			&\quad +2 (p-1) M_b(\epsilon)\left|E_1(\epsilon\sqrt{2(p-1)M_b(\epsilon)},\ph,x)\right|.
		\end{align*}
		It remains to determine the powers of $\epsilon$. At this point, the choice of $M_b(\epsilon)$ plays a crucial role.
		
		In the first term in expression above, we only have to take $m_b(\epsilon)=\epsilon^2$ and $M_b(\epsilon)=\epsilon^{-\frac{1}{2}\min\left[\frac{p-2}{p-1},\, \frac{2}{3(p-1)}\right]}$ in order to obtain the first part of inequalities \eqref{error-B-grad} and \eqref{error-B-gr-not}.
		
		The most important part of this proof is bounding the second term. Here, we use the bound \eqref{E1} (when $\gr \ph(x)\not=0$) or \eqref{E1-gr-not} with the term $\epsilon \sqrt{2(p-1) M_b(\epsilon)}$.
		
		First, we complete the proof when $\gr \ph(x)\not=0$. On the one hand, since
		$$
		\min\left[\frac{p-2}{p-1},\, \frac{2}{3(p-1)}\right]\leq \frac{p-2}{p-1},
		$$
		we can estimate the term involving $k_1 \|D^2\ph\|_{\LL^\infty(B_{\beps}(x))}$ in $E_1$ as follows:
		\begin{equation}\label{ec1}
			\begin{split}
				(2(p-1))^{2-\frac{2}{p}}k_1 &\left(\epsilon^{1-\frac{1}{4}\min\left[\frac{p-2}{p-1},\, \frac{2}{3(p-1)}\right]}\right)^{2-\frac{4}{p}}\epsilon^{-\frac{1}{2}\min\left[\frac{p-2}{p-1},\, \frac{2}{3(p-1)}\right]}\|D^2\ph\|_{\LL^\infty(B_{\beps}(x))}\\
				&\leq k_5(d,p) \epsilon^{\frac{4(p-1)-(p-2)}{4(p-1)}2\frac{p-2}{p}-\frac{p-2}{2(p-1)}}\|D^2\ph\|_{\LL^\infty(B_{\beps}(x))}\\
				&= k_5(d,p) \epsilon^{1-\frac{2}{p}}\|D^2\ph\|_{\LL^\infty(B_{\beps}(x))}.
			\end{split}
		\end{equation}
		
		On the other hand, since
		\[
		\min\left[\frac{p-2}{p-1},\, \frac{2}{3(p-1)}\right]\leq \frac{2}{3(p-1)},
		\]
		we can estimate the term involving $k_2 \left(\|\gr\ph\|+\|D^3\ph\|+\frac{\left|D^2\ph(x)\right|^2}{\left|\gr \ph(x)\right|}\right)$ in $E_1$ as follows:
		
		\begin{equation}\label{ec2}
			\begin{split}
				(2(p-1))^{\frac{3(p-1)}{3p-4}}k_2 &\left(\epsilon^{1-\frac{1}{4}\min\left[\frac{p-2}{p-1},\, \frac{2}{3(p-1)}\right]}\right)^{\frac{2}{3p-4}}\epsilon^{-\frac{1}{2}\min\left[\frac{p-2}{p-1},\, \frac{2}{3(p-1)}\right]}\left(\|\gr\ph\|+\|D^3\ph\|+\frac{\left|D^2\ph(x)\right|^2}{\left|\gr \ph(x)\right|}\right)\\
				&\leq k_6(d,p)\epsilon^{\frac{6(p-1)-1}{6(p-1)}\frac{2}{3p-4}-\frac{1}{3(p-1)}}\left(\|\gr\ph\|+\|D^3\ph\|+\frac{\left|D^2\ph(x)\right|^2}{\left|\gr \ph(x)\right|}\right)\\
				&= k_6(d,p)\epsilon^{\frac{1}{3p-4}}\left(\|\gr\ph\|+\|D^3\ph\|+\frac{\left|D^2\ph(x)\right|^2}{\left|\gr \ph(x)\right|}\right).
			\end{split}
		\end{equation}
		Combining estimates~\eqref{ec1} and~\eqref{ec2}, we conclude~\eqref{error-B-grad}.
		
		For the estimate in the region where the gradient may vanish, we use~\eqref{E1-gr-not}. Observe that the same powers of $\epsilon$ arise in this case. Thus, an analogous computation yields~\eqref{error-B-gr-not}.
	\end{proof}
	
	This last result also implies an asymptotic expansion for $C^3$ functions in the framework of equations \eqref{eq_p-q} and \eqref{eq-parabolic}.
	
	\begin{cor}\label{exp-c3-p_q}
		Let $d\geq 2$, $p>q>2$, $R>0$, $x=(x',x'')\in \Rd$ with $x'=(x_1,\cdots x_n)$ and $x''=(x_{n+1},\cdots x_d)$. Let $\ph \in C^3(B_R(x))$ and assume \eqref{m-M} and \eqref{m-M-b} for both $p$ and $q$. Then, there exists $0<\beps<R$ such that for all $\epsilon\in(0,\beps)$ we have
		\begin{equation}
			\C_\epsilon[\ph](x',x'')=\ph(x',x'')+\frac{\epsilon^2}{2}\left(\plap^{x'}\ph(x',x'')+\qlap^{x''}\ph(x',x'')\right)+\epsilon^2\left(E_p(\epsilon,\ph,x)+E_q(\epsilon,\ph,x)\right).
		\end{equation}
		Here, $E_p(\epsilon,\ph,x)$ and $E_q(\epsilon,\ph,x)$ have the bounds of \eqref{error-B-grad} if $\gr \ph\not=0$ or \eqref{error-B-gr-not} in any case.
	\end{cor}
	\begin{proof}
		This result follows by Theorem \ref{exp-c3},
		\begin{align*}
			\frac{\C_\epsilon[\ph](x',x'')-\ph(x',x'')}{\epsilon^2}&=\frac{\B_\epsilon^p[\ph](x',x'')-\ph(x',x'')+\B_\epsilon^q[\ph](x',x'')-\ph(x',x'')}{2\epsilon^2}\\
			&= \frac{1}{2}\left(\plap^{x'}\ph(x',x'')+\qlap^{x''}\ph(x',x'')+E_p(\epsilon,\ph,x)+E_q(\epsilon,\ph,x)\right).
		\end{align*}
	\end{proof}
	
	\begin{cor}\label{exp-c3-parabolic}
		Let $d\in \NN$, $p>2$, $x\in \Rd$, $t>0$, $R>0$ and let $\ph \in C^3(B_R(x)\times (t-R,t+R))$. Assume that \eqref{m-M} and \eqref{m-M-b} holds. Then, there exists $0<\beps<R$ such that for all $\epsilon\in(0,\beps)$ we have
		\begin{equation}
			\B_\epsilon^p[\ph](x,t)=\ph(x,t+\epsilon^2)+\epsilon^2 \left(\plap \ph(x,t)-\partial_t \ph(x,t)\right)+\epsilon^2 E(\epsilon,\ph,x,t).
		\end{equation}
		Where
		\begin{equation}
			\left|E(\epsilon,\ph,x,t)\right|\leq \left|E_p(\epsilon,\ph,x,t)\right|+ \frac{\epsilon^2}{2}\|\partial_{tt}\ph\|_{\LL^\infty(t-\beps,t+\beps)}.
		\end{equation} 
		Remember that we can estimate $E_p(\epsilon,\ph,x)$ as in \eqref{error-B-grad} if $\gr\ph\not=0$ or \eqref{error-B-gr-not} in any case.
	\end{cor}
	\begin{proof}
		It also follows by Theorem \ref{exp-c3},
		\begin{align*}
			\frac{\B_\epsilon^p[\ph](x,t)-\ph(x,t+\epsilon^2)}{\epsilon^2}&=\frac{\B_\epsilon^p[\ph](x,t)-\ph(x,t)}{\epsilon^2}-\frac{\ph(x,t+\epsilon^2)-\ph(x,t)}{\epsilon^2}\\
			&=\plap \ph(x)+E_p(\epsilon,\ph,x)-\partial_{t}\ph(x,t)+R(\epsilon,\ph,x,t).
		\end{align*}
		Where by Taylor we have
		$$\left|R(\epsilon,\ph,x,t)\right|\leq \frac{\epsilon^2}{2}\|\partial_{tt}\ph\|_{\LL^\infty(t-\beps,t+\beps)}.$$
	\end{proof}
	
	As explained in Section \ref{sec2}, it is convenient to analyze the asymptotic expansions of the operators $\widehat{\A_\epsilon^p}$ and $\widehat{\B_\epsilon^p}$ defined in \eqref{operator-A-promedio} and \eqref{operator-B-promedio}, respectively. Notice that the only difference from the previous setting is that we approximate the function at a given point by its average over a very small neighborhood of that point. Thus, the previous arguments can be reproduced step by step. The only novelty is the appearance of an additional term in the estimate of the error term; see \eqref{error-A-promedio} and \eqref{error-B-promedio}. This additional term does not affect any of the other results. 
	
	Since we will use these operators for the elliptic DPP, we state the result only for this setting, namely for the mean operator $\widehat{C_\epsilon}$.
	
	\begin{thm}\label{exp-c3-promedio}
		Let the assumptions of Corollary \ref{exp-c3-p_q} hold.  Then, there exists $0<\beps<R$ such that, for all $\epsilon\in(0,\beps)$,
		\begin{equation}\label{exp-promedio}
			\widehat{\C_\epsilon}[\ph](x',x'')
			=
			\ph(x',x'')
			+\frac{\epsilon^2}{2}\left(
			\plap^{x'}\ph(x',x'')
			+\qlap^{x''}\ph(x',x'')
			\right)
			+\epsilon^2\left(
			\widehat{E_p}(\epsilon,\ph,x)
			+\widehat{E_q}(\epsilon,\ph,x)
			\right).
		\end{equation}
		Here, $\widehat{E_p}(\epsilon,\ph,x)$ and $\widehat{E_q}(\epsilon,\ph,x)$ can be estimate by the bounds \eqref{error-B-grad} and \eqref{error-B-gr-not} plus \eqref{error-B-promedio}.
	\end{thm}
	
	As a consequence of the preceding asymptotic expansions, we obtain a characterization of smooth subsolutions and supersolutions of differents equations which include the $p$-Laplacian.
	
	\begin{lem}\label{characterizacion-1}
			The following statements are true:
		\begin{enumerate}[label=(\roman*)]
			\item Let the assumptions of Theorem \ref{exp-c3} hold. Then, the following inequalities are equivalent:
			\begin{equation}
				\plap\ph(x)\leq f(x)
				\Longleftrightarrow
				\B_\epsilon^p[\ph](x)
				\leq
				\ph(x)+\epsilon^2f(x)+o(\epsilon^2)
				\qquad\mbox{as }\epsilon\to 0^+.
			\end{equation}
			\begin{equation}
				\plap\ph(x)\geq f(x)
				\Longleftrightarrow
				\B_\epsilon^p[\ph](x)
				\geq
				\ph(x)+\epsilon^2f(x)+o(\epsilon^2)
				\qquad\mbox{as }\epsilon\to 0^+.
			\end{equation}
			
			\item Let the assumptions of Corollary \ref{exp-c3-parabolic} hold. Then, the following inequalities are equivalent:
			\begin{equation}
				\plap\ph(x,t)\leq\partial_t\ph(x,t)
				\Longleftrightarrow
				\B_\epsilon^p[\ph](x,t)
				\leq\ph(x,t+\epsilon^2)+o(\epsilon^2)
				\qquad\mbox{as }\epsilon\to0^+.
			\end{equation}
			\begin{equation}
				\plap\ph(x,t)\geq\partial_t\ph(x,t)
				\Longleftrightarrow
				\B_\epsilon^p[\ph](x,t)
				\geq\ph(x,t+\epsilon^2)+o(\epsilon^2)
				\qquad\mbox{as }\epsilon\to0^+.
			\end{equation}
			
			\item et the assumptions of Theorem \ref{exp-c3-promedio} hold.  Then, the following inequalities are equivalent:
			\begin{equation}
				\plap^{x'}\ph(x',x'')+\qlap^{x''}\ph(x',x'')\leq f(x)
				\Longleftrightarrow
				\widehat{\C_\epsilon}[\ph](x)
				\leq
				\ph(x)+\frac{\epsilon^2}{2}f(x)+o(\epsilon^2)
				\qquad\mbox{as }\epsilon\to 0^+.
			\end{equation}
			\begin{equation}
				\plap^{x'}\ph(x',x'')+\qlap^{x''}\ph(x',x'')\geq f(x)
				\Longleftrightarrow
				\widehat{\C_\epsilon}[\ph](x)
				\geq
				\ph(x)+\frac{\epsilon^2}{2}f(x)+o(\epsilon^2)
				\qquad\mbox{as }\epsilon\to 0^+.
			\end{equation}
		\end{enumerate}
	\end{lem}
	\begin{proof}
		The result follows from applying the asymptotic expansions for $C^3$ functions in each case.
	\end{proof}
	 
	To conclude this section, we provide an asymptotic mean value characterization of viscosity solutions for some equations which involve the $p$-Laplacian. Note that no a priori regularity assumptions on $u$ are imposed beyond continuity.
	
	\begin{lem}\label{characterizacion-2}
				The following statements are true:
		\begin{enumerate}[label=(\roman*)]
			\item Let the assumptions of Theorem \ref{exp-c3} hold. Then, $u$ is a viscosity solution of
			$$\plap u(x)=f(x)\qquad \mbox{for all }x\in\Omega,$$
			if and only if $u$ is a viscosity solution of the following AMVF
			$$u(x)=\B_\epsilon^p[u](x)-\epsilon^2 f(x) + o(\epsilon^2)\qquad \mbox{in }\Omega,\quad \mbox{as }\epsilon\to 0^+.$$
			
			\item Let the assumptions of Corollary \ref{exp-c3-parabolic} hold. Then, $u$ is a viscosity solution of
			$$\partial_tu(x,t)=\plap u(x,t)
			\qquad\mbox{for all }(x,t)\in\Omega\times(0,\infty),$$
			if and only if $u$ is a viscosity solution of the following AMVF
			$$u(x,t+\epsilon^2)=\B_\epsilon^p[u](x,t)+o(\epsilon^2)
			\qquad\mbox{in }\Omega\times(0,\infty),\quad \mbox{as }\epsilon\to0^+.$$
			
			\item Let the assumptions of Theorem \ref{exp-c3-promedio} hold.  Then, $u$ is a viscosity solution of
			$$\plap^{x'}\ph(x',x'')+\qlap^{x''}\ph(x',x'')= f(x)\qquad\mbox{for all }x\in \Omega,$$
			if and only if $u$ is a viscosity solution of the following AMVF
			$$u(x)=\widehat{\C_\epsilon}[u](x)-\frac{\epsilon^2}{2}f(x)+o(\epsilon^2)
			\qquad\mbox{in }\Omega,\quad\mbox{as }\epsilon\to 0^+.$$
		\end{enumerate}
	\end{lem}
	\begin{proof}
		We only provide a sketch of the proof for one of the cases, as the other cases follow by the same argument. Let $\ph\in C^2(B_R(x))$ be a test function touching $u$ from below at $x$. Then, since $u$ is a viscosity solution of equation \ref{eq_p-q} we have
		$$
		\plap^{x'}\ph(x',x'')+\qlap^{x''}\ph(x',x'')\leq f(x),
		$$
		and, by Lemma~\ref{characterizacion-1}, we obtain that
		$$
		\widehat{\C_\epsilon}[\ph](x)
		\leq
		\ph(x)+\frac{\epsilon^2}{2}f(x)+o(\epsilon^2).
		$$
		Therefore, $u$ is a viscosity supersolution of the AMVF. If the test function $\ph$ touches $u$ from above, we conclude analogously that $u$ is a viscosity subsolution of the AMVF and, hence, a viscosity solution. 
		
		The converse implication is proved in the same way. Let $\ph\in C^2(B_R(x))$ be a test function touching $u$ from below at $x$. Then
		$$
		\widehat{\C_\epsilon}[\ph](x)
		\leq
		\ph(x)+\frac{\epsilon^2}{2}f(x)+o(\epsilon^2),
		$$
		and Lemma~\ref{characterizacion-1} yields
		$$
		\plap^{x'}\ph(x',x'')+\qlap^{x''}\ph(x',x'')\leq f(x).
		$$
		The case in which $\ph$ touches $u$ from above is analogous.
	\end{proof}
	
	
	\section{Dynamic programming principle for the anisotropic $p-q$ Laplacian}\label{sec4}
	
	Once the operator $\widehat{\C_\epsilon}$ has been defined and shown to approximate equation~\eqref{eq_p-q} for fixed $\epsilon>0$, it induces the Dynamic Programming Principle \eqref{DPP-aniso}. For the reader's convenience, we recall this DPP below.
	\begin{equation}\label{DPP}
		\left\{\begin{array}{rlll}
			\ue(x)&=&\frac{1}{2}\left(\widehat{\B_\epsilon^p}[\ue](x',x'')+\widehat{\B_\epsilon^q}[\ue](x',x'')\right)-\frac{\epsilon^2}{2}f(x), & \mbox{in } \Omega,\\ 
			\ue(x)&=&g(x), & \mbox{in } \partial \Omega.
		\end{array}\right.
	\end{equation}
	Here, $\Omega\subset \Rd$ is a bounded regular domain, $f\in C(\overline{\Omega})$, $g\in C_b(\Rd\setminus\Omega)$, and $x=(x',x'')\in \Rd$, where $x'=(x_1,\ldots,x_n)$ and $x''=(x_{n+1},\ldots,x_d)$.
	
	The structure of this section is divided into two clearly distinguished parts. First, we establish the existence and uniqueness of solutions to~\eqref{DPP}. We then show that this unique solution converges uniformly to the viscosity solution of \eqref{anosotropic-problem} with the same prescribed data.

	\subsection{Existence and Uniqueness}\label{sec4.1}
	
	We begin with the uniqueness part, which follows from the following comparison result.
	
	\begin{thm}\label{comparison}
		Let $d\geq 2$, $p>q>2$, $\epsilon\in (0,1)$, and let $\Omega\subset \Rd$ be a bounded domain. Assume that~\eqref{HT} and~\eqref{HT-b} hold for both $p$ and $q$. Let $g_1,g_2\in C_b(\Rd\setminus\Omega)$ and $f_1,f_2\in C(\overline{\Omega})$. Moreover, let $\ua,\ub:\Rd\longrightarrow \RR$ be two bounded measurable functions satisfying
		$$\left\{\begin{array}{rlll}
			\ua(x)&\geq&\frac{1}{2}\left(\widehat{\B_\epsilon^p}[\ua](x',x'')+\widehat{\B_\epsilon^q}[\ua](x',x'')\right)-\frac{\epsilon^2}{2}f_1(x), & \mbox{in } \Omega,\\ 
			\ua(x)&\geq&g_1(x), & \mbox{in } \partial \Omega.
		\end{array}\right.$$
		and
		$$\left\{\begin{array}{rlll}
			\ub(x)&\leq&\frac{1}{2}\left(\widehat{\B_\epsilon^p}[\ub](x',x'')+\widehat{\B_\epsilon^q}[\ub](x',x'')\right)-\frac{\epsilon^2}{2}f_2(x), & \mbox{in } \Omega,\\ 
			\ub(x)&\leq&g_2(x), & \mbox{in } \partial \Omega.
		\end{array}\right.$$
		If $f_1\leq f_2$ and $g_1\geq g_2$, then
		$$\ua\geq\ub\quad \mbox{in }\Rd.$$
	\end{thm}
	
	\begin{rem}\label{remark-boundary}{\rm
			In order to apply this comparison principle, it is enough to compare the exterior data in a neighborhood of the domain. More precisely, it suffices to impose the condition $g_1\geq g_2$ in the set
			$$
			\partial\Omega_\epsilon := \Omega_\epsilon \setminus \Omega\qquad \mbox{with } \Omega_\epsilon:=(\Omega + B_{\epsilon^\gamma}(0)),
			$$
			where $\epsilon^\gamma=\max\left\{2(p-1)M_b(\epsilon)\epsilon^2 M(\epsilon)^{1-\alpha}, \, \sqrt{2(p-1)M_b(\epsilon)}\epsilon \, m(\epsilon)^{-\alpha/2},\, (2(p-1)M_b(\epsilon))^{3/2}\epsilon^3\right\}$.}
	\end{rem}
	
	Before proceeding with the proof, we require a couple of technical lemmas concerning differences of maxima and their application to the operator $\widehat{\A_\epsilon^p}$.
	
	\begin{lem}\label{max-ineq}
		Let $A_1$, $A_2$, $B_1$, and $B_2$ be arbitrary real numbers. Then,
		$$
		\max\left[A_1,\,A_2\right]-\max\left[B_1,\,B_2\right]
		\leq
		\max\left[A_1-B_1,\,A_2-B_2\right].
		$$
	\end{lem}
	
	\begin{lem}
		Let $d\in \NN$, $p>2$, and $\epsilon\in(0,1)$. Assume that \eqref{HT} hold, and let $\ua,\ub$ be two bounded measurable functions such that $\ub\leq\ua$ in $\Rd$. Then, for any $x_0\in\Rd$, there exists a positive constant $\eta$ such that
		$$
		\widehat{\A_\epsilon^p}[\ub](x_0)-\widehat{\A_\epsilon^p}[\ua](x_0)
		\leq
		\eta\fint_{B_{\epsilon^8}(x_0)}\left(\ub(y)-\ua(y)\right)\,\dy.
		$$
	\end{lem}
	
	The proofs can be found in Appendix~\ref{apenA}.
	
	\begin{proof}[Proof of Theorem \ref{comparison}]
		Since $\ua \geq g_1 \geq g_2 \geq \ub$ in $\Rd\setminus \Omega$, it suffices to prove this inequality in $\Omega$.
		
		Assume, by contradiction, that there exists a point $x_0=(x_0',x_0'')\in \Omega$ such that
		$$\sup_{x\in \Rd}\left\{\ub(x)-\ua(x)\right\}=\ub(x_0)-\ua(x_0)=S>0.$$
		We then define $\vb=\ub-S$. It can be easily verified that this function $\vb$ is a subsolution of~\eqref{DPP} with data $f_2$ and $g_2$. 
		
		Moreover, $\vb(x_0)=\ua(x_0)$ and $\vb(x)\leq \ua(x)$ at every other point by construction. Therefore, using the preceding lemmas, we obtain
		\begin{align*}
			0&= \vb(x_0)-\ua(x_0)\leq \widehat{\C_\epsilon}[\vb](x_0)-\widehat{\C_\epsilon}[\ua](x_0)+\frac{\epsilon^2}{2}\left(f_2(x_0)-f_1(x_0)\right)\\
			&\leq \sup_{b\in[m_b,M_b]}\left\{\widehat{\A_{\epsilon\sqrt{2(p-1)b}}^p}[\vb](x_0',x_0'')-\widehat{\A_{\epsilon\sqrt{2(p-1)b}}^p}[\ua](x_0',x_0'')\right\}\\
			&\quad + \sup_{b\in[m_b,M_b]}\left\{\widehat{\A_{\epsilon\sqrt{2(q-1)b}}^q}[\vb](x_0',x_0'')-\widehat{\A_{\epsilon\sqrt{2(q-1)b}}^q}[\ua](x_0',x_0'')\right\}\\
			&\leq \eta_p \fint_{B_{(2(p-1)m_b)^{4}\epsilon^8}(x_0)}\vb(y)-\ua(y) \, \dy+ \eta_q \fint_{B_{(2(q-1)m_b)^{4}\epsilon^8}(x_0)} \vb(y)-\ua(y)\, \dy.
		\end{align*}
		
		Since $\vb\leq \ua$, we conclude that there exists a sufficiently small radius $\gamma$ such that
		$\vb\equiv \ua$ in $B_\gamma(x_0)$.
		Repeating this argument at any point in the ball $B_\gamma(x_0)$ where $\vb$ and $\ua$ coincide, we deduce the existence of a point $x_1$ and a radius $\gamma$ such that the two functions also coincide in $B_\gamma(x_1)$. Since the domain is bounded, after finitely many steps we eventually find a point $x_k\in \mathbb{R}^d\setminus\Omega$ such that $\vb(x_k)=\ua(x_k)$, which leads us to the following contradiction:
		$$\ub(x_k)\leq g_2(x_k)\leq g_1(x_k)\leq \ua(x_k)=\vb(x_k)\leq \ub(x_k)-S.$$
	\end{proof}
	
	The uniqueness of solutions to~\eqref{DPP} follows by taking $f_1=f_2=f$ and $g_1=g_2=g$. To establish existence, we shall employ an iterative scheme for the equation. In order to carry out this argument, the existence of subsolutions and supersolutions to the problem is required; this will be proved below. Such functions can be constructed as in~\cite{delTeso_Rossi_2024}, with only minor modifications.
	
	\begin{lem}
		Let $d\geq 2$, $p>q>2$, $\epsilon\in(0,1)$, and let $\Omega\subset\Rd$ be a bounded domain. Assume that \eqref{HT} and \eqref{HT-b} hold for both $p$ and $q$. Let $f\in C(\overline{\Omega})$ and $g\in C_b(\Rd\setminus\Omega)$. Then there exist two measurable bounded functions $\ua$ and $\ub$ such that
		$$
		\left\{
		\begin{array}{rlll}
			\ub(x)&\leq&\frac{1}{2}\left(\widehat{\B_\epsilon^p}[\ub](x',x'')+\widehat{\B_\epsilon^q}[\ub](x',x'')\right)-\frac{\epsilon^2}{2}f(x), & \mbox{in } \Omega,\\ 
			\ub(x)&\leq&g(x), & \mbox{in } \partial \Omega.
		\end{array}
		\right.
		$$
		and
		$$
		\left\{
		\begin{array}{rlll}
			\ua(x)&\geq&\frac{1}{2}\left(\widehat{\B_\epsilon^p}[\ua](x',x'')+\widehat{\B_\epsilon^q}[\ua](x',x'')\right)-\frac{\epsilon^2}{2}f(x), & \mbox{in } \Omega,\\ 
			\ua(x)&\geq&g(x), & \mbox{in } \partial \Omega.
		\end{array}
		\right.
		$$
	\end{lem}
	
	\begin{proof}
		Fix $\epsilon\in(0,1)$. Since $\Omega$ is bounded, we choose $R>0$ sufficiently large such that $\Omega_\epsilon\subset B_R(0)$. We define the following bounded measurable function:
		$$
		\ub(x)=
		\left\{
		\begin{array}{ll}
			\ee^{Lx_1}-T & \mbox{if } x\in B_R(0),\\ 
			-\|g\|_{\LL^\infty(\Rd\setminus\Omega)} & \mbox{if } x\in\Rd\setminus B_R(0),
		\end{array}
		\right.
		$$
		where $L,T>0$ are two constants to be determined.
		
		Using
		$$
		\widehat{\A^p_\epsilon}[\ph](x)\geq\min\left[
		\inf_{c\in[m_p,M_p]}
		\left\{
		\alpha_p \inf_{B_{2\epsilon^2 c^{1-\alpha_p}}(x)}\ph
		+(1-\alpha_p) M_{\sqrt{2}\epsilon c^{-\alpha_p/2}}[\ph](x)
		\right\},
		\fint_{B_{\epsilon^8}(x)}\ph(y)\,\dy
		\right],
		$$
		for any function $\ph$, it follows from \cite[Lemma 4.3]{delTeso_Rossi_2024} that, for any $b\in[m_b,M_b]$,
		\begin{align*}
			\frac{\widehat{\A_{\epsilon\sqrt{2(p-1)b}}^p}[\ub](x)-\ub(x)}{\epsilon^2}
			&\geq \ee^{Lx_1}
			\min\Bigg[
			-4(p-1)M_b(\epsilon)\alpha_pM_p(\epsilon)^{1-\alpha_p}L\\
			&\qquad\qquad
			+8(p+d)(p-1)m_b(\epsilon)K_d(1-\alpha_p)
			M_p(\epsilon)^{-\alpha_p}L^2,\\
			&\qquad\qquad
			\left(8(p+d)(p-1)m_b(\epsilon)\right)^3
			K_d\epsilon^4L^2
			\Bigg]\\
			&:=A_\epsilon^p(L).
		\end{align*}
		Therefore, for $L$ sufficiently large, this quantity is positive. In fact, $A_\epsilon^p(L)$ is increasing in $L$.
		
		Moreover, 
		\begin{align*}
			&\frac{\widehat{\B_\epsilon^p}[\ub](x)+\widehat{\B_\epsilon^q}[\ub](x)-2\ub(x)}{\epsilon^2}-f(x)\\
			&\geq
			\frac{1}{2}\sup_{b\in[m_b,M_b]}
			\left\{
			\frac{\widehat{\A_{\epsilon\sqrt{2(p-1)b}}^p}[\ub](x)-\ub(x)}{\epsilon^2}
			-2(p-2)b^{\frac{p-1}{p-2}}
			\right\}\\
			&\qquad
			+\frac{1}{2}\sup_{b\in[m_b,M_b]}
			\left\{
			\frac{\widehat{\A_{\epsilon\sqrt{2(q-1)b}}^q}[\ub](x)-\ub(x)}{\epsilon^2}
			-2(q-2)b^{\frac{q-1}{q-2}}
			\right\}
			-\|f\|_{\LL^\infty(\Omega)}\\
			&\geq
			\frac{1}{2}
			\left(
			A_\epsilon^p(L)
			-2(p-2)m_b(\epsilon)^{\frac{p-1}{p-2}}
			-\|f\|_{\LL^\infty(\Omega)}
			\right)\\
			&\qquad
			+\frac{1}{2}
			\left(
			A_\epsilon^q(L)
			-2(q-2)m_b(\epsilon)^{\frac{q-1}{q-2}}
			-\|f\|_{\LL^\infty(\Omega)}
			\right)\\
			&\geq 0.
		\end{align*}
		Here, in the last inequality, we have fixed $L$ sufficiently large. Thus, for all $x\in\Omega$, we have
		$$
		\ub(x)\leq
		\frac{1}{2}\left(
		\widehat{\B_\epsilon^p}[\ub](x)+
		\widehat{\B_\epsilon^q}[\ub](x)
		\right)
		-\frac{\epsilon^2}{2}f(x),
		$$
		and this inequality is independent of the choice of $T$. Finally, we choose $T$ sufficiently large so that
		$$
		\ub(x)\leq-\|g\|_{\LL^\infty(\Rd\setminus\Omega)}
		\qquad\mbox{for all }x\in B_R(0).
		$$
		This implies that $\ub$ is a subsolution of \eqref{DPP}.
		
		To prove the existence of a supersolution, we repeat the above argument with $\ua=-\ub$, using
		$$
		\widehat{\A^p_\epsilon}[\ph](x)\leq
		\max\left[
		\sup_{c\in[m_p,M_p]}
		\left\{
		\alpha_p\sup_{B_{2\epsilon^2 c^{1-\alpha_p}}(x)}\ph
		+(1-\alpha_p)M_{\sqrt{2}\epsilon c^{-\alpha_p/2}}[\ph](x)
		\right\},
		\fint_{B_{\epsilon^8}(x)}\ph(y)\,\dy
		\right].
		$$
	\end{proof}
	
	Now we are ready to establish the existence result.
	\begin{thm}\label{existence}
		Let $d\geq 2$, $p>q>2$, $\epsilon\in(0,1)$, and let $\Omega\subset\Rd$ be a bounded domain. Assume that \eqref{HT} and \eqref{HT-b} hold for both $p$ and $q$. Let $f\in C(\overline{\Omega})$ and $g\in C_b(\Rd\setminus\Omega)$. Then there exists a unique measurable and bounded function $\ue$ satisfying \eqref{DPP}.
	\end{thm}
	
	\begin{proof}
		Uniqueness has already been established. We now prove existence. Let $u_0$ be a subsolution of \eqref{DPP}, namely,
		$$
		\left\{
		\begin{array}{rlll}
			u_0(x)&\leq&
			\frac{1}{2}\left(
			\widehat{\B_\epsilon^p}[u_0](x',x'')
			+\widehat{\B_\epsilon^q}[u_0](x',x'')
			\right)-\frac{\epsilon^2}{2}f(x),
			& \mbox{in } \Omega,\\ 
			u_0(x)&\leq&g(x), & \mbox{in } \partial\Omega.
		\end{array}
		\right.
		$$
		We then define the sequence of functions $\{u_k\}$ recursively by
		$$
		u_k(x)=
		\left\{
		\begin{array}{rlll}
			\frac{1}{2}\left(
			\widehat{\B_\epsilon^p}[u_{k-1}](x',x'')
			+\widehat{\B_\epsilon^q}[u_{k-1}](x',x'')
			\right)-\frac{\epsilon^2}{2}f(x),
			& \mbox{in } \Omega,\\ 
			g(x), & \mbox{in } \partial\Omega.
		\end{array}
		\right.
		$$
		We split the remainder of the proof into several steps.
		
		\noindent\textit{Step 1: The sequence $\{u_k\}$ is nondecreasing.}
		Since $u_k=g$ in $\Rd\setminus\Omega$ for every $k$, it suffices to show that the sequence is nondecreasing in $\Omega$. We proceed by induction.
		
		It is clear that $u_0\leq u_1$ in $\Omega$. Suppose that $u_{k-1}\leq u_k$. Then, for any $x\in\Omega$,
		\begin{equation*}
			u_{k+1}(x)
			= \widehat{\C_\epsilon}[u_k](x)-\frac{\epsilon^2}{2}f(x)
			\geq \widehat{\C_\epsilon}[u_{k-1}](x)-\frac{\epsilon^2}{2}f(x)
			= u_k(x).
		\end{equation*}
		Thus, $\{u_k\}$ is nondecreasing.
		
		\noindent\textit{Step 2: Each function $u_k$ is a subsolution of \eqref{DPP}.}
		Since $u_k=g$ in $\Rd\setminus\Omega$, it suffices to verify the subsolution inequality in $\Omega$. By Step 1, for any $x\in\Omega$,
		\begin{equation*}
			u_k(x)
			=\widehat{\C_\epsilon}[u_{k-1}](x)-\frac{\epsilon^2}{2}f(x)
			\leq \widehat{\C_\epsilon}[u_k](x)-\frac{\epsilon^2}{2}f(x).
		\end{equation*}
		Therefore, $u_k$ is a subsolution of \eqref{DPP}.
		
		\noindent\textit{Step 3: The sequence $\{u_k\}$ converges pointwise.}
		Let $\ua$ be a supersolution of \eqref{DPP}. Since $\{u_k\}$ is monotone nondecreasing and, by the comparison principle, bounded above by $\ua$, there exists a measurable function $u$ such that
		\begin{equation*}
			u_k\to u \qquad \mbox{pointwise in }\Rd.
		\end{equation*}
		
		\noindent\textit{Step 4: The sequence $\{u_k\}$ converges uniformly.}
		Again, it suffices to establish convergence in $\Omega$, since all the functions in the sequence are equal to $g$ outside $\Omega$. Let $m>n$. Then, for any $x=(x',x'')\in\Omega$, we have
		\begin{align*}
			\Big|u_{m+1}(x)-u_{n+1}(x)\Big|
			&= \widehat{\C_\epsilon}[u_m](x)
			-\widehat{\C_\epsilon}[u_n](x)
			\\
			&\leq \frac{1}{2}
			\sup_{b\in[m_b,M_b]}
			\Big\{
			\widehat{\A_{\epsilon\sqrt{2(p-1)b}}^p}[u_m](x',x'')
			-\widehat{\A_{\epsilon\sqrt{2(p-1)b}}^p}[u_n](x',x'')
			\Big\}
			\\
			&\quad
			+\frac{1}{2}
			\sup_{b\in[m_b,M_b]}
			\Big\{
			\widehat{\A_{\epsilon\sqrt{2(q-1)b}}^q}[u_m](x',x'')
			-\widehat{\A_{\epsilon\sqrt{2(q-1)b}}^q}[u_n](x',x'')
			\Big\}
			\\
			&\leq \frac{1}{2}
			\max\Bigg[
			\alpha_p\sup_{\Rd}|u_m-u_n|
			+(1-\alpha_p)\frac{p-2}{p+d}
			\sup_{\Rd}|u_m-u_n|
			\\
			&\quad
			+(1-\alpha_p)\frac{d+2}{p+d}
			\left|B_{(4(p-1)(p+d)M_b)^{1/2}
				\epsilon m_p^{-\alpha_p/2}}\right|^{-1}
			\int_{\Omega}|u_m(y)-u_n(y)|\,\dy,
			\\
			&\qquad
			\left|B_{(2(p-1)M_b)^{3/2}\epsilon^3}\right|^{-1}
			\int_{\Omega}|u_m(y)-u_n(y)|\,\dy
			\Bigg]
			\\
			&\quad 
			+\frac{1}{2}\max\Bigg[
			\alpha_q\sup_{\Rd}|u_m-u_n|
			+(1-\alpha_q)\frac{q-2}{q+d}
			\sup_{\Rd}|u_m-u_n|
			\\
			&\quad
			+(1-\alpha_q)\frac{d+2}{q+d}
			\left|B_{(4(q-1)(q+d)M_b)^{1/2}
				\epsilon m_q^{-\alpha_q/2}}\right|^{-1}
			\int_{\Omega}|u_m(y)-u_n(y)|\,\dy,
			\\
			&\qquad
			\left|B_{(2(q-1)M_b)^{4}\epsilon^8}\right|^{-1}
			\int_{\Omega}|u_m(y)-u_n(y)|\,\dy
			\Bigg]
			\\
			&\leq
			\Bigg(
			\frac{1}{2}
			\left(
			\alpha_p
			+(1-\alpha_p)\frac{p-2}{p+d}
			\right)
			+\frac{1}{2}
			\left(
			\alpha_q
			+(1-\alpha_q)\frac{q-2}{q+d}
			\right)
			\Bigg)
			\sup_{\Rd}|u_m-u_n|
			\\
			&\quad
			+C\|u_m-u_n\|_{\LL^1(\Omega)}.
		\end{align*}
		
		Note that the coefficient of the supremum term, which we denote by $\eta$, is strictly less than $1$. Letting $m\to\infty$ and using the arbitrariness of $x$, we obtain
		$$
		\sup_{\Rd}|u-u_{n+1}|
		\leq
		\eta\sup_{\Rd}|u-u_n|
		+C\|u-u_n\|_{\LL^1(\Omega)}.
		$$
		Since $u_n\to u$ pointwise and the sequence $(u_n)_{n\geq 1}$ is uniformly bounded, the Dominated Convergence Theorem implies that the $\LL^1$-norm of the second term converges to $0$. Consequently,
		\begin{equation*}
			\lim_{n\to\infty}\sup_{\Rd}|u-u_n|
			=
			\lim_{n\to\infty}\sup_{\Rd}|u-u_{n+1}|
			\leq
			\eta\lim_{n\to\infty}\sup_{\Rd}|u-u_n|.
		\end{equation*}
		Since $\eta<1$, this implies
		$$
		\lim_{n\to\infty}\sup_{\Rd}|u-u_n|=0,
		$$
		and hence $u_n\to u$ uniformly in $\Rd$.
		
		\noindent\textit{Step 5: The limit $u$ is a solution of \eqref{DPP}.}
		Since $u_k=g$ in $\Rd\setminus\Omega$ for every $k$, we have $u(x)=g(x)$ for every $x\in\Rd\setminus\Omega$. For $x\in\Omega$, we have
		\begin{equation*}
			u(x)
			=\lim_{k\to\infty}u_k(x)
			=\lim_{k\to\infty}
			\widehat{\C_\epsilon}[u_{k-1}](x)
			-\frac{\epsilon^2}{2}f(x)
			=\widehat{\C_\epsilon}[u](x)
			-\frac{\epsilon^2}{2}f(x).
		\end{equation*}
		Therefore, $u$ satisfies \eqref{DPP}, which completes the proof.
	\end{proof}
	
	\subsection{Convergence as $\epsilon\to 0^+$}\label{sec4.2}
	
	In this second part, we show that the solution to~\eqref{DPP} converges uniformly to the viscosity solution of the problem \eqref{anosotropic-problem}. Before proceeding, we need to construct smooth barriers for $\ue$ that are independent of $\epsilon$. To this end, we first establish the following computation.
	
	\begin{lem}\label{lemma_p-q_computation}
		Let $1\leq n<d$, $p>q>2$, and $x=(x',x'')$. For any $z\in\Rd$ and $\alpha>0$, define
		\[
		v(x)=|x-z|^{-\alpha}.
		\]
		Then the $p$-Laplacian with respect to the variables $x'$ is given by
		\begin{equation*}
			\begin{split}
				\plap^{x'}v(x)
				=\alpha^{p-1}
				\frac{|x'-z'|^{p-2}}
				{|x-z|^{(\alpha+2)(p-1)+2}}
				\Bigg(
				\alpha(p-1)|x'-z'|^2\\
				+(p-n)|x'-z'|^2
				-(n+p-2)|x''-z''|^2
				\Bigg).
			\end{split}
		\end{equation*}
		Similarly, one can compute that
		\begin{equation*}
			\begin{split}
				\qlap^{x''}v(x)
				=\alpha^{q-1}
				\frac{|x''-z''|^{q-2}}
				{|x-z|^{(\alpha+2)(q-1)+2}}
				\Bigg(
				\alpha(q-1)|x''-z''|^2\\
				+(q-(d-n))|x''-z''|^2
				-(d-n+q-2)|x'-z'|^2
				\Bigg).
			\end{split}
		\end{equation*}
	\end{lem}
	
	The computations can be found in Appendix~\ref{apenA}.
	
	The construction of such barriers is precisely where a certain degree of regularity of the boundary of the domain becomes necessary. In particular, we shall assume that the domain satisfies the uniform exterior ball condition.
	
	\begin{thm}\label{barreras}
		Let $d\geq 2$, $p>q>2$, $\epsilon\in(0,1)$, and let $\Omega\subset\Rd$ be a bounded domain satisfying the uniform exterior ball condition. Assume that \eqref{m-M} and \eqref{m-M-b} hold for both $p$ and $q$. Let $\ue$ be the solution of \eqref{DPP} with data $f\in C(\overline{\Omega})$ and $g\in C_b(\Rd\setminus\Omega)$. Then
		$$
		\|\ue\|_{\LL^\infty(\Rd)}\leq C.
		$$
		Moreover, for any $x_0\in\partial\Omega$, we have
		$$
		\limsup_{\epsilon\to0\, ,\, x\to x_0}\ue(x)\leq g(x_0)
		\qquad\mbox{and}\qquad
		\liminf_{\epsilon\to0\, ,\, x\to x_0}\ue(x)\geq g(x_0).
		$$
	\end{thm}
	
	\begin{proof}
		Given $x_0=(x_0',x_0'')\in\partial\Omega$, the uniform exterior ball condition ensures the existence of $R>0$ and $z=(z',z'')\in\Rd\setminus\Omega$ such that $\overline{B_R(z)}\cap\overline{\Omega}=\{x_0\}$.
		
		For any $\eta>0$, we define the lower barrier
		$$
		\wb(x)=\kappa\left(|x-z|^{-\alpha}-R^{-\alpha}\right)+g(x_0)-\eta,
		$$
		where $\kappa$ and $\alpha$ are positive constants to be determined. We shall show that there exists $\beps\in(0,1)$ such that, for all $\epsilon\in(0,\beps)$, the function $\wb$ is a subsolution of \eqref{DPP} and satisfies $\wb\leq\ue$.
		
		Since $\wb$ is smooth in $\Omega$, we may apply Theorem~\ref{exp-c3-promedio}. Furthermore, by Lemma~\ref{lemma_p-q_computation}, we obtain
		\begin{align*}
			\frac{\widehat{\C_\epsilon}[\wb](x)-\wb(x)}{\epsilon^2/2}
			&=\plap^{x'}\wb(x',x'')
			+\qlap^{x''}\wb(x',x'')
			+E_p(\epsilon,\wb,x)+E_q(\epsilon,\wb,x)\\
			&=\left|x-z\right|^{-(\alpha+2)(p-1)-2}\kappa^{p-1}
			\Big(
			(p-1)\alpha^p\left|x'-z'\right|^p\\
			&\qquad\qquad\qquad
			+(p-n)\alpha^{p-1}\left|x'-z'\right|^p\\
			&\qquad\qquad\qquad
			-(n+p-2)\alpha^{p-1}
			\left|x'-z'\right|^{p-2}
			\left|x''-z''\right|^2
			\Big)\\
			&\quad
			+\left|x-z\right|^{-(\alpha+2)(q-1)-2}\kappa^{q-1}
			\Big(
			(q-1)\alpha^q\left|x''-z''\right|^q\\
			&\qquad\qquad\qquad
			+(q-(d-n))\alpha^{q-1}\left|x''-z''\right|^q\\
			&\qquad\qquad\qquad
			-(d-n+q-2)\alpha^{q-1}
			\left|x''-z''\right|^{q-2}
			\left|x'-z'\right|^2
			\Big)\\
			&\quad
			+E_p(\epsilon,\wb,x)+E_q(\epsilon,\wb,x).
		\end{align*}
		
		We now observe that
		$$
		R^2\leq |x'-z'|^2+|x''-z''|^2\leq C.
		$$
		Consequently, one of the following two situations may occur. Either both terms are uniformly positive; for instance, $|x'-z'|^2,\ |x''-z''|^2\geq \frac{R^2}{5}$.
		In this case, the preceding expression is dominated by
		$$
		2(p-1)\alpha^p|x'-z'|^p
		\qquad\mbox{and}\qquad
		2(q-1)\alpha^q|x''-z''|^q,
		$$
		respectively, and hence we may choose $\alpha$ sufficiently large to ensure that the total quantity is positive.
		
		The other possibility is that one of the two terms vanishes or becomes arbitrarily small; for example, $|x''-z''|\sim0$. In this case, the other term remains uniformly positive, and the dominant contribution to the total sum is
		$$
		2(p-1)\alpha^p|x'-z'|^p.
		$$
		Hence, once again, by taking $\alpha$ sufficiently large, we ensure that the total quantity is positive.
		
		Finally, since $f$ is continuous on $\overline{\Omega}$ and therefore bounded, we may choose $\kappa$ sufficiently large so that
		$$
		\frac{\widehat{\C_\epsilon}[\wb](x)-\wb(x)}{\epsilon^2/2}
		\geq\|f\|_{\LL^\infty(\Omega)}+1.
		$$
		
		Let us now show that $\wb(x)\leq g(x)$ in $\partial\Omega_\epsilon:=(\Omega+B_{\epsilon^\gamma})\setminus\Omega,$
		where $\epsilon^\gamma$ is the largest radius appearing in $\widehat{\B_\epsilon^p}$ and $\widehat{\B_\epsilon^q}$; see Remark~\ref{remark-boundary}. We first prove that $\wb(x)\leq g(x)$ on $\partial\Omega$. Clearly,
		$$
		\wb(x)\leq g(x_0)-\eta
		\qquad\mbox{for all }x\in\partial\Omega.
		$$
		Since $g\in C_b$, there exists $\rho>0$, depending only on $g$ and $\eta$, such that
		$$
		\wb(x)\leq g(x)-\frac{\eta}{2}
		\qquad\mbox{for all }x\in B_\rho(x_0)\cap\partial\Omega.
		$$
		For the remaining boundary points, we may again choose $\kappa$ sufficiently large so that
		$$
		\wb(x)\leq g(x)-\frac{\eta}{2}
		\qquad\mbox{for all }x\in
		\partial\Omega\setminus\big(B_\rho(x_0)\cap\partial\Omega\big).
		$$
		Then, the two previous inequalities hold for all $x\in\partial\Omega$.
		We now extend this result to a neighborhood of the boundary, namely $\partial\Omega_\epsilon$. Let $x\in\partial\Omega_\epsilon$. Then there exists $\widehat{x}\in\partial\Omega$ such that
		$|x-\widehat{x}|\leq\epsilon^\gamma.$
		Therefore, using the smoothness of $\wb$ and the continuity of $g$, we obtain
		\begin{equation*}
			\wb(x)
			\leq \wb(\widehat{x})
			+\|\gr\wb\|_{\LL^\infty(\Omega_\epsilon)}\epsilon^\gamma
			\leq g(\widehat{x})-\frac{\eta}{4}
			\leq g(x)-\frac{\eta}{8}.
		\end{equation*}
		Here, we have taken $\epsilon\in(0,\beps)$ for some $\beps>0$ depending on $\eta$.
		
		By the comparison principle, Theorem~\ref{comparison}, we obtain
		$$
		\wb(x)\leq\ue(x)
		\qquad\mbox{for all }x\in\overline{\Omega},
		$$
		and hence
		$$
		\liminf_{\epsilon\to0\, ,\, x\to x_0}\ue(x)
		\geq
		\liminf_{\epsilon\to0\, ,\, x\to x_0}\wb(x)
		=g(x_0)-\frac{\eta}{8}.
		$$
		Since $\eta>0$ is arbitrary, the desired result follows. 
		
		The upper bound is obtained by repeating the argument with the upper barrier
		$$
		\wa(x)=-\kappa\left(|x-z|^{-\alpha}-R^{-\alpha}\right)+g(x_0)+\eta.
		$$
	\end{proof}
	
	A final key observation in the convergence argument is the monotonicity of the operator $\widehat{\C_\epsilon}$, inherited from the monotonicity of $\widehat{\A_\epsilon^p}$. More precisely, given two bounded measurable functions satisfying $\ph_2\leq \ph_1$, a bounded measurable function $\ph$, and any constant $\xi\in\RR$, the following identities hold:
	$$\widehat{C_\epsilon}[\ph_1]\geq \widehat{C_\epsilon}[\ph_2]\qquad \mbox{and}\qquad \widehat{C_\epsilon}[\ph+\xi]=\widehat{C_\epsilon}[\ph]+\xi.$$
	
	With all these tools in place, we are now in a position to prove the convergence, as $\epsilon \to 0^+$, of the solutions of the DPP to the viscosity solution of the elliptic problem.
	\begin{thm}\label{convergence-elliptic}
		Let $d\geq 2$, $p>q>2$, $\epsilon\in (0,1)$, and let $\Omega\subset \Rd$ be a bounded domain satisfying the uniform exterior ball condition. Assume that \eqref{m-M} and \eqref{m-M-b} hold for both $p$ and $q$. If $\ue$ is the solution of \eqref{DPP} with data $f\in C(\overline{\Omega})$ and $g\in C_b(\Rd\setminus \Omega)$, define
		\begin{equation*}
			\ua(x)=\limsup_{\epsilon\to 0 \, \, \, y\to x}\ue(y)\qquad\mbox{and}\qquad \ub(x)=\liminf_{\epsilon\to 0 \, \, \, y\to x} \ue(y).
		\end{equation*}
		Then, $\ua=\ub$. In particular, if $u=\ua=\ub$, then
		\begin{equation*}
			\ue \longrightarrow u \qquad \mbox{ as }\epsilon\to 0^+\quad \mbox{uniformly in }\overline{\Omega}.
		\end{equation*}
		This limit coincides with the viscosity solution $u$ of~\eqref{anosotropic-problem}.
	\end{thm}
	\begin{proof}
		By construction, we have $\ub(x)\leq \ua(x)$. Thus, uniform convergence will be established once we prove the reverse inequality. This follows from the comparison principle, see Lemma \ref{com-vis-elliptic} and \cite{Crandall_Ishii_Lions_1992}, provided that we show that $\ua$ is a viscosity subsolution of problem \eqref{anosotropic-problem}.
		
		On the one hand,
		$$\ua(x)=\limsup_{\epsilon\to 0 \, \, \, y\to x}\ue(y)\leq g(x)\qquad \mbox{for all }x\in \partial \Omega.$$
		
		On the other hand, let $x_0=(x_0',x_0'')\in \Omega$, $R>0$, and $\ph\in C^\infty(B_R(x_0))$ be any test function such that $\ph(x_0)=\ua(x_0)$ and $\ph(x)>\ua(x)$ at all other points of the ball. Since $x_0$ is a strict maximum point of $\ua-\ph$, by~\cite{Barles_Souganidis_1991}, there exists a sequence $(\epsilon_n,x_n)\to(0^+,x_0)$ such that $x_n$ are quasi-maximum points of $u_{\epsilon_n}-\ph$, namely,
		\begin{equation}\label{quasi-max-ellip}
			u_{\epsilon_n}(x)-\ph(x)\leq u_{\epsilon_n}(x_n)-\ph(x_n)+\ee^{\frac{1}{\epsilon_n}}.
		\end{equation}
		
		Now, for $n$ sufficiently large, we have $\epsilon_n<\beps$ and $x_n\in \Omega$, and therefore we can apply the asymptotic formulas from Theorem \ref{exp-c3-p_q}. Define $\xi_n=u_{\epsilon_n}(x_n)-\ph(x_n)$. Since $u_{\epsilon_n}$ is a solution of \eqref{DPP}, we have
		\begin{align*}
			0&=\frac{u_{\epsilon_n}(x_n)-\widehat{\C_\epsilon}[u_{\epsilon_n}](x_n)}{\epsilon_n^2/2}+f(x_n)\\
			&\geq \frac{\ph(x_n)+\xi_n -\widehat{\C_\epsilon}[\ph+\xi_n+\ee^{\frac{1}{\epsilon_n}}](x_n)}{\epsilon_n^2/2}+f(x_n)\\
			&=-\plap^{x'}\ph(x_n',x_n'')-\qlap^{x''}\ph(x_n',x_n'')+E(\epsilon_n,\ph,x_n)-\frac{2 \ee^{\frac{1}{\epsilon_n}}}{\epsilon_n^2}+f(x_n).
		\end{align*}
		Here, we have used \eqref{quasi-max-ellip} and Theorem~\ref{exp-c3-promedio}. Furthermore, the error term can be bounded independently of the point $x_n$; see formulas~\eqref{error-B-grad} and~\eqref{error-B-gr-not}. Indeed, from some point onward, all points in the sequence lie in a neighborhood of $x_0$. Taking limit as $n\to\infty$, and using the regularity of $\ph$ and $f$, we conclude that $\ua$ is a viscosity subsolution, namely,
		\begin{equation*}
			\plap^{x'} \ph(x_0',x_0'')+\qlap^{x''}\ph(x_0',x_0'')\geq f(x_0).
		\end{equation*}
		
		In order to see that $\ub$ is a viscosity supersolution, we only need to repeat the same argument as above. The only differences are that
		$$\ub(x)=\liminf_{\epsilon\to 0 \, \, \, y\to x}\ue(y)\geq g(x)\qquad \mbox{for all }x\in \partial \Omega,$$
		and that the sequence of points $(\epsilon_n,x_n)\to (0^+,x_0)$ now satisfies the following condition:
		$$u_{\epsilon_n}(x_n)-\ph(x_n)-\ee^{\frac{1}{\epsilon_n}}\leq u_{\epsilon_n}(x)-\ph(x).$$
	\end{proof}
	

	\section{The associated game}\label{sec5}
	
	The DPP studied in Section \ref{sec4} is related with a game which is described in subsection \ref{sec-main-game}. Our goal here is to show that it has a value that coincides with the solution of the DPP \eqref{DPP}.
	
	The game associated with $\widehat{\C_\epsilon}$, generates a sequence of states (i.e., positions of the token) starting from $x_0\in \Omega$,
	
	$$
	P=\left\{x_0, x_1, \ldots , x_\tau \right\},
	$$
	
	where $x_\tau$ is the final state, with $\tau$ being the stopping time defined as the first time at which $x_\tau\in \Rd\setminus \Omega$.
	Observe that the game terminates after finitely many plays with probability one, since, for $\epsilon>0$, we added a small random movement in the definition of $\widehat{\A_\epsilon^p}$.
	
	The final payoff corresponding to a realization of the game is given by
	$$g(x_\tau)-\frac{\epsilon^2}{2}\sum_{i=0}^{\tau-1}f(x_i),$$
	which consists of the final payoff $g(x_\tau)$ and the running payoff accumulated along the sequence of states $x_0, x_1, \ldots , x_{\tau-1}$.
	
	A strategy for Player~$I$ is a function $S_I$ defined on the partial histories of the game that specifies the choices available to Player~$I$ at each round. On the one hand, if the first fair coin lands heads, it determines the choices of Player~$I$ when the game associated with $\widehat{\B_\epsilon^p}$ is played in balls of $\mathbb{R}^n$ centered at $x=(x',x'')\in \Rd$, with the last coordinates $x''$ kept fixed. On the other hand, if the first coin lands tails, $S_I$ specifies how Player~$I$ acts while playing $\widehat{\B_\epsilon^q}$ in balls of $\mathbb{R}^{d-n}$ centered at $x=(x',x'')\in \Rd$, with the first coordinates $x'$ kept fixed.
	Similarly, a strategy for Player~$II$ is a function $S_{II}$, also defined on the partial histories of the game, that specifies how this player acts at every round.
	
	Once the two players have fixed their strategies, $S_I$ and $S_{II}$, the outcome of the game depends only on the initial position and the sequence of coin tosses involved in each play, which induce a probability measure on the space of sequences of states. With respect to this probability measure, we can compute the expected payoff when playing with strategies $S_I$ and $S_{II}$ and starting from $x_0$, namely,
	
	$$
	\mathbb{E}_{S_I, S_{II}}^{x_0}\left( g(x_\tau)-\epsilon^2\sum_{i=0}^{\tau-1}\frac{f(x_i)}{2} \right).
	$$
	
	Since Player~$I$ seeks to maximize the expected payoff, whereas Player~$II$ seeks to minimize it, the lower and upper values of the game are defined, respectively, by
	
	$$
	u^\epsilon_I(x_0)
	=
	\sup_{S_{I}}
	\left(
	\inf_{S_{II}} \left(
	\mathbb{E}_{S_I, S_{II}}^{x_0}
	\left(g(x_\tau)-\epsilon^2\sum_{i=0}^{\tau-1}\frac{f(x_i)}{2}\right)
	\right)\right)
	$$
	
	and
	
	$$
	u^\epsilon_{II}(x_0)
	=
	\inf_{S_{II}}
	\left(
	\sup_{S_{I}} \left(
	\mathbb{E}_{S_I, S_{II}}^{x_0}
	\left(g(x_\tau)-\epsilon^2\sum_{i=0}^{\tau-1}\frac{f(x_i)}{2}\right)
	\right) \right).
	$$
	
	These represent the best and worst expected outcomes that the two players can guarantee, respectively. In game theory, the game is said to have a value if if these two quantities coincide, that is,
	$u^\epsilon_{I}(x_0) = u^\epsilon_{II}(x_0)$.
	
	Our next result shows that the game described above has a value $\ue(x)$, which is characterized as the unique solution of the problem \eqref{DPP} studied in the previous section.
	
	\begin{thm} \label{teo.juego.p-q}
		Let the assumptions of Theorem \ref{convergence-elliptic} hold. Then, the game associated with $\widehat{\C_\epsilon}$ has a 
		value which is characterized by the unique solution $\ue$ to the \eqref{DPP}, that is
		$$\ue(x)=u_I^\epsilon(x)=u_{II}^\epsilon(x)\qquad\mbox{for all }x\in \Rd.$$
		
		In particular, the value of the game converges uniformly in $\overline{\Omega}$ as $\epsilon\to 0^+$ to viscosity solution to the elliptic problem \eqref{anosotropic-problem}.
	\end{thm}
	\begin{proof}
		It is clear that $u_I^\epsilon(x_0)\leq u_{II}^\epsilon(x_0)$. Thus, it only remains to prove the reverse inequality.
		
		\noindent \textit{Step 1: We show that $\ue\leq u_I^\epsilon$.} Fix any $\eta>0$. Given any strategy $S_{II}$ for Player~$II$, we choose a strategy $S_I^*$ for Player~$I$ by following the solution of the DPP \eqref{DPP} starting at $x_k$.
		
		\begin{itemize}[leftmargin=*]
			\item \textbf{First toss of a fair coin, with probability $1/2$ of landing heads and probability $1/2$ of landing tails.}\\
			If the outcome is heads, a round of $\widehat{\B_\epsilon^p}$ is played. Otherwise, a round of $\widehat{\B_\epsilon^q}$ is played. Since the descriptions of the two games are analogous, we explain only the case of $\widehat{\B_\epsilon^p}$. The other case is entirely analogous, except that one works with $q>2$ and the last $d-n$ variables $x_k''$, while keeping the first $n$ variables $x_k'$ fixed.
			
			\item \textbf{Second toss of a fair coin, with probability $1/2$ of landing heads and probability $1/2$ of landing tails.}\\
			If the outcome is heads, then Player~$I$ chooses a constant $b^*\in[m_b,M_b]$ such that
			\begin{equation*}
				\widehat{\A_{\epsilon\sqrt{2(p-1)b^*}}^p} [\ue](x_k)-2(p-2)\epsilon^2 (b^*)^{\frac{p-1}{p-2}}\geq \sup_{b\in[m_b,M_b]}\left\{\widehat{\A_{\epsilon\sqrt{2(p-1)b}}^p} [\ue](x_k)-2(p-2)\epsilon^2 b^{\frac{p-1}{p-2}}\right\}-\frac{\eta}{2^{k+2}}.
			\end{equation*}
			After that, Player~$I$ pays Player~$II$ the amount $2(p-2)\epsilon^2 (b^*)^{\frac{p-1}{p-2}}$, and a round of the game associated with $\widehat{\A_{\epsilon \sqrt{2(p-1)b^*}}^p}$ is played.
			
			\item \textbf{Third toss of a fair coin, with probability $1/2$ of landing heads and probability $1/2$ of landing tails.}\\
			If the outcome is heads, Player~$I$ decides to move randomly within a small ball $B_{(2(p-1)b^*)^4\epsilon^8}(x_k)$ if
			\begin{equation}\label{eq1-elliptic-game}
				\begin{split}
					\fint_{B_{(2(p-1)b^*)^4\epsilon^8}(x_k)}\ue(y',x_k'')\, \dy'>\inf_{c\in[m_p,M_p]}&\Biggl\{\alpha_p \sup_{B_{4(p-1)b^* \epsilon^2 c^{1-\alpha_p}}(x_k)}\ue(\cdot,x_k'')\\ &\quad+(1-\alpha_p)M_{\epsilon\sqrt{4(p-1)b^*}c^{-\alpha_p/2}}[\ue](x_k',x_k'')\Biggr\}.
				\end{split}
			\end{equation}
			If this does not hold, Player~$II$ chooses any $c\in[m_p,M_p]$, and Player~$I$ determines the next position as follows:
			
			\item \textbf{Fourth toss of a biased coin, with probability $\alpha_p=(p-2)/(p-1)$ of landing heads and probability $1-\alpha_p=1/(p-1)$ of landing tails.}\\
			If the outcome is heads, Player~$I$ chooses the next position $x_{k+1}^*\in B_{4(p-1)b^*\epsilon^2 c^{1-\alpha_p}}(x_k)$ such that
			\begin{equation}\label{eq2-elliptic-game}
				\ue(x_{k+1}^*)\geq \sup_{B_{4(p-1)b^*\epsilon^2 c^{1-\alpha_p}}(x_k)}\ue(\cdot,x_k'')-\frac{\eta}{2^{k+1}}.
			\end{equation}
			
			If the outcome is tails, a tug-of-war with noise is played. In this case, Player~$I$ chooses $\widehat{x}_{k+1}\in B_{\sqrt{8(p+d)(p-1)b^*} \epsilon c^{-\alpha_p/2}}(x_k)$ such that
			\begin{equation} \label{eq3-elliptic-game}
				\ue(\widehat{x}_{k+1})
				\geq
				\sup_{B_{\sqrt{8(p+d)(p-1)b^*} \epsilon c^{-\alpha_p/2}}(x_k)}
				\ue(\cdot,x_k'')
				-
				\frac{\eta}{2^{k+1}}.
			\end{equation}
			
			\item If, in the third toss of the fair coin, the outcome is tails and Player~$II$ chooses to play, then Player~$I$ chooses a constant $c^*\in[m_p,M_p]$ such that
			\begin{equation}\label{eq4-elliptic-game}
				\begin{split}
					\alpha_p& \inf_{4(p-1)b^* \epsilon^2 (c^*)^{1-\alpha_p}}\ue(\cdot,x_k'')+(1-\alpha_p)M_{\sqrt{4(p-1)b^*}\epsilon (c^*)^{-\alpha_p/2}}[\ue](x_k',x_k'')
					\geq\\
					&\sup_{c\in[m_p,M_p]}\left\{\alpha_p \inf_{4(p-1)b^* \epsilon^2 c^{1-\alpha_p}}\ue(\cdot,x_k'')+(1-\alpha_p)M_{\sqrt{4(p-1)b^*}\epsilon c^{-\alpha_p/2}}[\ue](x_k',x_k'')\right\}-\frac{\eta}{2^{k+1}}.
				\end{split}
			\end{equation}
			Moreover, if a round of tug-of-war is played, Player~$I$ chooses $\widehat{x}_{k+1}$ as in~\eqref{eq3-elliptic-game}.
			
			\item If, in the second toss, the outcome is tails, then Player~$II$ chooses any $b\in [m_b,M_b]$ and a round of the game generated by $\widehat{\A_{\epsilon\sqrt{2(p-1)b}}^p}$ is played. Player~$I$ chooses the next position $x_{k+1}$ according to the previous rules \eqref{eq1-elliptic-game}, \eqref{eq2-elliptic-game}, \eqref{eq3-elliptic-game}, and \eqref{eq4-elliptic-game}, with $b$ in place of $b^*$. With a slight abuse of notation, we denote these positions by $x_{k+1}^*$, $\widehat{x}_{k+1}$ and the constant by $c^*$, as before.
		\end{itemize}
		
		Fix the strategies $S_I^*$ and $S_{II}$, and consider the sequence of random variables
		$$M_k=\ue(x_k)-\frac{\eta}{2^k}-\epsilon^2 \sum_{i=0}^{k-1}\frac{f(x_i)}{2}.$$
		To compute the conditional expectation of the random variable $M_{k+1}$ given the previous position $x_k$, we need to take into account that $\ue$ is the solution of the DPP. For this purpose, we introduce the following notation:
		\begin{align*}
			&N_p[\ue](x_k):=\frac{1}{2}\Bigg(-2(p-2)\epsilon^2 (b^*)^{\frac{p-1}{p-2}}
			+\frac{1}{2}\max\Bigg[ \fint_{B_{(2(p-1)b^*)^4\epsilon^8}(x_k)}\ue(y',x_k'')\, \dy'-\frac{\eta}{2^{k+1}},\\
			&\qquad \qquad\qquad \qquad \quad\quad \alpha_p \ue(x_{k+1}^*) +(1-\alpha_p)\frac{\beta_p}{2}\left(\ue(\widehat{x}_{k+1})+\ue(x_{k+1}^{II})\right) 
			\\ & \qquad \qquad \qquad \qquad\qquad \qquad \qquad +(1-\alpha_p)(1-\beta_p)\fint_{B_{\sqrt{8(p+d)(p-1)b^*}\epsilon c^{-\alpha_p/2}}(x_k)}\ue(y',x_k'') \,\dy' \Bigg]\\
			&\quad+\frac{1}{2}\min\Bigg[\fint_{B_{(2(p-1)b^*)^4\epsilon^8}(x_k)}\ue(y',x_k'')-\frac{\eta}{2^{k+1}},\\ 
			&\quad \quad\alpha_p \inf_{B_{4(p-1)b^* \epsilon^2 (c^*)^{1-\alpha_p}}(x_k)}\ue(\cdot,x_k'')+(1-\alpha_p)M_{\sqrt{4(p-1)b^*}\epsilon(c^*)^{-\alpha_p/2}}[\ue](x_k',x_k'')\Bigg]\Bigg)\\
			&+ \frac{1}{2}\Bigg(2(p-2)\epsilon^2 b^{\frac{p-1}{p-2}}
			+\frac{1}{2}\max\Bigg[ \fint_{B_{(2(p-1)b^*)^4\epsilon^8}(x_k)}\ue(y',x_k'')\, \dy'-\frac{\eta}{2^{k+1}},\\
			&\qquad \qquad\qquad \qquad\qquad \qquad\quad\quad \alpha_p \ue(x_{k+1}^*)  +(1-\alpha_p)\frac{\beta_p}{2}\left(\ue(\widehat{x}_{k+1})+\ue(x_{k+1}^{II})\right) \\
			& \qquad \qquad\qquad \qquad\qquad \qquad\qquad \qquad+(1-\alpha_p)(1-\beta_p)\fint_{B_{\sqrt{8(p+d)(p-1)b}\epsilon c^{-\alpha_p/2}}(x_k)}\ue(y',x_k'') \,\dy' \Bigg]\\
			&\quad+\frac{1}{2}\min\Bigg[\fint_{B_{(2(p-1)b^*)^4\epsilon^8}(x_k)}\ue(y',x_k'')\, \dy'-\frac{\eta}{2^{k+1}},\\
			&\quad \quad \alpha_p \inf_{B_{4(p-1)b \epsilon^2 (c^*)^{1-\alpha_p}}(x_k)}\ue(\cdot,x_k'')+(1-\alpha_p)M_{\sqrt{4(p-1)b}\epsilon(c^*)^{-\alpha_p/2}}[\ue](x_k',x_k'')\Bigg]\Bigg),
		\end{align*}
		
		Now, taking into account the way in which the strategy $S_I^*$ has been defined, we can use the inequalities \eqref{eq1-elliptic-game}, \eqref{eq2-elliptic-game}, \eqref{eq3-elliptic-game}, and \eqref{eq4-elliptic-game} to obtain
		\begin{align*}
			&\mathbb{E}^{x_0}_{S_I^*, S_{II}}[M_{k+1}\mid x_k]
			=\frac{1}{2}N_p[\ue](x_k) + \frac{1}{2}N_q[\ue](x_k) - \frac{\epsilon^2}{2}\sum_{i=0}^{k}f(x_i)\\
			&\geq \frac{1}{4}\left(\sup_{b\in[m_b,M_b]}\left\{\widehat{\A_{\epsilon\sqrt{2(p-1)b}}^p}[\ue](x_k)-2(p-2)\epsilon^2 b^{\frac{p-1}{p-2}}\right\}-\frac{\eta}{2^{k+2}}-\frac{2\eta}{2^{k+2}}-(\alpha_p+(1-\alpha_p)\beta_p)\frac{\eta}{2^{k+2}}\right)\\
			&\quad\quad + \frac{1}{4}\left(\inf_{b\in[m_b,M_b]}\left\{\widehat{\A_{\epsilon\sqrt{2(p-1)b}}^p}[\ue](x_k)-2(p-2)\epsilon^2 b^{\frac{p-1}{p-2}}\right\}-\frac{\eta}{2^{k+2}}-\frac{2\eta}{2^{k+2}}-(\alpha_p+(1-\alpha_p)\beta_p)\frac{\eta}{2^{k+2}}\right)\\
			&\quad\quad+
			\frac{1}{4}\left(\sup_{b\in[m_b,M_b]}\left\{\widehat{\A_{\epsilon\sqrt{2(p-1)b}}^q}[\ue](x_k)-2(q-2)\epsilon^2 b^{\frac{q-1}{q-2}}\right\}-\frac{\eta}{2^{k+2}}-\frac{2\eta}{2^{k+2}}-(\alpha_q+(1-\alpha_q)\beta_q)\frac{\eta}{2^{k+2}}\right)\\
			&\quad\quad + \frac{1}{4}\left(\inf_{b\in[m_b,M_b]}\left\{\widehat{\A_{\epsilon\sqrt{2(q-1)b}}^q}[\ue](x_k)-2(q-2)\epsilon^2 b^{\frac{q-1}{q-2}}\right\}-\frac{\eta}{2^{k+2}}-\frac{2\eta}{2^{k+2}}-(\alpha_q+(1-\alpha_q)\beta_q)\frac{\eta}{2^{k+2}}\right)\\
			&\quad \quad - 
			\frac{\epsilon^2}{2}\sum_{i=0}^{k}f(x_i)\\
			&= \widehat{\C_\epsilon}[\ue](x_k)-\frac{1}{2}(3+\alpha_p+(1-\alpha_p)\beta_p)\frac{\eta}{2^{k+2}}-\frac{1}{2}(3+\alpha_q+(1-\alpha_q)\beta_q)\frac{\eta}{2^{k+2}}- 
			\frac{\epsilon^2}{2}\sum_{i=0}^{k}f(x_i)\\
			&\geq \ue(x_k)-\frac{\eta}{2^k} - 
			\frac{\epsilon^2}{2}\sum_{i=0}^{k}f(x_i).
		\end{align*}
		
		Once we have established that $(M_k)_{k\geq 0}$ is a submartingale, we apply the Optional Stopping Theorem (cf.\ \cite{Williams}) to obtain
		$$\mathbb{E}^{(x_0)}_{S^*_I,S_{II}}
		[M_{\tau}]
		\geq
		M_0
		=
		\ue(x_0)-\eta.$$
		
		Since $S_{II}$ is an arbitrary strategy, we have
		$$\inf_{S_{II}} \mathbb{E}^{x_0}_{S^*_I,S_{II}}
		[g(x_\tau)]
		=
		\inf_{S_{II}} \mathbb{E}^{x_0}_{S^*_I,S_{II}}
		[\ue(x_\tau)]
		\geq
		\inf_{S_{II}} \mathbb{E}^{x_0}_{S^*_I,S_{II}}
		[M_{\tau}]
		\geq
		\ue(x_0)-\eta.$$
		Therefore,
		\begin{align*}
			\ue(x_0)-\eta
			&\leq
			\inf_{S_{II}}
			\mathbb{E}^{x_0}_{S_I^*, S_{II}}
			\left[
			g(x_\tau)-\epsilon^2\sum_{i=0}^{\tau-1}\frac{f(x_i)}{2}
			\right]\\
			&\leq
			\sup_{S_I}
			\left(
			\inf_{S_{II}}
			\mathbb{E}^{x_0}_{S_I, S_{II}}
			\left[
			g(x_\tau)-\epsilon^2\sum_{i=0}^{\tau-1}\frac{f(x_i)}{2}
			\right]
			\right)
			=
			u_I^\epsilon(x_0).
		\end{align*}
		The first step is completed by the arbitrariness of $\eta$.
		
		\noindent\textit{Step 2: We show that $u^{\epsilon}_{II}\leq \ue$.} The proof follows the same argument as in Step 1, since the game is symmetric. For this reason, we omit the details and provide only a sketch of the argument. Let $\eta>0$ be arbitrary and, given any strategy $S_I$ for Player $I$, choose $S_{II}^*$ analogously to Step 1.
		
		Once the strategies $S_I$ and $S_{II}^*$ are fixed, we define the random variable
		$$M_k=\ue(x_k',x_k'')+\frac{\eta}{2^k}-\frac{\epsilon^2}{2} \sum_{i=0}^{k-1}f(x_i).$$
		It can be shown that $(M_k)_k$ is a supermartingale. Therefore, by the Optimal Stopping Theorem, we obtain
		
		$$\mathbb{E}^{(x_0,t_0)}_{S_I,S_{II}^*}[M_{\tau}]\leq
		M_0= \ue(x_0,t_0)+\eta,$$
		and hence
		\begin{align*}
			\ue(x_0)+\eta
			\geq
			\sup_{S_{I}}
			\mathbb{E}^{(x_0)}_{S_I, S_{II}^*}[g(x_\tau)-\epsilon^2\sum_{i=0}^{\tau-1}\frac{f(x_i)}{2}]
			\geq
			\inf_{S_{II}}
			\left(
			\sup_{S_{I}}
			\mathbb{E}^{(x_0)}_{S_I, S_{II}}[g(x_\tau)-\epsilon^2\sum_{i=0}^{\tau-1}\frac{f(x_i)}{2}]
			\right)
			=
			u_{II}^\epsilon(x_0).
		\end{align*}
		Since $\eta$ is arbitrary, we obtain $\ue(x_0)\geq u_{II}^\epsilon(x_0)$. Therefore, we conclude that
		$$
		\ue(x_0)
		=  u_{I}^\epsilon(x_0) = u_{II}^\epsilon(x_0),
		$$
		proving that the game has a value which coincides with the unique solution of \eqref{DPP}.
		
		Since $\ue(x)\to u(x)$ uniformly in $\overline{\Omega}$ as $\epsilon\to 0^+$, see Theorem \ref{convergence-elliptic}, we conclude that the game has a value, which converges to $u$, the viscosity solution of \eqref{anosotropic-problem}.
	\end{proof}
	
	\begin{rem}\label{teo.juego.p}\rm
		With the same arguments it can be proved that the game associated to 
		$\widehat{\B_\epsilon^p}$ (the operator which give us an asymptotic expansion for the $p$-Laplacian) has a value that converges uniformly in $\overline{\Omega}$ to the solution to the Poisson problem for the $p-$Laplacian,
		\begin{equation}\label{p-Laplacian-poisson}
			\left\{\begin{array}{rlll}
				\plap u(x)&=&f(x), & \mbox{in } \Omega,\\ 
				u(x)&=&g(x), & \mbox{in } \partial \Omega
			\end{array}\right.
		\end{equation}.
	\end{rem}
	A different game for the problem \eqref{p-Laplacian-poisson} was analyzed in \cite{delTeso_Rossi_2024}.
	The game introduced here (associated with $\widehat{\B_\epsilon^p}$)
	is more flexible and can be adapted to tackle problems with different homogeneities (as we previously show) or parabolic problems (see
	Section \ref{sec7} below).
	
	
	\section{Holder regularity}\label{sec6}
	
	In this section, we show how to establish a certain degree of interior Hölder regularity for solutions to~\eqref{anosotropic-problem} with $f=0$. To this end, we draw inspiration from the techniques developed in~\cite{Peres+Wang_2026}, see also \cite{Attouchi+Luiro+Parviainen-2021, Luiro+Parviainen+Saksman-2013}, where Hölder regularity and Harnack inequalities are established for viscosity solutions of the homogeneous Poisson problem \eqref{p-Laplacian-poisson}. The strategy consists in controlling the oscillation of the DPP solutions $\ue$ in balls of radius $\delta$ by their oscillation in balls of radius $1$, independently of $\epsilon$, using the game induced by $\widehat{\C_\epsilon}$. We then take limits as $\epsilon\to 0$, relying on the convergence to the unique viscosity solution of the problem \eqref{anosotropic-problem}. Moreover, in \cite{Arroyo+Luiro+Parviainen+Ruosteenoja-2020 }, the method is used to prove the Lipschitz continuity for $p(x)$-harmonic functions. In \cite{Attouchi+Luiro+Parviainen-2021},a regularized tug-of-war game is constructed to provide a Lipschitz continuous game value.
	
	In order to arrange the various radii appearing in the mean operator, it is necessary to slightly modify the regions over which the infima and suprema are taken in $\widehat{\A_\epsilon^p}$ and $\widehat{\B_\epsilon^p}$; see formulas \eqref{operator-A-promedio} and \eqref{operator-B-promedio}. More precisely, we make the following modifications:
	
	On the one hand, the constant $c\in [m_p,M_p]$ appearing in $\widehat{\A_\epsilon^p}$ is taken from the new interval
	\begin{equation}\label{m-M-2}
		m_p(\epsilon)=\epsilon^{\frac{2(p-1)}{3p-4}}\qquad \mbox{and}\qquad M_p(\epsilon)=\epsilon^{-\frac{p-1}{p}}.
	\end{equation}
	On the other hand, the constant $b\in[m_b,M_b]$ appearing in $\widehat{\B_\epsilon^p}$ is taken from the new interval
	\begin{equation}\label{m-M-b-2}
		m_b(\epsilon)=\epsilon^2\qquad \mbox{and}\qquad M_b(\epsilon)=\epsilon^{-\frac{1}{2}\min\left[\frac{2(p-2)}{3p-2} ,\, \frac{2}{3(p-1)}\right]}.
	\end{equation}
	With these new intervals, it is straightforward to verify that, for any choice of $c \in [m_p,M_p]$ and $b \in [m_b,M_b]$, the following inequalities hold:
	\begin{equation}\label{radios-ordenados}
		(2(p-1)b)^4 \epsilon^8 \ll 4(p-1)b \epsilon^2 c^{1-\alpha_p}<\sqrt{4(p-1)(p+d)b}\epsilon c^{-\alpha_p/2}\qquad\mbox{for any }\epsilon<1.
	\end{equation}
	
	These modifications do not affect the previous results, but only slightly change the error estimates in the asymptotic formulas for $C^3$ functions, as summarized below.
	\begin{enumerate}[leftmargin=*]
		\item In Theorem \ref{thm.asexp}, we have to replace $\epsilon^{2-\frac{4}{p}}$ by $\epsilon^{1-\frac{2}{p}}$ in formulas \eqref{E1} and \eqref{E1-gr-not}.
		
		\item In Theorems \ref{main-asym-exp-plap} and \ref{exp-c3}, $\epsilon^{-\frac{1}{2}\min\left[\frac{p-2}{p-1}, \frac{2}{3(p-1)}\right]}$ should be replaced by $\epsilon^{-\frac{1}{2}\min\left[\frac{2(p-2)}{3p-4}, \frac{2}{3(p-1)}\right]}$. Moreover, $\epsilon^{1-\frac{2}{p}}$ should be replaced by $\epsilon^{\frac{p-2}{2p}}$ in the second and third lines, respectively, of \eqref{error-B-grad} and \eqref{error-B-gr-not}.
	\end{enumerate}
	
	\begin{lem}\label{lema-Holder}
		Let the assumptions of Theorem \ref{convergence-elliptic} hold, with these new intervals \eqref{m-M-2} and \eqref{m-M-b-2}. Let $u$ be the viscosity solution of \eqref{anosotropic-problem} in $B_R(0)$ with $f=0$ and $R>1$. Then there exists a constant $K>0$, depending only on $p$, $q$, and $d$, such that, for every $0<\delta<1$,
		\begin{equation}
			\sup_{B_\delta} u - \inf_{B_\delta} u \leq (K+2)\delta^{1/3}\left(\sup_{B_1} u -\inf_{B_1} u\right).
		\end{equation}
	\end{lem}
	\begin{proof}
		Let $\delta\in (0,1)$ and $x,y\in B_\delta(0)$. We consider two games (induced by $\widehat{\C_\epsilon}$) starting simultaneously from $x$ and $y$, respectively, with the same boundary data $g$ on $\partial B_1(0)$. Following the notation of \cite{Peres+Wang_2026}, we denote the resulting processes by $(X_n)$ and $(Y_n)$, respectively. For each $n\geq 1$, let $\mathcal{F}_n$ denote the $\sigma$-field generated by $(X_i,Y_i)_{i\leq n}$. We also introduce the notation
		
		$$
		Z_n:=X_n-Y_n
		\qquad\mbox{and}\qquad
		\widehat{v}:=\frac{v}{|v|}.
		$$
		
		Given a strategy $S_I$ used by Player~$I$ (who seeks to maximize the payoff) in the game starting from $x$, and a strategy $S_{II}$ used by Player~$II$ (who seeks to minimize the payoff) in the game starting from $y$, we define the counter-strategies $S_{II}^*(S_I,S_{II})$ for Player~$II$ in the game starting from $x$ and $S_I^*(S_I,S_{II})$ for Player~$I$ in the game starting from $y$ as follows. The idea is to replicate the choices of the constants $b\in[m_b,M_b]$ and $c\in[m_p,M_p]$ (or $c\in[m_q,M_q]$) prescribed by the strategies $S_I$ and $S_{II}$, respectively, and then move towards the position occupied by the other game. More precisely, whenever it is Player~$I$'s turn to move, the strategy $S_I^*$ prescribes a move from $Y_n$ in the direction $\widehat{Z_n}$, that is, towards $X_n$. Conversely, whenever it is Player~$II$'s turn to move, the analogous prescription is made in the opposite game. We now define the strategies $S_I^*$ and $S_{II}^*$.
		
		\begin{itemize}[leftmargin=*]
			\item The toss of the first fair coin determines whether the exponent $p$ or $q$ is selected. Since the game is completely symmetric, we may assume without loss of generality that $p$ is selected, as the case $q$ is entirely analogous.
			
			\item The toss of the second fair coin determines which player chooses the value of $b\in[m_b,M_b]$. If Player~$I$ chooses it, the strategy $S_I$ determines $b_I$ for the game associated with $X_n$, while the strategy $S_I^*$ prescribes choosing the same value $b_I$ in the game associated with $Y_n$. Otherwise, the strategy $S_{II}$ determines $b_{II}$ for the game associated with $Y_n$, while $S_{II}^*$ prescribes choosing the same value $b_{II}$ in the game associated with $X_n$.
			
			\item The toss of the third fair coin determines which player decides whether to remain ``almost stationary''---that is, whether to move randomly within a ball of radius $(4(p-1)b)^4\epsilon^8$---or to initiate the movement. If Player~$I$ is the one who makes this decision according to the strategy $S_I$ in the game associated with $X_n$, then the strategy $S_I^*$ makes the same decision---whether to remain ``almost stationary'' or to move---in the game associated with $Y_n$. Likewise, if Player~$II$ makes the decision, the strategy $S_{II}^*$ replicates the action prescribed by $S_{II}$.
			
			\item If it is decided to remain ``almost stationary'', a vector $v$ is chosen uniformly in $B_{(2(p-1)b)^4\epsilon^8}(0)$. If Player~$I$ makes this choice, the new positions are
			$$
			X_{n+1}=X_n+v,
			\qquad
			Y_{n+1}=Y_n+\overline{v},
			$$
			whereas if Player~$II$ makes this choice, they are
			$$
			X_{n+1}=X_n+\overline{v},
			\qquad
			Y_{n+1}=Y_n+v.
			$$
			Here, $\overline{v}$ denotes the reflection of $v$ across the direction of $Z_n$, namely $\overline{v}=2\langle v,\widehat{Z_n}\rangle\widehat{Z_n}-v$.
			
			\item If the previous decision was to move, the next choice concerns the constant $c\in[m_p,M_p]$. If the decision was made by Player~$I$, then Player~$II$ chooses it. Thus, the strategy $S_{II}$ determines $c_{II}$, while $S_{II}^*$ selects the same constant. Conversely, if the decision was made by Player~$II$, then Player~$I$ chooses $c_I$ according to $S_I$ and $S_I^*$ selects the same constant.
			
			Observe that, up to this point, we have constructed $S_I^*$ and $S_{II}^*$ by mimicking the given strategies, with the aim of ensuring that the players always move in balls of exactly the same radius.
			
			\item Once the decision to move has been made in both games, the strategies determine the direction in which each player would move. The strategy $S_I^*$ for Player~$I$ in the game associated with $Y_n$ prescribes moving towards the position $X_n$; that is,
			$$
			Y_{n+1}=Y_n+R\widehat{Z}_n,
			$$
			where $R$ denotes the corresponding radius. Similarly, $S_{II}^*$ prescribes moving towards $Y_n$,
			$$
			X_{n+1}=X_n-R\widehat{Z}_n.
			$$
			If the first biased coin lands heads, with probability $\alpha_p$, then Player~$I$ (or Player~$II$, as appropriate) moves in both games within a ball of radius
			$$
			4(p-1)b\epsilon^2c^{1-\alpha_p}.
			$$
			If it lands tails, a round of tug-of-war with noise is played.
			If the second biased coin lands heads, with probability $\beta_p$, then the players move within balls of radius
			$$
			\sqrt{4(p-1)(p+d)b}\,\epsilon\,c^{-\alpha_p/2}.
			$$
			A fair coin is then tossed to determine which player moves. If it lands heads, Player~$I$ moves in the game associated with $X_n$, while Player~$II$ moves in the game associated with $Y_n$; if it lands tails, the roles are reversed.
			
			If the second biased coin lands tails, then a vector $u$ is chosen uniformly from the ball of radius $\sqrt{4(p-1)(p+d)b},\epsilon,c^{-\alpha_p/2}$ centered at $0$, and the new positions are given by
			$$
			X_{n+1}=X_n+u
			\qquad\mbox{and}\qquad
			Y_{n+1}=Y_n+u.
			$$
		\end{itemize}
		Now let $\beta=1/6$ and define the random variable 
		$$M_n=|X_n|^2+|Y_n|^2+K|Z_n|^{2\beta}.$$
		We want to show that it is a supermartingale for a suitable choice of $K$. In order to do it, we fix $\eta>0$ until the end of the proof, assume that $\eta<|Z_i|<1$ for all $i\leq n$ and we write $D_{n+1}=Z_{n+1}-Z_n$ to be the increment. As it is done in \cite{Peres+Wang_2026}, we use Taylor two times to obtain
		\begin{align*}
			|Z_{n+1}|^{2\beta}&=|Z_n|^{2\beta}\left(1+2 \widehat{Z_n} \cdot \frac{D_{n+1}}{|Z_n|}+\left|\frac{D_{n+1}}{|Z_n|}\right|^2\right)^\beta\\
			&= |Z_n|^{2\beta}\left(1+2\beta \widehat{Z_n} \cdot \frac{D_{n+1}}{|Z_n|}+\beta \left|\frac{D_{n+1}}{|Z_n|}\right|^2 -2\beta (1-\beta)\left(\widehat{Z_n}\cdot \frac{D_{n+1}}{|Z_n|}\right)^2+ O(\epsilon^3)\right).
		\end{align*}
		To compute the conditional expectation of \(M_{n+1}\), we fix a small angle $\theta_0<1/10$, and the calculation is divided into three cases. Observe that, due to the symmetry of the game, we may assume without loss of generality that the first three coin tosses result in heads. Any other combination can be treated in exactly the same way.
		
		\noindent\textit{Case 1: The moves prescribed by the strategies $S_I$ and $S_{II}$ are aligned with $\widehat{Z_n}$. } More precisely, if we denote by $U_{n+1}$ the move that Player~$I$ would make in the game starting from $x$, and by $V_{n+1}$ the move that Player~$II$ would make in the game starting from $y$, then
		$$\widehat{Z_n}\cdot \widehat{U_{n+1}}>\cos \theta_0\qquad \mbox{and}\qquad -\widehat{Z_n}\cdot \widehat{V_{n+1}}>\cos \theta_0.$$
		\begin{center}
			\begin{tikzpicture}[
				scale=0.93,
				transform shape,
				>=Stealth,
				line cap=round,
				line join=round
				]
				
				
				\coordinate (X) at (0,2.3);
				\coordinate (Y) at (0,-2.3);
				
				
				\draw[blue!50!black, dashed, thick]
				(X) -- (-1.15,4.1);
				
				\draw[blue!50!black, dashed, thick]
				(X) -- (1.15,4.1);
				
				\draw[blue!50!black, dashed, thick]
				(Y) -- (-1.15,-4.1);
				
				\draw[blue!50!black, dashed, thick]
				(Y) -- (1.15,-4.1);
				
				
				\draw[red!70!black, very thick, ->]
				(X) -- (0.1,4.0)
				node[above=2pt] {$U_{n+1}$};
				
				\draw[red!70!black, very thick, ->]
				(X) -- (0,0.45)
				node[midway,right=5pt] {$S^*_{II}$};
				
				\draw[red!70!black, very thick, ->]
				(Y) -- (0,-0.45)
				node[midway,right=5pt] {$S^*_{I}$};
				
				\draw[red!70!black, very thick, ->]
				(Y) -- (0.2,-4.0)
				node[below=2pt] {$V_{n+1}$};
				
				
				\draw[blue!70!black, very thick]
				(0,0.45) -- (0,-0.45);
				
				\node[blue!70!black, right=0.75cm] at (0,0)
				{$Z_n=X_n-Y_n$};
				
				
				\fill[blue!70!black] (X) circle (2.5pt);
				\fill[blue!70!black] (Y) circle (2.5pt);
				
				\node[blue!70!black, right=7pt] at (X)
				{$X_n$};
				
				\node[blue!70!black, right=7pt] at (Y)
				{$Y_n$};
				
				
				\draw[blue!70!black, thick]
				(X) ++(-0.32,0.50)
				arc[start angle=114,end angle=66,radius=0.59];
				
				\node[blue!70!black] at (0,3.05)
				{$2\theta$};
				
				
				\draw[blue!70!black, thick]
				(Y) ++(-0.32,-0.50)
				arc[start angle=246,end angle=294,radius=0.59];
				
				\node[blue!70!black] at (0,-3.05)
				{$2\theta$};
				
			\end{tikzpicture}
		\end{center}
		
		If the first biased coin lands heads with probability $\alpha_p$, then Player~$I$ (or equivalently Player~$II$) moves in the two game, and
		\begin{align*}
			|X_{n+1}|^2&\leq |X_n|^2 + 16 (p-1)^2 b^2 \epsilon^4 c^{2(1-\alpha_p)}.\\
			|Y_{n+1}|^2&\leq |Y_n|^2 + 16 (p-1)^2 b^2 \epsilon^4 c^{2(1-\alpha_p)}.\\
			|Z_{n+1}|^{2\beta}&\leq |Z_n|^{2\beta}.
		\end{align*}
		If both the first and the second biased coins land tails with probability $(1-\alpha_p)(1-\beta_p)$, then
		\begin{align*}
			|X_{n+1}|^2&\leq |X_n|^2 +4(p-1)(p+d)b \epsilon^2 c^{-\alpha_p} .\\
			|Y_{n+1}|^2&\leq |Y_n|^2 +4(p-1)(p+d)b \epsilon^2 c^{-\alpha_p} .\\
			|Z_{n+1}|^{2\beta}&= |Z_n|^{2\beta}.
		\end{align*}
		Finally, if the first biased coin lands tails and the second one lands heads, a round of tug of war is played. As before, it is clear that
		\begin{align*}
			|X_{n+1}|^2&\leq |X_n|^2 +4(p-1)(p+d)b \epsilon^2 c^{-\alpha_p} .\\
			|Y_{n+1}|^2&\leq |Y_n|^2 +4(p-1)(p+d)b \epsilon^2 c^{-\alpha_p} .
		\end{align*}
		However, since with probability $1/2$ the distance between the points increases by the quantity $4\sqrt{(p-1)(p+d)b}\,\epsilon c^{-\alpha_p/2}$, and with probability $1/2$ it decreases by the same amount, it follows that
		\begin{equation}\label{obs-1}
			\mathbb{E}[D_{n+1}\mid \mathcal{F}_n]=0\qquad \mbox{and}\qquad \mathbb{E}[|D_{n+1}|^2\mid \mathcal{F}_n]\geq C b \epsilon^2 c^{-\alpha_p}.
		\end{equation}
		Moreover, since the movements are aligned with $\widehat{Z_n}$ and the angle $\theta_0$ is sufficiently small, it is straightforward to verify that
		\begin{equation}\label{obs-2}
			\mathbb{E} [|D_{n+1}\cdot \widehat{Z_n}|^2\mid \mathcal{F}_n]\geq \frac{3}{4}\mathbb{E}[|D_{n+1}|^2 \mid \mathcal{F}_n].
		\end{equation}
		Therefore, using the Taylor expanssion for $|Z_{n+1}|^{2\beta}$
		\begin{align*}
			\mathbb{E}[|Z_{n+1}|^{2\beta}\mid \mathcal{F}_n]&=|Z_n|^{2\beta}\left(1+\beta |Z_n|^{-2}\mathbb{E}[|D_{n+1}|^2 \mid \mathcal{F}_n]-2\beta (1-\beta)|Z_n|^{-2}\mathbb{E}[|D_{n+1}\cdot \widehat{Z_n}|^2 \mid \mathcal{F}_n]+O(\epsilon^3)\right)\\
			&\leq |Z_n|^{2\beta} -\frac{\beta}{4}|Z_n|^{2\beta-2}\mathbb{E}[|D_{n+1}|^2\mid \mathcal{F}_n]+O(\epsilon^3)\\
			&\leq |Z_n|^{2\beta}- \frac{C}{24} b \epsilon^2 c^{-\alpha_p},
		\end{align*}
		where we have used \eqref{obs-1} in the first and third line, and \eqref{obs-2} in the second line.
		
		Now, combining the preceding estimates in each case, we obtain
		\begin{align*}
			\mathbb{E}[M_{n+1}\mid \mathcal{F}_n]&\leq \alpha_p\left(M_n+32(p-1)^2 b^2 \epsilon^4 c^{2(1-\alpha_p)}\right) + (1-\alpha_p)(1-\beta_p)\left(M_n + 8(p-1)(p+d) b \epsilon^2 c^{-\alpha_p}\right)\\
			&\quad+ (1-\alpha_p) \beta_p \left(M_n + 8(p-1)(p+d) b \epsilon^2 c^{-\alpha_p} - K\frac{C}{24} b \epsilon^2 c^{-\alpha_p} \right)\\
			&\leq M_n.
		\end{align*}
		Where in the last inequality we have chosen $K$ large enough. Note that, $K$ is independent of $\epsilon$ since $b^2 \epsilon^4 c^{2(1-\alpha_p)}\leq b \epsilon^2 c^{-\alpha_p}$ for all $c\in[m_p,M_p]$ and $b\in[m_b,M_b]$. This is precisely where the modifications \eqref{m-M-2} and \eqref{m-M-b-2} play a crucial role.
		
		\vspace{6mm}
		\noindent\textit{Case 2: Unaligned case. }At least one of the moves $U_{n+1}$ or $V_{n+1}$ does not exhibit a substantial deviation in the direction of $\widehat{Z_n}$. 
			\begin{center}
			\begin{tikzpicture}[
				scale=0.93,
				transform shape,
				>=Stealth,
				line cap=round,
				line join=round
				]
				
				
				\coordinate (X) at (0,2.3);
				\coordinate (Y) at (0,-2.3);
				
				
				\draw[blue!50!black, dashed, thick]
				(X) -- (-1.15,4.1);
				
				\draw[blue!50!black, dashed, thick]
				(X) -- (1.15,4.1);
				
				\draw[blue!50!black, dashed, thick]
				(Y) -- (-1.15,-4.1);
				
				\draw[blue!50!black, dashed, thick]
				(Y) -- (1.15,-4.1);
				
				
				\draw[red!70!black, very thick, ->]
				(X) -- (2,1)
				node[above=2pt] {$U_{n+1}$};
				
				\draw[red!70!black, very thick, ->]
				(X) -- (0,0.45)
				node[midway,right=5pt] {$S^*_{II}$};
				
				\draw[red!70!black, very thick, ->]
				(Y) -- (0,-0.45)
				node[midway,right=5pt] {$S^*_{I}$};
				
				\draw[red!70!black, very thick, ->]
				(Y) -- (-1,-2.0)
				node[below=2pt] {$V_{n+1}$};
				
				
				\draw[blue!70!black, very thick]
				(0,0.45) -- (0,-0.45);
				
				\node[blue!70!black, right=0.75cm] at (0,0)
				{$Z_n=X_n-Y_n$};
				
				
				\fill[blue!70!black] (X) circle (2.5pt);
				\fill[blue!70!black] (Y) circle (2.5pt);
				
				\node[blue!70!black, right=7pt] at (X)
				{$X_n$};
				
				\node[blue!70!black, right=7pt] at (Y)
				{$Y_n$};
				
				
				\draw[blue!70!black, thick]
				(X) ++(-0.32,0.50)
				arc[start angle=114,end angle=66,radius=0.59];
				
				\node[blue!70!black] at (0,3.05)
				{$2\theta$};
				
				
				\draw[blue!70!black, thick]
				(Y) ++(-0.32,-0.50)
				arc[start angle=246,end angle=294,radius=0.59];
				
				\node[blue!70!black] at (0,-3.05)
				{$2\theta$};
				
			\end{tikzpicture}
		\end{center}
		
		This case is somewhat simpler, since the new distance between the points always decreases. The situations in which the first biased coin lands heads, or in which both the first and the second coins land tails, are handled analogously. Suppose, therefore, that we are in the tug-of-war case. Then using the Taylor expanssion of second order for $|Z_{n+1}|^{2\beta}$ we obtain
		\begin{align*}
			\mathbb{E}[|Z_{n+1}|^{2\beta} \mid \mathcal{F}_n]&= |Z_n|^{2\beta}\left(1+2\beta \mathbb{E}[\widehat{Z_n}\cdot \frac{D_{n+1}}{|Z_n|}\mid \mathcal{F}_n]+O(\epsilon^2)\right)\\
			&\leq |Z_n|^{2\beta}-2\beta (1-\cos\theta_0)2\sqrt{(p-1)(p+d)b}\epsilon c^{-\alpha_p/2} + O(,\epsilon^2)\\
			&\leq |Z_n|^{2\beta}-\frac{\sqrt{(p-1)(p+d)}}{300} b^{1/2} \epsilon c^{-\alpha_p/2}.
		\end{align*}
		Therefore, if $K$ is large enough we obtain
		\begin{align*}
			\mathbb{E}[M_{n+1}\mid \mathcal{F}_n]&\leq \alpha_p\left(M_n+32(p-1)^2 b^2 \epsilon^4 c^{2(1-\alpha_p)}\right) + (1-\alpha_p)(1-\beta_p)\left(M_n + 8(p-1)(p+d) b \epsilon^2 c^{-\alpha_p}\right)\\
			&\quad+ (1-\alpha_p) \beta_p \left(M_n + 8(p-1)(p+d) b \epsilon^2 c^{-\alpha_p} - K\frac{\sqrt{(p-1)(p+d)}}{300} b^{1/2} \epsilon c^{-\alpha_p/2} \right)\\
			&\leq M_n.
		\end{align*}
		
		\noindent\textit{Case 3: Player~$I$ (or Player~$II$) decides to remain ``almost stationary''. }This case is handled similarly to the aligned tug-of-war case. Indeed, the uniform distribution assigns the same probability to moving in either direction within the ball. Consequently, the new distance between the points either increases by $32 (p-1)^4 b^4 \epsilon^8$ or decreases by the same amount, so that \eqref{obs-1} holds once again. Moreover, by complementing the moves with the reflected vector, \eqref{obs-2} is also satisfied. Therefore, one may argue exactly as in Case~1 to obtain
		\begin{equation*}
			\mathbb{E}[M_{n+1}\mid \mathcal{F}_n]\leq M_n + 512 (p-1)^8 b^8 \epsilon^{16}-K\frac{C}{24} 256 (p-1)^4 b^4 \epsilon^{16}\leq M_n.
		\end{equation*}
		
		Once we have established that $M_n$ is a supermartingale for fixed $K$, we define the stopping time
		
		$$
		\tau=\inf\left\{n\geq 0 : |Z_n|<\eta \mbox{ or } |X_n|>1 \mbox{ or } |Y_n|>1\right\}.
		$$
		
		Then, by the optional stopping theorem,
		
		$$
		\mathbb{E}[M_\tau]\leq M_0\leq 2\delta^2+K\delta^{2\beta}.
		$$
		
		Since $M_\tau$ is nonnegative, it follows that
		\begin{equation*}
			\mathbb{P}[|X_\tau|^2+|Y_\tau|^2\geq 1]
			\leq \mathbb{P}[M_\tau\geq 1]
			\leq (K+2)\delta^{2\beta}.
		\end{equation*}
		
		We now turn to the value of the game. We have previously shown that
		
		$$
		\ue(x)\leq \sup_{S_I}\mathbb{E}_{S_I,S_{II}^*}[\ue(X_\tau)]
		\qquad\mbox{and}\qquad
		\ue(y)\geq \inf_{S_{II}}\mathbb{E}_{S_I^*,S_{II}}[\ue(Y_\tau)].
		$$
		
		Therefore,
		\begin{align*}
			\mathbb{E}[\ue(X_\tau)-\ue(Y_\tau)]
			&\leq \mathbb{P}[|X_\tau-Y_\tau|<\eta]
			\mathbb{E}[\ue(X_\tau)-\ue(Y_\tau)\mid |X_\tau-Y_\tau|<\eta]\\
			&\quad + \mathbb{P}[|X_\tau-Y_\tau|>\eta]
			\left(\sup_{B_1}\ue-\inf_{B_1}\ue\right)\\
			&\leq \sup_{\substack{z,t\in B_1\ |z-t|<\eta}}
			\left|\ue(z)-\ue(t)\right|
			+(K+2)\delta^{2\beta}
			\left(\sup_{B_1}\ue-\inf_{B_1}\ue\right).
		\end{align*}
		Taking the supremum over $S_I$ and the infimum over $S_{II}$, we obtain
		\begin{equation*}
			\ue(x)-\ue(y)
			\leq \sup_{\substack{z,t\in B_1\ |z-t|<\eta}}
			\left|\ue(z)-\ue(t)\right|
			+(K+2)\delta^{2\beta}
			\left(\sup_{B_1}\ue-\inf_{B_1}\ue\right).
		\end{equation*}
		
		Since $\ue$ converges uniformly to $u$, we may pass to the limit and obtain
		\begin{equation*}
			u(x)-u(y)
			\leq \sup_{\substack{z,t\in B_1\ |z-t|<\eta}}
			\left|u(z)-u(t)\right|
			+(K+2)\delta^{2\beta}
			\left(\sup_{B_1}u-\inf_{B_1}u\right).
		\end{equation*}
		Finally, letting $\eta\to 0$ and using the uniform continuity of $u$, we obtain 
		\begin{equation*}
			u(x)-u(y)
			\leq (C+2)\delta^{2\beta}
			\left(\sup_{B_1}u-\inf_{B_1}u\right).
		\end{equation*}
		This concludes the proof of the Hölder continuity.
	\end{proof}
	
	To conclude this section, we show that a Harnack inequality holds in certain regimes of $d$, $n$, $p$, and $q$. This inequality was recently established in \cite{Byun+Kim-Harnack} by analytical methods for a class of equations that includes the one considered here. Once the Hölder estimate above is established, the result follows by applying the following lemma, which can be found in \cite[Lemma 3.1]{Peres+Wang_2026}.
	
	\begin{lem}
		Let $u$ be a positive continuous function on $B_5(0)$ with $u(0)=1$. Suppose that there exists $A\geq 1$ and $\gamma >0$ such that for all $0<r<R$ and $|x|<2$,
		\begin{equation}\label{condicion-1}
			\osc_{B_r(x)}u \leq A \left( \frac{r}{R} \right)^\gamma \osc_{B_R(x)}u.
		\end{equation}
		Assume also that there exists $C>0$ and $\lambda>0$ such that for all $r\leq 1$ and $|x|<2-r$,
		\begin{equation}\label{condicion-2}
			\inf_{B_r(x)}u\leq C r^{-\lambda}.
		\end{equation}
		Then,
		\begin{equation}\label{harnack}
			\sup_{B_1(0)} u \leq 4C \exp\left( \lambda \left(\frac{1}{\gamma}\log 4A+2 \log 2 \lambda\right)\right).
		\end{equation}
	\end{lem}
	
	\begin{lem}
		Let the assumptions of Lemma \ref{lema-Holder} hold. Let $u$ be a positive continuous solution to \eqref{eq_p-q} with $f\equiv 0$ on $B_5(0)$. Then, $u$ satisfies the Harnack Inequality,
		\begin{equation*}
			\sup_{B_1(0)} u \leq H \inf_{B_1(0)} u
		\end{equation*}
	\end{lem}
	\begin{proof}
		By scaling and translation it is enough to check that a positive function $u$, which is solution of \eqref{eq_p-q} in the domain $B_5(0)$, normalized with $u(0) = 1$ satisfies the conditions of the previous lemma.
		
		On the one hand, condition \eqref{condicion-1} is verified by Lemma \ref{lema-Holder}.
		
		On the other hand, condition \eqref{condicion-2} follows using comparison and the maximum principle. Fix $z$ and $r\leq 1$ such that $|z|<2-r$. We may assume that $0\not \in B_r(z)$ since otherwise the estimate automatically holds.
		
		We now consider the radial function
		$$
		v(x)=c\left(9-|x-z|^2\right),
		\qquad
		x\in\Omega:=B_3(z)\setminus\overline{B_r(z)},
		$$
		where $c>0$ is chosen so that $v(0)=2$. Observe that $v=0$ in $\partial B_3(z)$ and $v>0$ in $\Omega$.
		
		We claim that $v$ is a supersolution of equation \eqref{eq_p-q}. A direct calculation yields
		\begin{align*}
			\plap^{x'}v(x)+\qlap^{x''}v(x)
			&=-(2c)^{p-1}|x'-z'|^{p-2}(n+p-2)\\
			&\quad -(2c)^{q-1}|x''-z''|^{q-2}(d-n+q-2)\\
			&\leq 0.
		\end{align*}
		
		Suppose that $u\geq v$ on $\partial B_r(z)$. Since $u\geq v$ also holds on $\partial B_3(z)$, the comparison principle implies that
		$$
		u\geq v
		\qquad\text{in }B_3(z)\setminus\overline{B_r(z)}.
		$$
		However, this leads to a contradiction, since
		$$
		u(0)=1<2=v(0).
		$$
		Consequently,
		$$
		\inf_{\partial B_r(z)}u
		\leq
		\frac{2}{9-|z|^2}(9-r^2)
		\leq \frac{18}{5},
		$$
		where the last inequality follows from $|z|\leq 2$.
		
		This yields an even stronger estimate. In particular, we may choose any $\lambda>0$ in \eqref{condicion-2}, since $r\leq 1$. Finally, \eqref{condicion-2} follows from the fact that $u$ solves the equation in $B_r(z)$, in particular $u$ attains its infimum on the boundary by the maximum principle.
	\end{proof}


	\section{Parabolic equation}\label{sec7}
	
	Starting from the asymptotic expansion for equation \eqref{eq-parabolic} established in Corollaries \ref{exp-c2-parabolic} and \ref{exp-c3-parabolic}, the aim of this section is to study the DPP and the game associated with the heat equation involving the $p$-Laplacian \eqref{parabolic-problem}.
	
	Remember that, in this setting, it is not necessary to modify the operator by replacing $\ph(x)$ with its average over a ball of radius $\epsilon^8$ in $\A_\epsilon$. Indeed, although the players are allowed to remain stationary, the time variable decreases by $\epsilon^2$ at each step, see \eqref{DPP-parab}, which guarantees that the game terminates in finite time.
	
	\subsection{Dynamic Programming Principle for the $p$ Laplacian heat equation}\label{sec7.1}
	
	Given $\epsilon>0$, the mean value operator $\B_\epsilon^p$ defined in~\eqref{operator-B} induces the following Dynamic Programming Principle \eqref{DPP-parab} associated with~\eqref{eq-parabolic}. For the reader's convenience, we write it below
	\begin{equation}\label{DPP-parabolic}
		\left\{\begin{array}{cll} \ue(x,t+\epsilon^2)=\B_\epsilon^p[\ue](x,t), &\mbox{in }\Omega\times (-\epsilon^2,\infty),
			\\ \ue(x,t)=u_0(x), &\mbox{in } \Omega\times (-\epsilon^2,0],
			\\ \ue(x,t)=g(x), &\mbox{in }\Rd\setminus \Omega\times (-\epsilon^2,\infty).\end{array}\right.
	\end{equation}
	where $\Omega\subset\Rd$ is a bounded smooth domain, and the initial and exterior data belong to the space of bounded continuous functions, namely, $g\in C_b(\Rd\setminus\Omega)$ and $u_0\in C_b(\overline{\Omega})$. Note that the initial datum is extended to negative times. This last is an equivalent way of stating this problem. Depending on the computations to be carried out, we will choose the formulation of the DPP that is most convenient.
	
	 Our goal is to prove the existence and uniqueness of  solutions and to show that, as $\epsilon\to0$, it converges to the viscosity solution of the parabolic problem~\eqref{parabolic-problem}. The existence and uniqueness of solutions to this parabolic DPP are established via an explicit Euler time-discretization scheme, namely,
	
	$$
	u_{\epsilon}^{j+1}=\A_\epsilon[u_{\epsilon}^j]
	\qquad\textup{in }\Omega.
	$$
	Here, for any initial time $t_0\in(-\epsilon^2,0]$ and $j\in\NN$, we have denoted $t_j=t_0+j\epsilon^2$ and $u_{\epsilon}^j=\ue(\cdot,t_j)$.
	
	The remaining properties of the DPP solutions $\ue$, such as comparison and convergence, can be established in a similar manner to those in~\cite{delTeso+Fuertes+Rossi_2026} for the doubly-nonlinear equation. We therefore state these results without providing the details of the proofs.
	
	\begin{thm}\label{comparison-II}
		Let $d\in\NN$, $p>2$, $\epsilon\in(0,1)$, and let $\Omega\subset\Rd$ be a bounded domain. Assume that \eqref{HT} and \eqref{HT-b} hold. Let $g_1,g_2\in C_b(\Rd\setminus\Omega)$ and $u_{0,1},u_{0,2}\in C_b(\overline{\Omega})$. Let $\ua,\ub:\Rd\times\RR\longrightarrow\RR$ be two bounded measurable functions such that
		$$\left\{\begin{array}{cll} \ua(x,t+\epsilon^2)\geq\B_\epsilon^p[\ua](x,t), &\mbox{in }\Omega\times (-\epsilon^2,\infty),\\
			\ua(x,t)\geq u_{0,1}(x), &\mbox{in } \Omega\times (-\epsilon^2,0],\\ \ua(x,t)\geq g_1(x), &\mbox{in }\Rd\setminus \Omega\times (-\epsilon^2,\infty).\end{array}\right.$$
		and
		$$\left\{\begin{array}{cll} \ub(x,t+\epsilon^2)\leq\B_\epsilon^p[\ub](x,t), &\mbox{in }\Omega\times (-\epsilon^2,\infty),\\
			\ub(x,t)\leq u_{0,2}(x), &\mbox{in } \Omega\times (-\epsilon^2,0],\\ \ub(x,t)\leq g_2(x), &\mbox{in }\Rd\setminus \Omega\times (-\epsilon^2,\infty).\end{array}\right.$$
		
		If $u_{0,1}\geq u_{0,2}$ in $\Omega$ and $g_1\geq g_2$ in $\Rd\setminus\Omega$, then:
		\begin{enumerate}[label=\alph*)]
			\item If $\ub(x,t_j)\leq\ua(x,t_j)$ for all $x\in\Rd$, then the ordering is preserved at the next time step, 
			$$
			\ub(x,t_{j+1})\leq\ua(x,t_{j+1}).
			$$
			
			\item Moreover, $$\ub(x,t)\leq\ua(x,t),\quad \mbox{ for all } (x,t)\in\Rd\times(-\epsilon^2,\infty).$$
		\end{enumerate}
	\end{thm}
	
	\begin{proof}
		\textit{(a).}
		On the one hand, if $x\in\Rd\setminus\Omega$, there is nothing to prove, since
		$$\ub(x,t)\leq g_2(x)\leq g_1(x)\leq\ua(x,t).$$
		
		On the other hand, let $x\in\Omega$ and assume that
		$\ub(x,t_j)\leq\ua(x,t_j)$ for some $j\in\NN$. Then, by~\cite[Lemma 4.5]{delTeso+Fuertes+Rossi_2026}, $\A_\epsilon^p[\ub](x,t_j)\leq\A_\epsilon^p[\ua](x,t_j),$
		and therefore, using Lemma~\ref{max-ineq} and the definition \eqref{operator-B} of the operator $\B_\epsilon^p$,
		\begin{equation*}
			\ub(x,t_{j+1})-\ua(x,t_{j+1})
			\leq
			\sup_{b\in[m_b,M_b]}
			\left\{
			\A_{\epsilon\sqrt{2(p-1)b}}^p[\ub](x,t_j)
			-
			\A_{\epsilon\sqrt{2(p-1)b}}^p[\ua](x,t_j)
			\right\}
			\leq0.
		\end{equation*}
		
		\textit{(b).}
		If $(x,t)\in(\Rd\setminus\Omega)\times\RR^+$ or $(x,t)\in\Omega\times\RR^-\cup\{0\}$, then
		$$
		\ub(x,t)\leq g_2(x)\leq g_1(x)\leq\ua(x,t)
		\qquad\mbox{and}\qquad
		\ub(x,t)\leq u_{0,2}(x)\leq u_{0,1}(x)\leq\ua(x,t),
		$$
		respectively.
		
		Finally, if $(x,t)\in\Omega\times\RR^+$, we write $t=t_0+j\epsilon^2$ with $t_0\in(-\epsilon^2,0]$
		and, applying part~(a) iteratively $j$ times, we obtain
		$$
		\ub(x,t_0)\leq\ua(x,t_0)
		\Rightarrow
		\ub(x,t_0+\epsilon^2)\leq\ua(x,t_0+\epsilon^2)
		\Rightarrow\cdots\Rightarrow
		\ub(x,t_0+j\epsilon^2)\leq\ua(x,t_0+j\epsilon^2).
		$$
		
	\end{proof}
	
	To construct the smooth barriers, we employ the same ones previously introduced in~\cite{delTeso+Fuertes+Rossi_2026}.
	\begin{lem}\label{bound-ue}
		Let $d\in \NN$, $p>2$, $\epsilon\in(0,1)$, and let $\Omega\subset\Rd$ be a bounded domain satisfying the uniform exterior ball condition. Let $u_0\in C_b(\Omega)$, $g\in C_b(\Rd\setminus\Omega)$, and assume that \eqref{HT} and \eqref{HT-b} hold. If $\ue$ is the solution of the DPP \eqref{DPP-parabolic}, then there exists a positive constant $C$, independent of $\epsilon$, such that
		$$\|\ue\|_{\LL^\infty(\Rd)}\leq C.$$
		Moreover, we have the following barriers:
		\begin{enumerate}[label=(\roman*)]
			\item\label{bound-g} For every $T>0$ and $x_0\in \partial\Omega$,
			\begin{equation*}
				\liminf_{\epsilon\to 0^+,\, y\to x_0,\, s\to T}\ue(y,s)\geq g(x_0),\qquad\mbox{and}\qquad
				\limsup_{\epsilon\to 0^+,\, y\to x_0,\, s\to T}\ue(y,s)\leq g(x_0).
			\end{equation*}
			
			\item\label{bound-u0} For all $x_0\in\Omega$,
			\begin{equation*}
				\liminf_{\epsilon\to 0^+,\, y\to x_0,\, s\to 0^+} \ue(y,s)\geq u_0(x_0),\qquad\mbox{and}\qquad
				\limsup_{\epsilon\to 0^+,\, y\to x_0,\, s\to 0^+} \ue(y,s)\leq u_0(x_0).
			\end{equation*} 
		\end{enumerate}
	\end{lem}
	It is straightforward to verify, by following the proof of~\cite[Lemma 4.9]{delTeso+Fuertes+Rossi_2026} and \cite[Lemma 4.10]{delTeso+Fuertes+Rossi_2026}, that the same lower and upper barriers ($\wb$ and $\wa$) also apply in this case.
	 Therefore, taking limits as $x\to x_0$ and $t\to T$ in \ref{bound-g} and taking limits as $x\to x_0$ and $t\to 0^+$ in \ref{bound-u0} yields the results.
	
	We can now establish the convergence exactly as in Section \ref{sec4}.
	\begin{thm}\label{convergence-parabolic}
		Let $d\geq 1$, $p>2$, $\epsilon\in(0,1)$, and let $\Omega\subset\Rd$ be a bounded domain satisfying the uniform exterior ball condition. Let $u_0\in C_b(\Omega)$, $g\in C_b(\Rd\setminus\Omega)$, and assume that \eqref{m-M} and \eqref{m-M-b} hold. Let $\ue$ be the solution of the DPP~\eqref{DPP-parabolic}. For all $x\in\overline{\Omega}$ and $t\geq 0$, we define
		\begin{align*}
			\ua(x,t):=\limsup_{\epsilon\to 0^+,\, y\to x,\, s\to t}\ue(y,s),\qquad\mbox{and}\qquad
			\ub(x,t):=\liminf_{\epsilon\to 0^+,\, y\to x,\, s\to t}\ue(y,s).
		\end{align*}
		Then $\ua(x,t)=\ub(x,t)$ for all $(x,t)\in\overline{\Omega}\times\RR^+$, and consequently the limit of $u_\epsilon$ as $\epsilon\to 0^+$ exists. This limit coincides with the viscosity solution $u$ of~\eqref{parabolic-problem}:
		\begin{equation}\label{convergence-1}
			\ue(x,t)\to u(x,t)=\ua(x,t)=\ub(x,t)
			\quad \mbox{uniformly in } \overline{\Omega}\times\RR^+
			\quad \mbox{as } \epsilon\to 0^+.
		\end{equation}
	\end{thm}

	\begin{proof}
		By construction, we have $\ub(x,t)\leq \ua(x,t)$. Hence, uniform convergence will be established once we prove the reverse inequality. This follows from the fact that $\ub(x,t)$ is a viscosity supersolution of \eqref{parabolic-problem}, while $\ua(x,t)$ is a viscosity subsolution, since the comparison principle holds; see \cite{Juutinen+Lindqvist+Manfredi_2001, Siltakoski-equivalence-solution}.
		
		Thanks to Lemma \ref{bound-ue}, if either $x\in\partial\Omega$ and $t>0$, or $x\in\Omega$ and $t=0$, then $\ub(x,t)\geq\ua(x,t)$ follows directly. It therefore remains only to verify this inequality in $\Omega\times\RR^+$.
		
		We first show that $\ua$ is a viscosity subsolution. Let $(x_0,t_0)\in \Omega\times \RR^+$ and let $\ph\in C^\infty(B_R(x_0)\times(t_0-R,t_0+R))$ be a test function such that $\ph(x_0,t_0)=\ua(x_0,t_0)$ and $\ph(x,t)>\ua(x,t)$ at all other points in the ball. Since $(x_0,t_0)$ is a strict maximum point of $\ua-\ph$, by \cite{Barles_Souganidis_1991} there exists a sequence $(\epsilon_n,y_n,t_n)\to(0^+,x_0,t_0)$ such that $(y_n,t_n)$ are ``quasi-maximum'' points of $u_{\epsilon_n}-\ph$, namely,
		\begin{equation}\label{quasi-max}
			u_{\epsilon_n}(x,t)-\ph(x,t)
			\leq u_{\epsilon_n}(y_n,t_n)-\ph(y_n,t_n)+e^{-1/\epsilon_n}.
		\end{equation}
		Notice that, for $n$ sufficiently large, $(y_n,t_n)\in\Omega\times\RR^+$, so we can use the asymptotic expansions of Corollary \eqref{exp-c3-parabolic}. Denote
		$$\xi_n:=u_{\epsilon_n}(y_n,t_n)-\ph(y_n,t_n).$$
		Using the fact that $u_{\epsilon_n}$ is a solution of the DPP, together with \eqref{quasi-max} and the asymptotic expansion, we obtain
		\begin{align*}
			0
			&=\frac{u_{\epsilon_n}(y_n,t_n)-\B_{\epsilon_n}^p[u_{\epsilon_n}](y_n,t_n-\epsilon_n^2)}{\epsilon_n^2} \\
			&\geq \frac{\ph(y_n,t_n)+\xi_n-\B_{\epsilon_n}^p[\ph+\xi_n+e^{-1/\epsilon_n}](y_n,t_n-\epsilon_n^2)}{\epsilon_n^2} \\
			&=\partial_t\ph(y_n,t_n-\epsilon_n^2)
			-\plap\ph(y_n,t_n-\epsilon_n^2)
			+E(\epsilon,\ph, y_n, t_n)
			-\frac{e^{-1/\epsilon_n}}{\epsilon_n^2}.
		\end{align*}
		Notice that the error term can be bounded independently of the point $(y_n,t_n)$; see formulas~\eqref{error-B-grad} and~\eqref{error-B-gr-not}. Indeed, from some point onward, all points in the sequence lie in a neighborhood of $(x_0,t_0)$. Taking limit as $n\to\infty$ and using the regularity of $\ph$, we conclude that $\ua$ is a viscosity subsolution:
		\begin{equation*}
			\partial_t\ph(x_0,t_0)\leq \plap\ph(x_0,t_0).
		\end{equation*}
		
		The proof that $\ub$ is a viscosity supersolution of~\eqref{parabolic-problem} is similar.
	\end{proof}
	
	\subsection{The asociated game}\label{sec7.2}
	
	The DPP studied in Section~\ref{sec7.1} is also related to a two-player zero-sum game which is already studied in Section \ref{sec2}. 
	
	This game described above generates a sequence of states
	$$
	P=\left\{(x_0,t_0), (x_1,t_1), \ldots , (x_\tau,t_\tau)\right\},
	$$
	where $(x_0,t_0)\in \Omega \times (0,\infty)$ is the initial state of the game, and $(x_\tau,t_\tau)$ is the final state, when either $x_\tau\in \Rd\setminus \Omega$ or $t_\tau\leq 0$. Observe that the game always terminates after a finite number of plays, since, for any $t_0>0$ and $\epsilon>0$, there exists $k\in \NN$ such that
	$t_k=t_0-k\epsilon^2\leq 0$.
	We write the final payoff function as
	
	$$
	\Phi(x_\tau,t_\tau)=
	\left\{
	\begin{array}{ll}
		u_0(x_\tau) & \mbox{if }(x_\tau,t_\tau)\in\Omega \times (-\infty,0],\\[4pt]
		g(x_\tau) & \mbox{if }(x_\tau,t_\tau)\in (\Rd\setminus \Omega)\times \RR.
	\end{array}
	\right.
	$$
	
	The notions of strategies for both players are introduced in Section~\ref{sec5}. Once the two players have fixed their strategies, $S_I$ and $S_{II}$, the evolution of the game depends only on the initial state and the outcomes of the coin tosses, which induce a probability measure on the space of sequences of states. With respect to this probability measure, we can compute the expected payoff obtained when playing according to the strategies $S_I$ and $S_{II}$, starting from $(x_0,t_0)$:
	
	$$
	\mathbb{E}_{S_I, S_{II}}^{(x_0,t_0)}
	\left(\Phi(x_\tau,t_\tau)\right).
	$$
	
	Since Player~I tries to maximize the expected payoff, while Player~II aims to minimize it, the values for Players~I and~II are defined, respectively, by
	
	$$
	u^\epsilon_I(x_0,t_0)
	=
	\sup_{S_{I}}
	\left(
	\inf_{S_{II}}
	\mathbb{E}_{S_I, S_{II}}^{(x_0,t_0)}
	(\Phi(x_\tau,t_\tau))
	\right)
	\quad \mbox{and}\quad
	u^\epsilon_{II}(x_0,t_0)
	=
	\inf_{S_{II}}
	\left(
	\sup_{S_{I}}
	\mathbb{E}_{S_I, S_{II}}^{(x_0,t_0)}
	(\Phi(x_\tau,t_\tau))
	\right).
	$$
	Notice that, trivially,
	$u^\epsilon_I\leq u^\epsilon_{II}$.
	We say that the game has a value if
	$u^\epsilon_I=u^\epsilon_{II}$.
	
	\begin{thm}\label{game-parabolic}
		Let the assumptions of Theorem \ref{convergence-parabolic} hold. Then, the parabolic game associated with $\B_\epsilon^p$ has a value characterized by the unique solution $\ue$ to the DPP \eqref{DPP-parabolic}; that is,
		$$
		\ue(x,t)=u_I^\epsilon(x,t)=u_{II}^\epsilon(x,t)
		\qquad\mbox{for }(x,t)\in \Rd\times \RR.
		$$
		
		In particular, the value of the game converges uniformly on $\overline{\Omega}\times\RR^+$ as $\epsilon\to 0^+$ to the viscosity solution of the parabolic problem \eqref{parabolic-problem}.
	\end{thm}
	\begin{proof}
		We already have $u_I^\epsilon(x_0,t_0)\leq u_{II}^\epsilon(x_0,t_0)$. Then, it remains only to prove the reverse inequality.
		
		\noindent \textit{Step 1: We show that $\ue\leq u_I^\epsilon$.}
		Fix any $\eta>0$. Given any strategy $S_{II}$ for Player~$II$, we construct a strategy $S_I^*$ for Player~$I$ by following the solution of the DPP \eqref{DPP-parabolic} starting from $(x_k,t_k)$.
		
		\begin{itemize}[leftmargin=*]
			\item \textbf{First toss of a fair coin, with probability $1/2$ of heads and probability $1/2$ of tails.}\
			If the outcome is heads, then Player~$I$ chooses a constant $b^*\in[m_b,M_b]$ such that
			\begin{equation*}
				\begin{split}
					& \A_{\epsilon\sqrt{2(p-1)b^*}}^p [\ue](x_k,t_{k+1})
					-2(p-2)\epsilon^2 (b^*)^{\frac{p-1}{p-2}}
					\\ &\geq
					\sup_{b\in[m_b,M_b]}
					\left\{
					\A_{\epsilon\sqrt{2(p-1)b}}^p [\ue](x_k,t_{k+1})
					-2(p-2)\epsilon^2 b^{\frac{p-1}{p-2}}
					\right\}\\
					&-\frac{\eta}{2^{k+2}}.
				\end{split}
			\end{equation*}
			Afterwards, Player~$I$ pays Player~$II$ the amount
			$2(p-2)\epsilon^2 (b^*)^{\frac{p-1}{p-2}}$,
			and a round of the game associated with
			$\A_{\epsilon\sqrt{2(p-1)b^*}}^p$
			is played.
			
			\item \textbf{Second toss of a fair coin, with probability $1/2$ of heads and probability $1/2$ of tails.}\\
 			If the outcome is heads, Player~$I$ decides to remain at the same position, $x_{k+1}=x_k$, if
			\begin{equation}\label{eq1-parabolic-game}
				\ue(x_k,t_{k+1})> \!\!\!\! \!\!
				\inf_{c\in[m_p,M_p]} \!\!
				\left\{
 				\alpha_p \!\!\!\!
				\sup_{B_{4(p-1)b^* \epsilon^2 c^{1-\alpha_p}}(x_k)} \!\!\! \!\!\!
				\ue(\cdot,t_{k+1})
				+
				(1-\alpha_p)
				M_{\epsilon\sqrt{4(p-1)b^*}c^{-\alpha_p/2}}
				[\ue](x_k,t_{k+1})
				\right\}.
			\end{equation}
			If the outcome is tails, Player~$II$ chooses any $c\in[m_p,M_p]$, and Player~$I$ determines the next position as follows:
			
			\item \textbf{Third toss of a biased coin, with probability $\alpha_p=(p-2)/(p-1)$ of heads and probability $1-\alpha_p=1/(p-1)$ of tails.}\\
			If the outcome is heads, Player~$I$ chooses the next position
			$x_{k+1}^*\in B_{4(p-1)b^*\epsilon^2 c^{1-\alpha_p}}(x_k)$
			such that
			\begin{equation}\label{eq2-parabolic-game}
				\ue(x_{k+1}^*,t_{k+1})
				\geq
				\sup_{B_{4(p-1)b^*\epsilon^2 c^{1-\alpha_p}}(x_k)}
				\ue(\cdot,t_{k+1})
				-\frac{\eta}{2^{k+1}}.
			\end{equation}
			
			If the outcome is tails, a tug-of-war with noise is played. In this case, Player~$I$ chooses the next position
			$\widehat{x}_{k+1}\in
			B_{\sqrt{8(p+d)(p-1)b^*}\epsilon c^{-\alpha_p/2}}(x_k)$
			such that
			\begin{equation}\label{eq3-parabolic-game}
				\ue(\widehat{x}_{k+1},t_{k+1})
				\geq
				\sup_{B_{\sqrt{8(p+d)(p-1)b^*}\epsilon c^{-\alpha_p/2}}(x_k)}
				\ue(\cdot,t_{k+1})
				-
				\frac{\eta}{2^{k+1}}.
			\end{equation}
			
			\item If, in the second toss of the fair coin, the outcome is tails and Player~$II$ decides to play, then Player~$I$ chooses a constant $c^*\in[m_p,M_p]$ such that
			\begin{equation}\label{eq4-parabolic-game}
				\begin{split}
					\alpha_p& \inf_{4(p-1)b^* \epsilon^2 (c^*)^{1-\alpha_p}}\ue(\cdot,t_{k+1})+(1-\alpha_p)M_{\sqrt{4(p-1)b^*}\epsilon (c^*)^{-\alpha_p/2}}[\ue](x_k,t_{k+1})
					\geq\\
					&\sup_{c\in[m_p,M_p]}\left\{\alpha_p \inf_{4(p-1)b^* \epsilon^2 c^{1-\alpha_p}}\ue(\cdot,t_{k+1})+(1-\alpha_p)M_{\sqrt{4(p-1)b^*}\epsilon c^{-\alpha_p/2}}[\ue](x_k,t_{k+1})\right\}-\frac{\eta}{2^{k+1}}.
				\end{split}
			\end{equation}
			Moreover, if a round of tug-of-war is played, Player~$I$ chooses $\widehat{x}_{k+1}$ as in~\eqref{eq3-parabolic-game}.
			
			\item If, in the first toss of the fair coin, the outcome is tails and Player~$II$ chooses any $b\in[m_b,M_b]$, then Player~$II$ pays Player~$I$ the amount $2(p-2)\epsilon^2 b^{\frac{p-1}{p-2}}$, and a round of the game generated by $\A_{\epsilon\sqrt{2(p-1)b}}^p$ is played. Player~$I$ chooses the next position $x_{k+1}$ according to the previous rules
			\eqref{eq1-parabolic-game}, \eqref{eq2-parabolic-game}, \eqref{eq3-parabolic-game}, and \eqref{eq4-parabolic-game},
			with $b$ in place of $b^*$.
		\end{itemize}
		Fix the strategies $S_I^*$ and $S_{II}$, and consider the sequence of random variables
		$$
		M_k=\ue(x_k,t_k)-\frac{\eta}{2^k}.
		$$
		Our next goal is to verify that $(M_k)_{k\geq 0}$ is a submartingale. This follows from the inequalities \eqref{eq1-parabolic-game}, \eqref{eq2-parabolic-game}, \eqref{eq3-parabolic-game}, and \eqref{eq4-parabolic-game}, together with the fact that $\ue$ is a solution of the DPP \eqref{DPP-parabolic}.
		\begin{align*}
			&\mathbb{E}^{(x_0,t_0)}_{S_I^*, S_{II}}[M_{k+1}\mid (x_k,t_k)]
			=\frac{1}{2}\Bigg(-2(p-2)\epsilon^2 (b^*)^{\frac{p-1}{p-2}}
			+\frac{1}{2}\max\Bigg[ \ue(x_k,t_{k+1})-\frac{\eta}{2^{k+1}},\\
			&\quad\quad \alpha_p \ue(x_{k+1}^*,t_{k+1}) +(1-\alpha_p)\frac{\beta_p}{2}\left(\ue(\widehat{x_{k+1}},t_{k+1})+\ue(x_{k+1}^{II},t_{k+1})\right)\\ &\quad\quad\quad+(1-\alpha_p)(1-\beta_p)\fint_{B_{\sqrt{8(p+d)(p-1)b^*}\epsilon c^{-\alpha_p/2}}(x_k)}\ue(y,t_{k+1}) \,\dy \Bigg]+\frac{1}{2}\min\Bigg[\ue(x_k,t_{k+1})-\frac{\eta}{2^{k+1}}\\
			&\quad\quad \alpha_p \inf_{B_{4(p-1)b^* \epsilon^2 (c^*)^{1-\alpha_p}}(x_k)}\ue(\cdot,t_{k+1})+(1-\alpha_p)M_{\sqrt{4(p-1)b^*}\epsilon(c^*)^{-\alpha_p/2}}[\ue](x_k,t_{k+1})\Bigg]\Bigg)\\
			&+ \frac{1}{2}\Bigg(2(p-2)\epsilon^2 b^{\frac{p-1}{p-2}}
			+\frac{1}{2}\max\Bigg[ \ue(x_k,t_{k+1})-\frac{\eta}{2^{k+1}},\\
			&\quad\quad \alpha_p \ue(x_{k+1}^*,t_{k+1}) +(1-\alpha_p)\frac{\beta_p}{2}\left(\ue(\widehat{x_{k+1}},t_{k+1})+\ue(x_{k+1}^{II},t_{k+1})\right)\\ &\quad\quad\quad+(1-\alpha_p)(1-\beta_p)\fint_{B_{\sqrt{8(p+d)(p-1)b}\epsilon c^{-\alpha_p/2}}(x_k)}\ue(y,t_{k+1}) \,\dy \Bigg]+\frac{1}{2}\min\Bigg[\ue(x_k,t_{k+1})-\frac{\eta}{2^{k+1}}\\
			&\quad\quad \alpha_p \inf_{B_{4(p-1)b \epsilon^2 (c^*)^{1-\alpha_p}}(x_k)}\ue(\cdot,t_{k+1})+(1-\alpha_p)M_{\sqrt{4(p-1)b}\epsilon(c^*)^{-\alpha_p/2}}[\ue](x_k,t_{k+1})\Bigg]\Bigg)\\
			&\geq \frac{1}{2}\left(\sup_{b\in[m_b,M_b]}\left\{\A_{\epsilon\sqrt{2(p-1)b}}^p[\ue](x_k,t_{k+1})-2(p-2)\epsilon^2 b^{\frac{p-1}{p-2}}\right\}-\frac{\eta}{2^{k+2}}-\frac{2\eta}{2^{k+2}}-(\alpha_p+(1-\alpha_p)\beta_p)\frac{\eta}{2^{k+2}}\right)\\
			&\quad + \frac{1}{2}\left( \!\! \inf_{b\in[m_b,M_b]}\left\{\A_{\epsilon\sqrt{2(p-1)b}}^p[\ue](x_k,t_{k+1})-2(p-2)\epsilon^2 b^{\frac{p-1}{p-2}}\right\}-\frac{\eta}{2^{k+2}}-\frac{2\eta}{2^{k+2}}-(\alpha_p+(1-\alpha_p)\beta_p)\frac{\eta}{2^{k+2}}\right)\\
			&= \B_\epsilon^p[\ue](x_k,t_k-\epsilon^2)-(3+\alpha_p+(1-\alpha_p)\beta_p)\frac{\eta}{2^{k+2}}\\
			&\geq \ue(x_k,t_k)-\frac{\eta}{2^k}.
		\end{align*}
		
		Once we have established that $(M_k)_{k\geq 0}$ is a submartingale, we apply the Optional Stopping Theorem (cf.~\cite{Williams}) to obtain
		$$
		\mathbb{E}^{(x_0,t_0)}_{S^*_I,S_{II}}
		[M_{\tau}]
		\geq
		M_0
		=
		\ue(x_0,t_0)-\eta.
		$$
		Since $S_{II}$ is an arbitrary strategy, we have
		$$
		\inf_{S_{II}} \mathbb{E}^{(x_0,t_0)}_{S^*_I,S_{II}}
		[\Phi(x_\tau,t_\tau)]
		=
		\inf_{S_{II}} \mathbb{E}^{(x_0,t_0)}_{S^*_I,S_{II}}
		[\ue(x_\tau,t_\tau)]
		\geq
		\inf_{S_{II}} \mathbb{E}^{(x_0,t_0)}_{S^*_I,S_{II}}
		[M_{\tau}]
		\geq
		\ue(x_0,t_0)-\eta,
		$$
		and therefore
		\begin{equation*}
			\ue(x_0,t_0)-\eta
			\leq
			\inf_{S_{II}}
			\mathbb{E}^{(x_0,t_0)}*{S_I^*, S*{II}}
			[\Phi(x_\tau,t_\tau)]
			\leq
			\sup_{S_I}
			\left(
			\inf_{S_{II}}
			\mathbb{E}^{(x_0,t_0)}*{S_I, S*{II}}
			[\Phi(x_\tau,t_\tau)]
			\right)
			=
			u_I^\epsilon(x_0,t_0).
		\end{equation*}
		The first step is completed by the arbitrariness of $\eta$.
		
		\noindent\textit{Step 2: We show that $u^{\epsilon}_{II}\leq \ue$.}
		The proof follows analogously to that of Step 1, since the game is symmetric. We therefore omit the computations and provide only a sketch of the argument. Let $\eta>0$ be arbitrary. Given any strategy $S_I$ for Player~$I$, we choose a strategy $S_{II}^*$ for Player~$II$ analogously to Step 1.
		
		Once the strategies $S_I$ and $S_{II}^*$ have been fixed, we define the sequence of random variables
		$$
		M_k=\ue(x_k,t_k)+\frac{\eta}{2^k}.
		$$
		It can be shown as before that $(M_k)_{k\geq 0}$ is a supermartingale. Therefore, by the Optional Stopping Theorem, we obtain
		$$
		\mathbb{E}^{(x_0,t_0)}_{S_I,S_{II}^*}[M_{\tau}]
		\leq
		M_0
		=
		\ue(x_0,t_0)+\eta,
		$$
		and hence
		\begin{equation*}
			\ue(x_0,t_0)+\eta
			\geq
			\sup_{S_I}
			\mathbb{E}^{(x_0,t_0)}*{S_I,S*{II}^*}
			[\Phi(x_\tau,t_\tau)]
			\geq
			\inf_{S_{II}}
			\left(
			\sup_{S_I}
			\mathbb{E}^{(x_0,t_0)}*{S_I,S*{II}}
			[\Phi(x_\tau,t_\tau)]
			\right)
			=
			u_{II}^\epsilon(x_0,t_0).
		\end{equation*}
		Since $\eta>0$ is arbitrary, we conclude that $\ue(x_0,t_0)\geq u_{II}^\epsilon(x_0,t_0)$. Therefore,
		$$
		\ue(x_0,t_0)
		=
		u_I^\epsilon(x_0,t_0)
		=
		u_{II}^\epsilon(x_0,t_0),
		$$
		which proves that the game has a value that coincides with the unique solution of the DPP \eqref{DPP-parabolic}.
		
		Since $\ue(x,t)\to u(x,t)$ uniformly in $\overline{\Omega}\times\RR^+$ as $\epsilon\to 0^+$, see Theorem~\ref{convergence-parabolic}, we conclude that the game has a value, which converges to $u$, the viscosity solution of \eqref{parabolic-problem}.
	\end{proof}

	
	\section{Final comments and possible extensions} \label{sect-final}
	
	In this final section we comment briefly on possible extensions of our results and include a brief remark
	on scaling invariance of our mean value formula for the $p-$Laplacian. 
	
	\subsection{Extensions}
	
	\begin{enumerate}
	
	\item With the same arguments used here we can tackle equations involving more than two $p-$Laplacians, like, for example,
	$$\Delta_p^{x'} u(x)+\Delta_q^{x''}u(x)+ \Delta_r^{x'''}u(x)+ ... =f(x).$$
	
	For instance, for the sum of three $p-$Laplacians in different sets of variables, 
	$$\Delta_p^{x'} u(x)+\Delta_q^{x''}u(x)+ \Delta_r^{x'''}u(x),$$ 
	we have the mean value formula
	\begin{equation*}
	\begin{array}{l}
	\displaystyle  
		\frac{1}{3}\B_\epsilon^p[\ph](x',x'',x''')
		+
		\frac{1}{3}\B_\epsilon^q[\ph](x',x'',x''')+ \frac{1}{3}\B_\epsilon^r[\ph](x',x'',x''') \\[6pt]
		\displaystyle =\ph (x',x'',x''')
		+\frac{\epsilon^2}{2}\left(\Delta_p^{x'} \ph (x',x'',x''')+\Delta_q^{x''}\ph (x',x'',x''')
		+ \Delta_r^{x'''} \ph (x',x'',x''')\right)+\epsilon^2\left(E (\epsilon,\ph,x))\right).
		\end{array}
	\end{equation*}

	\item We can also deal with equations involving a gradient term and a $p-$Laplacian, like, for example,
	$$\Delta_p u(x)+ |\nabla u (x) | =f(x).$$
	In fact, in this case we have 
	\begin{equation*}  
		(1-\epsilon) B_\epsilon^p[\ph](x) + \epsilon \left( \sup_{B_\epsilon (x)} \ph (y)\right)
		\displaystyle =\ph (x)
		+\frac{\epsilon^2}{2}\left(\Delta_p \ph (x)+ |\nabla \ph (x) | \right)+\epsilon^2\left(E (\epsilon,\ph,x))\right).
	\end{equation*}
	Here we only have to observe that 
	$$
	\epsilon \left( \sup_{B_\epsilon (x)} \ph (y) - \ph (x) \right) \approx \epsilon^2 |\nabla \ph (x) |.
	$$
	
	The case with the reverse sign,
	$$\Delta_p u(x) - |\nabla u (x) | =f(x),$$
	corresponds to 
	\begin{equation*}
			(1-\epsilon) B_\epsilon^p[\ph](x) + \epsilon \left( \inf_{B_\epsilon (x)} \ph (y)\right)
		\displaystyle =\ph (x)
		+\frac{\epsilon^2}{2}\left(\Delta_p \ph (x)+ |\nabla \ph (x) | \right)+\epsilon^2\left(E (\epsilon,\ph,x))\right).
		\end{equation*}
	\end{enumerate}
	
	It is worth to mention one more time that none of the previous equations is homogeneous.

	\subsection{Scaling invariance} Finally we observe that 
		the operator $\B_\epsilon^p$, which we have just shown asymptotically approximates the $p$-Laplacian (Theorems~\ref{exp-c2} and~\ref{exp-c3}), also satisfies the same scaling property as the $p$-Laplacian. More precisely, if $w(x)=u(\lambda x)$ for some $\lambda>0$, then
		\begin{equation*}
			\plap w(x)=\lambda^p \plap u(\lambda x)\qquad\mbox{and}\qquad \B_\epsilon^p[w](x)=\lambda^p \B_{\tilde{\epsilon}}^p[u](\lambda x).
		\end{equation*}
	
	This happens since when we compute $\B_\epsilon^p[w](x)$ we have (up to a constant) the supremum, infimum and integrals on the following balls
	$$B_{\epsilon^2 b c^{1-\alpha_p}}(x)=B_{\lambda \epsilon^2 b c^{1-\alpha_p}}(\lambda x)\qquad \mbox{and}\qquad B_{\epsilon b^{1/2} c^{-\alpha_p/2}}(x)=B_{\lambda \epsilon b^{1/2} c^{-\alpha_p/2}}(\lambda x).$$
	Thus, what we need is to distribute $\lambda$ as follows:
	\begin{align*}
		\lambda \epsilon^2 b c^{1-\alpha_p}&= \tilde{\epsilon}^2 \tilde{b} \tilde{c}^{1-\alpha_p}.\\
		\lambda \epsilon b^{1/2} c^{-\alpha_p/2}&= \tilde{\epsilon} \tilde{b}^{1/2} \tilde{c}^{-\alpha_p/2}.\\
		\epsilon^2 b^{\frac{p-1}{p-2}}&= \tilde{\epsilon}^2 \tilde{b}^{\frac{p-1}{p-2}}.
	\end{align*}
	This is satisfied if and only if
	$$\lambda^{-\frac{1}{p-1}}c^{1-\alpha_p}=\tilde{c}^{1-\alpha_p},\qquad \lambda^p\epsilon^2=\tilde{\epsilon}^2\qquad \mbox{ and }\qquad \lambda^{\frac{p}{p-1}-p}b=\tilde{b}.$$
	Therefore, we conclude that
	\begin{equation*}
		\frac{\B_\epsilon^p[w](x)-w(x)}{\epsilon^2}=\lambda^p\frac{\B_{\tilde{\epsilon}}^p[u](\lambda x)}{\tilde{\epsilon}^2}.
	\end{equation*}

	
	\appendix
	
	\section{Technical results}\label{apenA}
	
	In this appendix, we focus exclusively on proving several of the technical results used throughout the paper. We have chosen to collect these proofs here in order to preserve the flow of the exposition, since they do not contribute any additional conceptual insight to the main results and consist primarily of necessary computations.
	
	We begin with Lemma~\ref{signed-power-lemma}, where we express the signed power of a real number as the sum of a supremum and an infimum.
	\begin{lem}\label{lem-A1}
		Let $a\geq 0$ and $r>1$. Then,
		$$a^r=\sup_{b>0}\left\{r a b -(r-1)b^{\frac{r}{r-1}}\right\}.$$
		Moreover, we can restrict the set of the supremum to obtain
		$$a^r=\sup_{b\in[m,M]}\left\{rab-(r-1)b^{\frac{r}{r-1}}\right\}+E(a,r,m,M),$$
		with
		\begin{equation}\label{error-1-apenA}
			\left|E(a,r,m,M)\right|\leq r m^{\frac{r}{r-1}}+r a \left(a^{r-1}-M\right)_+.
		\end{equation}
	\end{lem}
	\begin{proof}
		The first part is proved as follows. Consider the function
		$$
		\Phi(b)=rab-(r-1)b^{\frac{r}{r-1}}, \qquad b>0.
		$$
		We look for its maximum, which is attained at $b=a^{r-1}$ and has value $a^r$.
		
		Let show how can we estimate the error in the case when we restrict the set. We split the proof of the second part in several cases:
		
		\noindent \textit{Case 1: $m\leq a^{r-1}\leq M$. }Then,
		\begin{equation*}
			\sup_{b>0}\left\{rab -(r-1)b^{\frac{r}{r-1}}\right\}=\sup_{b\in[m,M]}\left\{rab -(r-1)b^{\frac{r}{r-1}}\right\}\qquad\mbox{and}\qquad E(a,r,m,M)=0.
		\end{equation*}
		
		\noindent \textit{Case 2: $a^{r-1}<m$. }Then, since function $\Phi$ is concave we have
		$$\sup_{b\in[m,M]}\left\{rab-(r-1)b^{\frac{r}{r-1}}\right\}=ram-(r-1)m^{\frac{r}{r-1}},$$
		and
		\begin{equation*}
			\left|E(a,r,m,M)\right|= a^r -ram +(r-1)m^{\frac{r}{r-1}}
			\leq r m^{\frac{r}{r-1}}.
		\end{equation*}
		
		\noindent \textit{Case 3: $a^{r-1}>M$. }Then, since function $\Phi$ is concave we have
		$$\sup_{b\in[m,M]}\left\{rab-(r-1)b^{\frac{r}{r-1}}\right\}=ra M -(r-1)M^{\frac{r}{r-1}},$$
		and 
		\begin{equation*}
			\left|E(a,r,m,M)\right|= a^r -raM +(p-1)M^{\frac{r}{r-1}}
			\leq ra (a^{r-1}-M).
		\end{equation*}
		Therefore, by combining the error terms arising in each case, the proof is complete.
	\end{proof}
	
	\begin{lem}
		Let $a\in\RR$ and $r>1$. Then, 
		\begin{equation*}
			a^r= \operatorname{sgn}(a)\,|a|^r
			=
			\sup_{b>0}
			\left\{
			rab-(r-1)b^{\frac{r}{r-1}}
			\right\}
			+
			\inf_{b>0}
			\left\{
			rab+(r-1)b^{\frac{r}{r-1}}
			\right\}.
		\end{equation*}
		Moreover, if we restrict the supremum and infimum to $b\in[m,M]$, then
		$$
		a^r
		=
		\sup_{b\in[m,M]}
		\left\{
		rab-(r-1)b^{\frac{r}{r-1}}
		\right\}
		+
		\inf_{b\in[m,M]}
		\left\{
		rab+(r-1)b^{\frac{r}{r-1}}
		\right\}
		+
		E(a,r,m,M),
		$$
		where
		\begin{equation}\label{error-2-apenA}
			\left|E(a,r,m,M)\right|
			\leq
			2r
			\left(
			m^{\frac{r}{r-1}}
			+
			m\left|a\right|
			+
			\left|a\right|
			\left(\left|a\right|^{r-1}-M\right)_+
			\right).
		\end{equation}
	\end{lem}
	\begin{proof}
			We split the proof into several steps.
		
		\noindent\textit{Step 1: The power of a positive or negative real number.}
		We first recall that, when $a\geq 0$, we have
		$$
		a^r=\sup_{b>0}\left\{rab-(r-1)b^{\frac{r}{r-1}}\right\}.
		$$
		Similarly, when $a<0$, applying Lemma \ref{lem-A1} to $-a$ yields
		$$
		a^r=-(-a)^r
		=\inf_{b>0}\left\{rab+(r-1)b^{\frac{r}{r-1}}\right\}.
		$$
		
		\noindent\textit{Step 2: Positive and negative parts.}
		We now consider the case in which the parameter $a$ has the opposite sign to that required in the expressions above. In this case, the corresponding supremum and infimum vanish, yielding the positive and negative parts, respectively. Namely,
		$$(a^r)_+=\sup_{b\in[m,M]}\left\{rab-(r-1)b^{\frac{r}{r-1}}\right\}\qquad \mbox{and}\qquad (a^r)_-=-\inf_{b\in[m,M]}\left\{rab+(r-1)b^{\frac{r}{r-1}}\right\}.$$
		
		Therefore, regardless of the sign of $a$, we obtain
		$$
		a^r
		=
		\sup_{b\in[m,M]}
		\left\{
		rab-(r-1)b^{\frac{r}{r-1}}
		\right\}
		+
		\inf_{b\in[m,M]}
		\left\{
		rab+(r-1)b^{\frac{r}{r-1}}
		\right\}.
		$$
		
		\noindent\textit{Step 3: Estimation of the error.}
		We first consider the error arising from restricting the supremum. On the one hand, by Lemma \ref{lem-A1}, if $a\geq 0$, the error is given by \eqref{error-1-apenA}. On the other hand, if $a<0$, the supremum is attained at $m$, and the distance between its value and $0$ is
		$$
		r|a|m+(r-1)m^{\frac{r}{r-1}}.
		$$
		
		Combining these two estimates gives
		\begin{equation*}
			\left|E(a,r,m,M)\right|
			\leq
			r
			\left(
			m^{\frac{r}{r-1}}
			+
			m|a|
			+
			|a|
			\left(|a|^{r-1}-M\right)_+
			\right).
		\end{equation*}
		
		The same argument applies to the infimum, accounting for the factor of $2$ in the final estimate \eqref{error-2-apenA}.
	\end{proof}
	Observe that the error tends to $0$ as $m\to 0^+$ and $M\to \infty$.
	
	We next turn to the technical results from Section~\ref{sec4.1}, where we estimate the difference of maxima and the quantity $\A_\epsilon^p[\ub]-\A_\epsilon^p[\ua]$ under the assumption that the functions are ordered. We do not have to prove Lemma \ref{max-ineq} since it is already done in \cite{delTeso+Fuertes+Rossi_2026}.
	
	\begin{lem}\label{lemma-A-gorro}
		Let $d\in \NN$, $p>2$, and $\epsilon\in (0,1)$. Assume \eqref{HT} and let $\ub,\ua:\Rd\longrightarrow \RR$ are two measurable functions such that $\ub\leq \ua$ in $\Rd$. Then, for any $x_0\in \Rd$ we have
		\begin{equation}
			\widehat{\A_\epsilon^p}[\ub](x_0)-\widehat{\A_\epsilon^p}[\ua](x_0)\leq   \fint_{B_{\epsilon^8}(x_0)}\ub(y)-\ua(y)\dy
		\end{equation}
	\end{lem}
	\begin{proof}
		Using the previous lemma we directly obtain
		\begin{align*}
			\widehat{\A_\epsilon^p}[\ub](x_0)-\widehat{\A_\epsilon^p}[\ua](x_0)
			&\leq \max\Bigg[
			\sup_{c\in[m_p,M_p]}\Bigg\{
			\alpha\Big(
			\sup_{B_{2\epsilon^2 c^{1-\alpha_p}}(x_0)}\ub
			-
			\sup_{B_{2\epsilon^2 c^{1-\alpha_p}}(x_0)}\ua
			\Big) \\
			&\qquad\qquad
			+ (1-\alpha)\Big(
			M_{\epsilon c^{-\alpha_p/2}}[\ub](x_0)
			-
			M_{\epsilon c^{-\alpha_p/2}}[\ua](x_0)
			\Big)
			\Bigg\}, \\
			&\qquad
			\fint_{B_{\epsilon^8}(x_0)} \ub(y)-\ua(y)\,dy
			\Bigg] \\
			&\leq \max\Bigg[
			\frac{d+2}{(p-1)(d+p)}
			\fint_{B_{\sqrt{2(p+d)}\epsilon M_p(\epsilon)^{-\alpha_p/2}}(x_0)}
			\ub(y)-\ua(y)\,dy, \\
			&\qquad
			\fint_{B_{\epsilon^8}(x_0)} \ub(y)-\ua(y)\,dy
			\Bigg].
		\end{align*}
		where in the last line we have used the definition of $M_\rho[\ph](x)$, see Lemma \ref{exp-plapN}, and the fact $\ub\leq \ua$ in any ball. The final result follows since $\epsilon^3<\sqrt{2(p+d)}\epsilon M_p(\epsilon)^{-\alpha_p/2}$.
	\end{proof}

	Finally, in Section~\ref{sec4.2}, it is necessary to compute the $p$-Laplacian with respect to the first $n$ variables of the radial function $|x-z|^{-\alpha}$.
	\begin{lem}
		Let $1\leq n<d$, $p>2$, and denote $x=(x',x'')\in \Rd$ with $x'=(x_1,\cdots,x_n)$ and $x''=(x_{n+1},\cdots,x_d)$. Define $v(x)=|x-z|^{-\alpha}$ for any $\alpha>0$ and $z\in \Rd$. Then, for any $x\not=z$ we have
		\begin{equation*}
			\plap^{x'} v(x)=\alpha^{p-1}\frac{|x'-z'|^{p-2}}{|x-z|^{(\alpha+2)(p-1)+2}}\left[\alpha(p-1)|x'-z'|^2+(p-n)|x'-z'|^2-(n+p-2)|x''-z''|^2\right].
		\end{equation*}
	\end{lem}
	\begin{proof}
		We first compute the gradient of $v$ in its first $n$ variables, that is
		$$\gr_{x'} v= -\alpha \left|x-z\right|^{-(\alpha+2)} (x'-z'),$$
		Therefore,
		$$\gr_{x'} v \left|\gr_{x'}v\right|^{p-2}=-\alpha^{p-1}\left|x-z\right|^{-(\alpha+2)(p-1)}\left|x'-z'\right|^{p-2}(x'-z').$$
		
		Finally, we compute the divergence of this last quantity
		\begin{align*}
			\plap^{x'} v(x)&=-\alpha^{p-1} n \frac{\left|x'-z'\right|^{p-2}}{\left|x-z\right|^{(\alpha+2)(p-1)}} - \alpha^{p-1}(p-2) \frac{\left|x'-z'\right|^{p-2}}{\left|x-z\right|^{(\alpha+2)(p-1)}} + \alpha^{p-1}(\alpha+2)(p-1)\frac{\left|x'-z'\right|^p}{\left|x-z\right|^{(\alpha+2)(p-1)+2}}\\
			&=\alpha^{p-1}\frac{\left|x'-z'\right|^{p-2}}{\left|x-z\right|^{(\alpha+2)(p-1)+2}}\left[-(n+p-2) \left|x-z\right|^2+(\alpha+2)(p-1)\left|x'-z'\right|^2\right]
		\end{align*}
		The result follows since $|x-z|^2=|x'-z'|^2+|x''-z''|^2$.
	\end{proof}
	Note that an analogous calculation shows us that
	\begin{equation*}
	\begin{array}{l}
	\displaystyle 
		\qlap^{x''} v(x)=\alpha^{q-1}\frac{|x''-z''|^{q-2}}{|x-z|^{(\alpha+2)(q-1)+2}}\Big[\alpha(q-1)|x''-z''|^2 \\[8pt]
		\qquad \qquad \qquad\qquad\qquad  \displaystyle +(q-(d-n))|x''-z''|^2-(d-n+q-2)|x'-z'|^2\Big].
		\end{array}
	\end{equation*}
	
	
	\section{Viscosity solutions}\label{apenB}
	
	We devote this appendix to discuss the notion of viscosity solutions for the problems under donsideration in this article. 
	We refer to \cite{Crandall_Ishii_Lions_1992} for a general reference on the topic.
	
	Let us start with the definition of a viscosity solution to the equation
	\begin{equation}\label{equation.A1}
		\plap^{x'} u(x',x'')+\qlap^{x''}u(x',x'')=f(x).
	\end{equation}
	\begin{defn}[\textbf{Viscosity Solution Elliptic I}]\label{viscosity-equation-elliptic}
		Let $d\in \NN$, $p>q>2$, $\Omega\subset \Rd$ be a bounded domain, and denote $x=(x',x'')\in \Rd$, with $x'=(x_1,\cdots,x_n)$ and $x''=(x_{n+1},\cdots,x_d)$. We say that a bounded upper semicontinuous function $u:\Omega\longrightarrow \RR$ is a viscosity subsolution of equation \eqref{equation.A1} in $\Omega$ if, whenever $x_0\in \Omega$ and $\ph\in C^2(B_R(x_0))$ such that $\ph(x_0)=u(x_0)$, and $\ph(x)>u(x)$ for all $x\in B_R(x_0)\setminus\left\{x_0\right\}$ then,
		$$\plap^{x'} u(x_0',x_0'')+\qlap^{x''}u(x_0',x_0'')\geq f(x_0).$$
		Similarly, a bounded lower semicontinuous function $u$ is a viscosity supersolution of \eqref{equation.A1} in $\Omega$ if, under the same assumptions but with $\ph(x_0)=u(x_0)$ and $\ph(x)<u(x)$ in $B_R(x_0)$, we have
		$$\plap^{x'} u(x_0',x_0'')+\qlap^{x''}u(x_0',x_0'')\leq f(x_0).$$
		Moreover, a bounded continuous function $u$ is a viscosity solution of \eqref{equation.A1} if it is both a viscosity subsolution and a viscosity supersolution.
	\end{defn}
	
	This notion of solution is equivalent to the following one in which our mean value formula appears, see Section \ref{sec3}. 
	
	\begin{defn}[\textbf{Viscosity Solution Elliptic II}]\label{viscosity-operator-elliptic}
		A bounded upper semicontinuous (resp.\ lower semicontinuous) function $u$ is a viscosity subsolution (resp.\ supersolution) associated with the operator $\widehat{\C_\epsilon}$ if the following holds: for every $(x_0)\in \Omega$ and every test function $\ph\in C^2(B_R(x_0))$ for some $R>0$, such that $\ph(x_0)=u(x_0)$ and $\ph(x)>u(x)$ (resp.\ $\ph(x)<u(x)$) in $B_R(x_0)$, we have
		\begin{equation*}
			\begin{split}
				&\widehat{\B_\epsilon^p}[\ph](x_0',x_0'')+\widehat{\B_\epsilon^q}[\ph](x_0',x_0'')-2\ph(x_0)\geq \epsilon^2 f(x_0)+o(\epsilon^2),
				\quad \text{as }\epsilon\to 0^+.\\
				&\qquad(\mbox{resp. } \widehat{\B_\epsilon^p}[\ph](x_0',x_0'')+\widehat{\B_\epsilon^q}[\ph](x_0',x_0'')-2\ph(x_0)\leq \epsilon^2 f(x_0)+o(\epsilon^2)).
			\end{split}
		\end{equation*}
	\end{defn}
	
	Here by $\ph(x_0,t_0+\epsilon^2)\geq \B_\epsilon^p[\ph](x_0,t_0)+o(\epsilon^2)$ (resp. $\ph(x,t+\epsilon^2)\leq \B_\epsilon^p[\ph](x,t)+o(\epsilon^2) $) we mean that
	$$
	0\geq \limsup_{\epsilon \to 0} \frac{ \B_\epsilon^p[\ph](x_0,t_0) - \ph(x_0,t_0+\epsilon^2)}{\epsilon^2},
	\qquad
	\left(\mbox{resp. } 0\leq \liminf_{\epsilon \to 0} \frac{ \B_\epsilon^p[\ph](x_0,t_0) - \ph(x_0,t_0+\epsilon^2)}{\epsilon^2}\right).
	$$
	
	Now, let us state the definition of being a viscosity solution to the boundary value problem
	\begin{equation}\label{problem-A-1}
		\left\{\begin{array}{rlll}
			\plap^{x'}u(x',x'')+\qlap^{x''}u(x',x'') &=& f(x), & \mbox{in } \Omega,\\
			u(x)&=&g(x), & \mbox{on } \partial \Omega.
		\end{array}\right.
	\end{equation}
	\begin{defn}[\textbf{Viscosity Solution Elliptic III}]\label{viscosity-dirichlet-elliptic}
		A bounded upper semicontinuous (resp.\ lower semicontinuous) function $u:\overline{\Omega}\to \RR$ is a viscosity subsolution (resp.\ supersolution) to the Dirichlet problem \eqref{problem-A-1} if:
		\begin{enumerate}[label=\roman*)]
			\item $u$ is a viscosity subsolution (resp.\ supersolution) of \eqref{equation.A1} in the sense of Definition~\ref{viscosity-equation-elliptic}.
			\item  $u(x)\leq g(x)$ (resp.\ $u(x)\geq g(x)$) in $\partial \Omega$.
		\end{enumerate}
		
		Moreover, a bounded continuous function $u$ is a viscosity solution of \eqref{problem-A-1} if it is both a viscosity subsolution and a viscosity supersolution.
	\end{defn}
	
	To the best of our knowledge, there is no equivalence result between viscosity and weak solutions for this elliptic problem without imposing additional regularity assumptions on the solutions. See~\cite{Ochoa+Valvelde_2024, Yuzhou+Ruadulescu+Zhang_2024} for a result in a more general framework under the assumption that the solutions are Lipschitz continuous.
	
	Nevertheless, the standard doubling variables method, see~\cite{Crandall_Ishii_Lions_1992}, yields directly the comparison principle for viscosity solutions (and consequently their uniqueness). We provide here a brief sketch of the proof in $d=2$.
	\begin{lem}\label{com-vis-elliptic}
		Let $u$ be an upper semicontinuous function and let $v$ be lower semicontinuous. If they are a viscosity subsolution and supersolution of \eqref{problem-A-1} respectively, then
		$$u\leq v \qquad \mbox{in }\RR^2.$$
	\end{lem}
	\begin{proof}
	It is clear that $u\leq v$ in $\mathbb{R}^2\setminus\Omega$, and hence it suffices to prove the inequality in $\Omega$. To maintain the notation of~\cite{Crandall_Ishii_Lions_1992}, we rewrite the equation in $\Omega$ as
	$$
	f(x,y)-\plap^x u-\qlap^y u=0.
	$$
	Assume, by contradiction, that $M:=\sup_{\Omega}\{u-v\}>0$. We divide the remainder of the proof into several steps.
	
	\noindent\textit{Step 1:}
	For every $\epsilon>0$, define
	$$
	\Phi_\epsilon(x,y,x',y')
	=
	u(x,y)-v(x',y')
	-\frac{(x-x')^2+(y-y')^2}{2\epsilon},
	$$
	and let $M_\epsilon:=\Phi_\epsilon(x_\epsilon,y_\epsilon,x'_\epsilon,y'_\epsilon)$ denote the supremum of $\Phi_\epsilon$. We first show that
	$$
	M_\epsilon\to M
	\qquad\mbox{and}\qquad
	\left|(x_\epsilon,y_\epsilon)-(x'_\epsilon,y'_\epsilon)\right|\to 0.
	$$
	
	For any $(x,y)\in\Omega$, we have
	$$
	\Phi_\epsilon(x,y,x,y)=u(x,y)-v(x,y)\leq M_\epsilon,
	$$
	and therefore $M\leq M_\epsilon$.
	
	On the other hand, by the definition of $M_\epsilon$ and the inequality $M\leq M_\epsilon$, we obtain
	$$
	\frac{(x_\epsilon-x'_\epsilon)^2+(y_\epsilon-y'_\epsilon)^2}{2\epsilon}
	\leq
	u(x_\epsilon,y_\epsilon)-v(x'_\epsilon,y'_\epsilon)-M
	\leq C,
	$$
	where the last inequality follows from the fact that $u$ is USC and $v$ is LSC. Consequently,
	$$
	\left|(x_\epsilon,y_\epsilon)-(x'_\epsilon,y'_\epsilon)\right|
	\to 0
	\qquad\mbox{as }\epsilon\to0^+.
	$$
	
	Passing to a subsequence if necessary, we may assume that
	$$
	(x_\epsilon,y_\epsilon),(x'_\epsilon,y'_\epsilon)
	\to
	(\overline{x},\overline{y})\in\overline{\Omega}.
	$$
	Moreover,
	$$
	M
	\leq
	\limsup_{\epsilon\to0}M_\epsilon
	\leq
	u(\overline{x},\overline{y})
	-
	v(\overline{x},\overline{y})
	\leq M.
	$$
	Therefore,
	$$
	M_\epsilon\to M.
	$$
	
	\noindent\textit{Step 2:}
	Applying \cite[Theorem 3.2]{Crandall_Ishii_Lions_1992} to
	$$
	F(x,y,u,Du,D^2u)
	=
	f(x,y)
	-
	|u_x|^{p-2}u_{xx}
	-
	|u_y|^{q-2}u_{yy},
	$$
	and setting $p_\epsilon=\epsilon^{-1}(x_\epsilon-x'_\epsilon,y_\epsilon-y'_\epsilon)$,
	we obtain two matrices $X,Y\in\mathcal{S}^d$ with $X\leq Y$ such that
	$$
	(p_\epsilon,X)
	\in
	\overline{J}^{2,+}u(x_\epsilon,y_\epsilon)
	\qquad\mbox{and}\qquad
	(p_\epsilon,Y)
	\in
	\overline{J}^{2,-}v(x'_\epsilon,y'_\epsilon),
	$$
	where $\overline{J}^{2,+}$ and $\overline{J}^{2,-}$ denote the closed superjet and subjet, respectively.
	
	\noindent\textit{Step 3:}
	Since $u$ is a subsolution and $v$ is a supersolution, we have
	\begin{align*}
		0
		&\geq
		f(x_\epsilon,y_\epsilon)
		-
		\frac{|x_\epsilon-x'_\epsilon|^{p-2}}
		{\epsilon^{p-2}}X_{11}
		-
		\frac{|y_\epsilon-y'_\epsilon|^{q-2}}
		{\epsilon^{q-2}}X_{22},
		\\
		0
		&\leq
		f(x'_\epsilon,y'_\epsilon)
		-
		\frac{|x_\epsilon-x'_\epsilon|^{p-2}}
		{\epsilon^{p-2}}Y_{11}
		-
		\frac{|y_\epsilon-y'_\epsilon|^{q-2}}
		{\epsilon^{q-2}}Y_{22}.
	\end{align*}
	Subtracting the two inequalities gives
	\begin{equation}\label{eq-doubling-variables}
		f(x_\epsilon,y_\epsilon)-f(x'_\epsilon,y'_\epsilon)
		+
		\frac{|x_\epsilon-x'_\epsilon|^{p-2}}
		{\epsilon^{p-2}}
		(Y_{11}-X_{11})
		+
		\frac{|y_\epsilon-y'_\epsilon|^{q-2}}
		{\epsilon^{q-2}}
		(Y_{22}-X_{22})
		\leq0.
	\end{equation}
	
	\noindent\textit{Step 4:}
	Given $\delta>0$, define $u_\delta(x,y)=u(x,y)+\delta(x^2+y^2)$.
	Since $x^2+y^2$ is strictly convex, $u_\delta$ is a strict viscosity subsolution; more precisely,
	$$
	f-\plap^x u_\delta-\qlap^y u_\delta
	\leq-C\delta.
	$$
	
	Repeating the previous steps with $u_\delta$ in place of $u$, and using \eqref{eq-doubling-variables}, we obtain
	$$
	f(x_\epsilon,y_\epsilon)-f(x'_\epsilon,y'_\epsilon)
	+
	\frac{|x_\epsilon-x'_\epsilon|^{p-2}}
	{\epsilon^{p-2}}
	(Y_{11}-X_{11})
	+
	\frac{|y_\epsilon-y'_\epsilon|^{q-2}}
	{\epsilon^{q-2}}
	(Y_{22}-X_{22})
	\leq-C\delta.
	$$
	Taking the limit as $\epsilon\to0$ leads to a contradiction.
	
	Hence, $u_\delta\leq v$ in $\Omega$ and, since $\delta>0$ is arbitrary, we conclude that
	$$
	u\leq v
	\qquad\mbox{in }\overline{\Omega}.
	$$
	\end{proof}
	
	Now we also need to define the notion of viscosity solution to the equation
	\begin{equation}\label{equation.A2}
		\partial_t u (x,t) - \plap u (x,t) = 0.
	\end{equation}
	
	\begin{defn}[\textbf{Viscosity Solution Parabolic I}]\label{viscosity-equation-parabolic}
		A bounded upper semicontinuous function $u:\Omega \times (0,T)\to \RR$ is a viscosity subsolution of the parabolic equation \eqref{equation.A2} in $\Omega\times (0,T)$ if, whenever $(x_0,t_0)\in \Omega\times (0,T)$ and $\ph\in C^2(B_R(x_0)\times [t_0-R,t_0+R])$ for some $R>0$ are such that $\ph(x_0,t_0)=u(x_0,t_0)$ and $\ph(x,t)>u(x,t)$ for all $(x,t)\in B_R(x_0)\times [t_0-R,t_0]$, then
		$$\partial_t \ph(x_0,t_0)\leq \plap \ph(x_0,t_0).$$
		Similarly, a bounded lower semicontinuous function $u$ is a viscosity supersolution of \eqref{equation.A2} in $\Omega\times (0,T)$ if, under the same assumptions but with $\ph(x_0,t_0)=u(x_0,t_0)$ and $\ph(x,t)<u(x,t)$ in $B_R(x_0)\times [t_0-R,t_0]$, we have
		$$\partial_t \ph(x_0,t_0)\geq \plap \ph(x_0,t_0).$$
		Moreover, a bounded continuous function $u$ is a viscosity solution of \eqref{equation.A2} if it is both a viscosity subsolution and a viscosity supersolution.
	\end{defn}
	
	This notion of solution is equivalent to the following one in which our mean value formula appears. 
	
	\begin{defn}[\textbf{Viscosity Solution Parabolic II}]\label{viscosity-operator-parabolic}
		A bounded upper semicontinuous (resp.\ lower semicontinuous) function $u$ is a viscosity subsolution (resp.\ supersolution) associated with the operator $\B_\epsilon^p$ if the following holds: for every $(x_0,t_0)\in \Omega \times [0,T]$ and every test function $\ph\in C^2(B_R(x_0)\times [t_0-R,t_0+R])$ for some $R>0$, such that $\ph(x_0,t_0)=u(x_0,t_0)$ and $\ph(x,t)>u(x,t)$ (resp.\ $\ph(x,t)<u(x,t)$) in $B_R(x_0)\times [t_0-R,t_0+R]$, we have
		\[
		\ph(x_0,t_0+\epsilon^2)\geq \B_\epsilon^p[\ph](x_0,t_0)+o(\epsilon^2),
		\quad \text{as }\epsilon\to 0^+, \qquad
		(\mbox{resp. } \ph(x,t+\epsilon^2)\leq \B_\epsilon^p[\ph](x,t)+o(\epsilon^2)).
		\]
	\end{defn}
	
	Here by $\ph(x_0,t_0+\epsilon^2)\geq \B_\epsilon^p[\ph](x_0,t_0)+o(\epsilon^2)$ (resp. $\ph(x,t+\epsilon^2)\leq \B_\epsilon^p[\ph](x,t)+o(\epsilon^2) $) we mean that
	$$
	0\geq \limsup_{\epsilon \to 0} \frac{ \B_\epsilon^p[\ph](x_0,t_0) - \ph(x_0,t_0+\epsilon^2)}{\epsilon^2},
	\qquad
	\left(\mbox{resp. } 0\leq \liminf_{\epsilon \to 0} \frac{ \B_\epsilon^p[\ph](x_0,t_0) - \ph(x_0,t_0+\epsilon^2)}{\epsilon^2}\right).
	$$
	The proof of the equivalence between both definitions follows from our asymptotic expansions, see Section \ref{sec7.1}.  
	
	Finally, let us state the definition of being a viscosity solution to the boundary value problem
	\begin{equation}\label{problem-A-2}
		\left\{\begin{array}{rlll}
			\partial_t u(x,t) &=& \plap u(x,t), & \mbox{in } \Omega\times (0,\infty),\\
			u(x,0)&=&u_0(x), & \mbox{in } \Omega,\\
			u(x,t)&=&g(x), & \mbox{on } \partial \Omega\times (0,\infty).
		\end{array}\right.
	\end{equation}
	
	\begin{defn}[\textbf{Viscosity Solution Parabolic III}]\label{viscosity-dirichlet-parabolic}
		A bounded upper semicontinuous (resp.\ lower semicontinuous) function $u:\overline{\Omega}\times [0,\infty)\to \RR$ is a viscosity subsolution (resp.\ supersolution) to the Dirichlet problem \eqref{problem-A-1} if:
		\begin{enumerate}[label=\roman*)]
			\item $u$ is a viscosity subsolution (resp.\ supersolution) of \eqref{equation.A2} in the sense of Definition~\ref{viscosity-equation-parabolic}.
			\item $u(x,0)\leq u_0(x)$ (resp.\ $u(x,0)\geq u_0(x)$) in $\Omega$, and $u(x,t)\leq g(x)$ (resp.\ $u(x,t)\geq g(x)$) in $\partial \Omega \times (0,\infty)$.
		\end{enumerate}
		
		Moreover, a bounded continuous function $u$ is a viscosity solution of \eqref{problem-A-2} if it is both a viscosity subsolution and a viscosity supersolution.
	\end{defn}
	The comparison principle, and consequently the uniqueness of solutions, follows directly from the equivalence between viscosity and weak solutions; see, for instance,~\cite{Juutinen+Lindqvist+Manfredi_2001}.
	
	\section*{Acknowledgments}
	
	This work started and has been mostly carried out during the research
	stay of C. Fuertes to Univ. Torcuato di Tella (Argentina). The authors
	are grateful to this institution for the friendly and stimulating
	working atmosphere and the partial fundings.

	C. Fuertes acknowledges the financial support from the grants
	PID2020-113596GBI00, PID2021-127105NB-100, PID2023-150166NB-I00,
	CNS2024-154515, RED2022-134784-T, and PCI2024-155066-2, all funded by
	Ministry of Science and Innovation, Spain, MCIN/
	AEI/10.13039/501100011033 and acknowledges the financial support from
	the Spanish Ministry of Science and Innovation, through the “Severo
	Ochoa Programme for Centres of Excellence in R\&D” CEX2023-001347S.

	J. D. Rossi was partially supported by CONICET PIP GI No
	11220150100036CO (Argentina), PICT-03183 (Argentina) and UBACyT
	20020160100155BA (Argentina).
	
	\bigskip
	
	On behalf of all authors, the corresponding author states that there is no conflict of interest.
	
	This paper has no associated data.


\end{document}